\documentclass[11pt,reqno]{amsart}
\usepackage[margin=1in]{geometry}
\usepackage{genpack1}
\usepackage[caption = false]{subfig}
\numberwithin{equation}{section}

\nc{\tor}{{\R/\Z}}
\nc{\torr}{\mathbb{T}}
\nc{\XX}{X}
\nc{\YY}{Y}
\nc{\tZ}{\wt{Z}}
\nc{\tY}{\wt{Y}}
\nc{\tW}{\wt{W}}
\nc{\bH}{\bs{H}}
\nc{\vphi}{\varphi}
\nc{\tphi}{\wt{\varphi}}
\nc{\tpsi}{\wt{\psi}}
\nc{\uS}{\ul{S}}
\nc{\tH}{\wt{H}}
\nc{\xx}{x}
\nc{\yy}{y}
\nc{\zz}{z}
\nc{\ww}{w}
\nc{\bt}{\bs{t}}
\nc{\bxi}{\bs{\xi}}
\nc{\bx}{\bs{x}}
\nc{\feta}{\bs{\eta}}
\nc{\LA}{\lambda}
\nc{\epp}{\ep}
\nc{\we}{\wt{\e}}
\nc{\hH}{\widehat{H}}
\dmo{\Leb}{Leb}
\nc{\Err}{\Upsilon}
\dmo{\brw}{{BRW}}
\dmo{\cue}{{CUE}}
\dmo{\cbe}{{C\beta E}}
\dmo{\TV}{TV}
\dmo{\cmu}{\wch{\mu}}
\dmo{\cnu}{\wch{\nu}}
\dmo{\bet}{\bs{\eta}}
\dmo{\tnu}{\wt{\nu}}
\nc{\urn}{Q}
\nc{\turn}{\wt{Q}}
\nc{\LL}{V}
\nc{\ET}{\Theta_0}
\nc{\eon}{{i}}
\nc{\Len}{{L}}
\nc{\gen}{{k}}
\nc{\Eons}{{K}}
\nc{\super}{\cS}
\nc{\duper}{\wt{\cS}}
\dmo{\Maj}{Maj}
\nc{\xcrit}{x_0}
\nc{\betacrit}{\beta_0}

\begin{document}

\title[Maximum of the characteristic polynomial for a random permutation]{Maximum of the characteristic polynomial for a random permutation matrix}

\author[N.\ Cook]{Nicholas Cook$^\ddagger$}\thanks{${}^\ddagger$Partially supported by NSF postdoctoral fellowship DMS-1606310}
 \address{$^\ddagger$Department of Mathematics, University of California
\newline\indent Los Angeles, CA 90095-1555}
\author[O.\ Zeitouni]{Ofer Zeitouni$^{\mathsection}$}\thanks{${}^{\mathsection}$Partially 
supported by  ERC advanced grant LogCorFields}
\address{$^{\mathsection}$Department of Mathematics, Weizmann Institute of Science 
 \newline\indent POB 26, Rehovot 76100, Israel}

\date{\today}

\begin{abstract}
Let $P_N$ be a uniform random $N\times N$ permutation matrix and let $\chi_N(z)=\det(zI_N- P_N)$ denote its characteristic polynomial. 
We prove a law of large numbers for the maximum modulus of $\chi_N$ on the unit circle, specifically, 
\[
\sup_{|z|=1}|\chi_N(z)|= N^{\xcrit + o(1)}
\]
with probability tending to one as $N\to \infty$,  
for a numerical constant $\xcrit\approx 0.652$.
The main idea of the proof is to uncover a logarithmic correlation structure for the distribution of (the logarithm of) $\chi_N$, viewed as a random field on the circle, and to adapt a well-known second moment argument for the maximum of the branching random walk. 
Unlike the well-studied \emph{CUE field} in which $P_N$ is replaced with a Haar unitary, the distribution of $\chi_N(e^{2\pi \ii t})$ is sensitive to Diophantine properties of the point $t$. To deal with this we borrow tools from the Hardy--Littlewood circle method in analytic number theory.
\end{abstract}


\maketitle

\let\oldtocsubsection=\tocsubsection
\renewcommand{\tocsubsection}[2]{\hspace*{1.0cm}\oldtocsubsection{#1}{#2}}
\let\oldtocsubsubsection=\tocsubsubsection
\renewcommand{\tocsubsubsection}[2]{\hspace*{1.8cm}\oldtocsubsubsection{#1}{#2}}

\setcounter{tocdepth}{1}
\tableofcontents

\section{Introduction}
\label{sec:intro}

For a large integer $N$ let $P_N$ be an $N\times N$ permutation matrix drawn uniformly at random, and consider its characteristic polynomial 
\[
\chi_N(z) = \det(zI_N-P_N). 
\]
Our goal is to understand the asymptotic size of the maximum of $|\chi_N(z)|$ over the unit circle, up to sub-polynomial factors. 
It will be convenient for us to work with the logarithm of the following modification of $\chi_N$:
\[
\wt{\chi}_N(z) = \det(I_N-zP_N) = z^N\chi_N(1/z).
\]
Note that for $|z|=1$ we have $|\chi_N(z)| = |\wt{\chi}_N(\bar{z})|$. 
Hence we will consider the following random field on the torus $\tor$: 
\begin{equation}
\XX_N(t) = \log\abs{\wt{\chi}_N(e(t))} = \log\abs{\det(I_N-e(t)P_N)}
\end{equation}
taking values in $[-\infty,\infty)$.
Here and throughout we abbreviate $e(t):= \exp(2\pi\ii t)$.
See Figure \ref{fig:XN} for some numerical simulations.

\begin{figure}
\captionsetup{singlelinecheck=off}
\subfloat[$N=100$, $I=(0,1)$]{\includegraphics[width = 8cm]{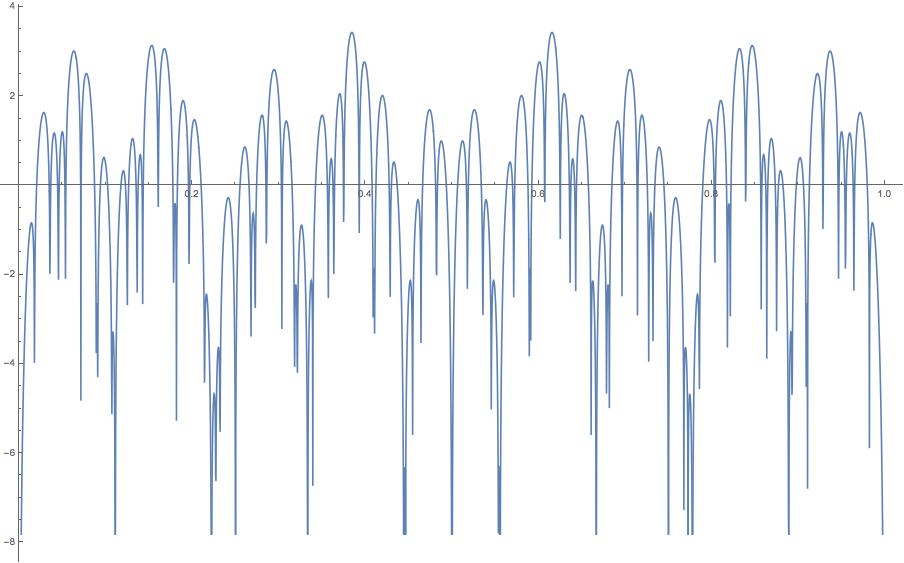}} 
\subfloat[$N=10^4$, $I=(0,1)$]{\includegraphics[width = 8cm]{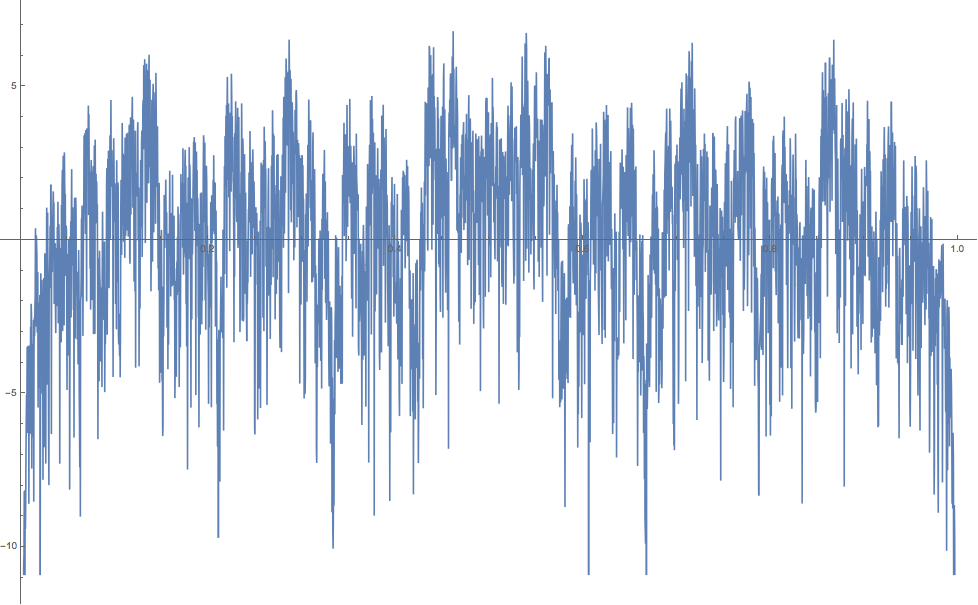}}\\
\subfloat[$N=10^4$, $I=(0.1,0.11)$]{\includegraphics[width = 8cm]{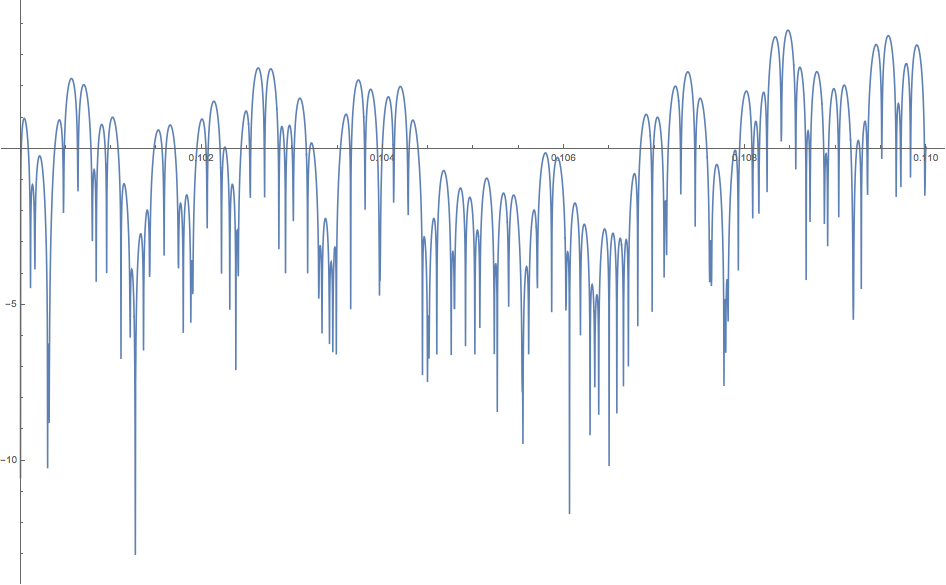}}
\subfloat[$N=10^9$, $I=(0,0.6)$]{\includegraphics[width = 8cm]{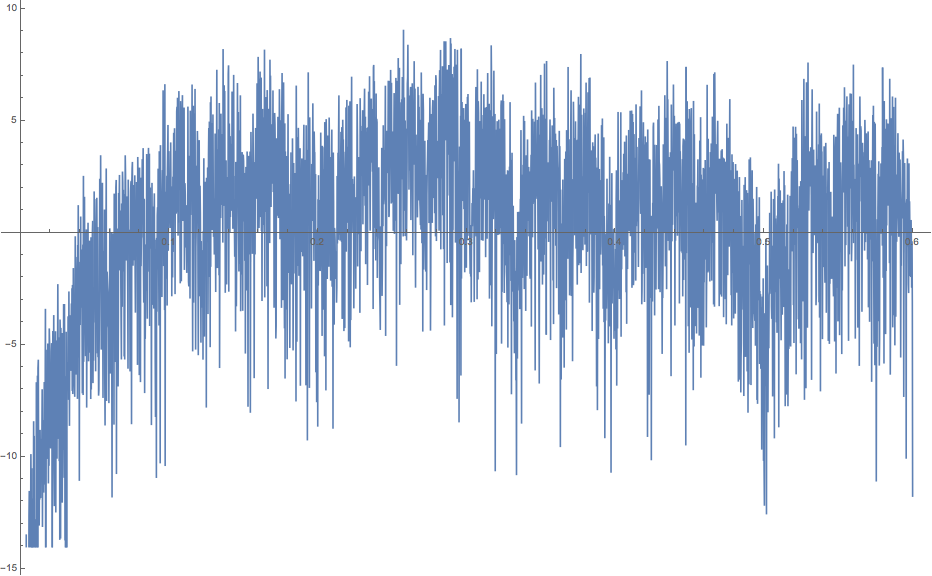}} 
\caption[cap]{Simulations of the field $X_N$ on subintervals $I\subset\tor$, computed from the cycle structures for the permutations $P_N$ using the formula \eqref{XN.cyc}). The cycle structures are random partitions of $[N]$ generated using the Chinese restaurant process. The respective partitions are: \begin{itemize}
\item[(A):]  $\{56, 22, 9, 9, 4\}$, 
\item[(B/C):]  $\{6310, 1914, 909, 668, 79, 47, 33, 19, 12, 5, 3, 1\}$, 
\item[(D):]  $\{892060223, 78087020, 19479718, 9152317, 630684, 352623, 114502, 104059, $
\item[] \text{ }\text{ }\text{ }\text{ }\text{ }\text{ }\text{ } $8973, 8193, 1641, 33, 5, 3, 2, 2, 1, 1\}$.
\end{itemize}
In (D) there are noticeable dips in the field near the rationals $0$, $1/2$, and $1/3$.}
\label{fig:XN}
\end{figure}

In \cite{HKOS}, Hambly, Keevash, O'Connell and Stark obtained a central limit theorem for the value of $X_N$ at fixed points $t\in \tor$ satisfying a qualitative condition on approximability by rationals:

\begin{theorem}[{\cite{HKOS}}]
\label{thm:HKOS}
Let $t\in \tor$ be of \emph{finite type}, that is,
\begin{equation}	\label{def:finite-type}
\liminf_{n\to \infty} n^{\gamma}\|nt\|_\tor >0
\end{equation}
for some constant $\gamma<\infty$, where $\|s\|_\tor$ is the distance from $s$ to the nearest integer. 
Then 
$
X_N(t)/\sqrt{\frac{\pi^2}{12}\log N}
$
converges in distribution to a standard normal variable.
Moreover, the same conclusion holds for $\Im\chi_N(e(t))$ in place of $X_N(t)$.
\end{theorem}

It is not hard to see that for any fixed rational $t=p/q\in \tor$ the distribution of $X_N(t)$ has an atom at $-\infty$, which explains why a hypothesis of the form \eqref{def:finite-type} should be needed (note in particular that $X_N(0)\equiv-\infty$). 

Earlier, Wieand had obtained a multidimensional CLT for the number of eigenvalues lying in a fixed collection of disjoint arcs \cite{Wieand_permutation}, which, as noted in \cite{HKOS}, is related to the distribution of the imaginary part of $\chi_N$ by the argument principle.
Ben Arous and Dang established non-Gaussian fluctuations for sufficiently smooth linear statistics of eigenvalues \cite{BeDa}. 
Theorem \ref{thm:HKOS} was extended by Dang and Zeindler \cite{DaZe} to a multidimensional CLT for fixed $d$-tuples $(t_1,\dots, t_d)\in (\tor)^d$ satisfying a multivariate version of the finite type condition \eqref{def:finite-type}.
We note that the multi-dimensional CLT is related to the notion of mod-Gaussian convergence for permutons, see \cite{FMN} for details. For information on mesoscopic 
and microscopic scales, see recent work of Bahier
\cite{bahier,Bahier18_micro}.
Many of these works also considered permutations sampled from the Ewens distribution (see also \cite{Zeindler13}); even more general distributions were considered in \cite{HNNZ}.
There has also been extensive work on ``modified permutation matrices" with entries weighted by i.i.d.\ complex variables $(z_i)_{i=1}^N$ \cite{Wieand_wreath, Evans_wreath, DaZe, bahier,NaNi13}; in the case that $|z_i|=1$ these matrices are representations of random elements of the wreath product $S^1\wr\fkS_N$ of the circle with the symmetric group on $N$ letters.

In the present article we are concerned with the global maximum of the field $X_N(t)$, for which we obtain a law of large numbers:

\begin{theorem}[Main result]	
\label{thm:main}
We have
\[
\frac{1}{\log N} \sup_{t\in \tor} \XX_N(t)\longrightarrow \xcrit\qquad \text{in probability},
\]
where $\xcrit$ is a numerical constant, defined implicitly in \eqref{def:xxc} below, with approximate value $\xcrit\approx 0.6524$.
\end{theorem}

A similar result for the imaginary part of the log-characteristic polynomial can be obtained by a much shorter argument -- see Section \ref{sec:overview.cycle}.

Theorem \ref{thm:main} parallels recent results on
the \emph{CUE field} $X_N^{\cue}(t)$, in which $P_N$ is replaced by an $N\times N$ Haar unitary matrix $U_N$. 
It was conjectured by Fyodorov, Hiary and Keating that 
\begin{equation}	\label{conj:FHK}
\max_{t\in \tor} X_N^{\cue}(t) = \log N - \frac34\log\log N + M_N,
\end{equation}
where $M_N$ is a sequence of random variables converging in law to an explicit distribution
\cite{FHK_freezing, FyKe_freezing}.
The leading order $\log N$ term in \eqref{conj:FHK} was established by Arguin, Belius and Bourgade in \cite{ABB}, thus obtaining the analogue of our Theorem \ref{thm:main} (with $\xcrit$ replaced by 1). 
The order $\log\log N$ correction was subsequently obtained by Paquette and the second author \cite{PaZe}.
Currently, the best result is due to Chhaibi, Madaule and Najnudel, who showed \eqref{conj:FHK} holds with $M_N$ having a tight sequence of distributions; moreover, they handled the more general C$\beta$E field, which specializes to the CUE at $\beta=2$ \cite{CMN}. Some important progress toward the identification of the limit law (for the case $\beta=2$) 
 is contained in \cite{Remy}.

The conjecture \eqref{conj:FHK} was motivated by a well-known analogy between the characteristic polynomial of a Haar unitary matrix and the Riemann zeta function $\zeta(s)$ in the neighborhood of a random point on the critical axis $\Re s = 1/2$. 
Specifically, letting $T$ be a large positive real, we draw $t_0\in [T,2T]$ uniformly at random and study the random function
$t\mapsto\zeta(1/2 + \ii (t_0+t))$.
In the analogy, $\log T$ plays the role of $N$, the dimension of the Haar unitary.
At the level of microscopic spacing of zeros (at scale $1/\log T$) the analogy is formalized in the Montgomery pair correlation conjecture \cite{Montgomery_pair}, as well as the stronger GUE hypothesis (see \cite{KaSa}).
A central limit theorem for $\log|\zeta(1/2 + \ii (t_0+t)|$ for fixed $t\in \R$ was obtained by Selberg  \cite{Selberg_clt}, and his result was extended to multidimensional CLTs in \cite{HNY, Bourgade_zeta-clt}.
The analogue of \eqref{conj:FHK} for the zeta function was recently established to leading order in \cite{ABBRS}, where it was shown
\[
\frac{1 }{\log\log T} \max_{|t-t_0|\le 1} \log|\zeta(1/2 + \ii (t_0+t))|\longrightarrow 1
\]
in probability. The same result was independently established in \cite{Najnudel_zetamax} (where the maximum of the imaginary part of $\log\zeta$ was also considered) conditional on the Riemann hypothesis.
The order $\log\log\log T$ correction was proved for a randomized model of the zeta function in \cite{ABH_randzeta}.

A key property of $X_N^{\cue}$ is that it is a \emph{logarithmically correlated field}, an archetypical example of which is the branching random walk. 
All of the aforementioned works made use of approaches developed for the extremes of branching random walk going back to Bramson \cite{Bramson78}.
See Section \ref{sec:overview} for further discussion of these ideas.

In the present work we also make use of analogies with branching random walk.
However, the discrete nature of the permutation matrix $P_N$ presents unique challenges for establishing Theorem \ref{thm:main}.
Here we highlight three key differences from the CUE field:
\begin{enumerate}[(a)]
\item\label{diff:dio} 
In contrast to the CUE field, the distribution of $X_N(t)$ is not invariant under rotations $t\mapsto t+s$. 
In particular, as we saw in Theorem \ref{thm:HKOS}, the distribution is sensitive to Diophantine properties of the point $t$.

\item\label{diff:clump}
Some difficulties arise from Poissonian aspects of the field $X_N(t)$. In particular,
the permutation is determined by $N$ discrete random
variables describing the number of cycles of different lengths. For fixed $t$,
all these variables participate
in determining the value of $X_N(t)$.
However,  certain 
unlikely (but not very unlikely) ``clumpings" of these variables can globally
affect $X_N$. 

\item\label{diff:tails}

For the CUE, for fixed $t \in\R/\Z$ a central limit theorem was proved for $X_N^{\cue}(t)/\sqrt{\frac12\log N}$ in \cite{KeSn} (see also \cite{BHNY}). 
Moreover, it turns out that the level $\sim \log N$ of the maximum is correctly predicted by modelling the field by a sequence of $N$ i.i.d.\ Gaussians of variance $\frac12\log N$. 
In contrast, a sequence of $N$ i.i.d.\ Gaussians of variance $\frac{\pi^2}{12}\log N$ (as suggested by Theorem \ref{thm:HKOS}) incorrectly predicts a maximum of $\sim \frac{\pi}{\sqrt{6}}\log N\approx 1.28\log N$ for $X_N(t)$.
The difference is due to the fact that the CUE enjoys strong comparisons with Gaussian tails, even in the large deviations regime, whereas the tails of $X_N(t)$ are non-Gaussian (even at points obeying a condition like \eqref{def:finite-type}). 

\end{enumerate}
We elaborate further on these points below.

To deal with \eqref{diff:dio} we borrow ideas from the Hardy--Littlewood circle method in analytic number theory, in particular the more Fourier-analytic version developed by Vinogradov in his work on the odd Goldbach conjecture.
In many applications of the method, one is faced with obtaining uniform control over the torus on an exponential sum $S_I(t) = \sum_{n\in I} f(n)e(nt)$, for some interval $I\subset \N$. If $I$ is of length $N$ and $f:I\to \C$ is bounded then the triangle inequality gives the trivial bound $\sup_{t\in \tor} |S_I(t)| \ll N$ (here we use the Vinogradov symbol; see Section \ref{sec:notation} for our conventions on asymptotic notation). 
To improve this estimate to $\sup_{t\in \tor} |S_I(t)| =o (N)$, one has to argue in different ways on two complementary subsets of the torus: the set of \emph{major arcs}, consisting of points lying close to a rational with small denominator, and the complementary set of \emph{minor arcs}. 
See \cite{Vaughan}, \cite[Chapter 13]{IwKo_book} and \cite{Tao_blog_254A8} for more background on the circle method.\footnote{We note that in most applications of the circle method, the major arcs are where the dominant contribution is; as we will see, here the major arcs are actually a nuisance region and do not play a role in the determination of
the maximum. See however Section \ref{sec:overview.cycle} where we argue that the opposite is true for the maximum of the imaginary part of $\log\chi_N(e(t))$.
We further note that major arcs were also a nuisance region for establishing concentration for the log-modulus of Kac polynomials in \cite{TaVu_poly}. }

To prove Theorem \ref{thm:main} we will also separately consider the behavior of the field $X_N(t)$ for $t$ lying in major and minor arcs. 
While the field is badly behaved on major arcs, it is not hard to see that it is likely to be very \emph{negative} there, so with high probability the points in major arcs are not contenders as maximizers for $X_N$.
For points $t\in \tor$ lying in minor arcs we can obtain accurate large deviation estimates for the field $X_N(t)$ that are uniform in $t$ by Fourier-analytic arguments.
We also obtain joint upper tail estimates for the field at two minor arc points $s,t$, which are crucial for our proof of the lower bound in Theorem \ref{thm:main} by the second moment method. 
These joint tail estimates reflect the logarithmic correlation structure, with respect to a certain notion of ``arithmetic distance" between $s$ and $t$ which quantifies the degree to which $1, s$ and $t$ are linearly independent over $\Z$.

To elaborate on \eqref{diff:clump}, the first step of our argument is to decompose $X_N$ into modes corresponding to the contribution of cycles of different lengths (see \eqref{XN.cyc}). We then replace the variables $C_\ell(P_N)$ counting cycles of length $\ell$ with independent Poi$(1/\ell)$ variables. (One can actually only replace the first $N^{1-\eps}$ of the variables $C_\ell(P_N)$ with Poisson variables; arguing that the cycles of length larger than $N^{1-\eps}$ have negligible impact is a significant technical hurdle for the proof of the lower bound.) In passing, we note that the tails 
of $|X_N(t)|$ for $t$ fixed are Poissonian, in contrast with
the Gaussian tails encountered in case of C$\beta$E.
The relatively heavy-tailed Poisson variables give rise to certain ``clumping" events which, though rare, are large enough to spoil the usual first and second moment arguments for the maximum of log-correlated fields.
This requires us (roughly speaking) to condition on a ``typical" realization of the counts of cycles with lengths in a fixed lacunary sequence of intervals. 


Finally, we elaborate on \eqref{diff:tails}.
First let us define the constant $\xcrit$ in Theorem \ref{thm:main}. 
Let $U\in \tor$ be uniformly distributed in the torus and consider the random variable 
\begin{equation}	\label{def:LL}
\LL= \log|1-e(U)|.
\end{equation}
From the fact that $\log|1-z|$ is harmonic on the open disk and the dominated convergence theorem, we have
\begin{equation}	\label{EV}
\e V = \int_\tor \log |1-e(u)|du = 0.
\end{equation}
Let us denote the logarithm of its Fourier--Laplace transform by
\begin{equation}	\label{def:LA}
\LA: \C_+\to \C, \quad \LA(z)= \log\e\expo{ z \LL} = \log\int_\tor |1-e(u)|^z du,
\end{equation}
where $\C_+$ denotes the open right half-plane, and we take the usual principal branch of the logarithm with branch cut along the negative real axis.
We have $\lambda(0)=0$, and (by routine considerations for cumulant generating functions) the restriction of $\LA$ to the 
positive real line is strictly increasing and convex.
We denote the Legendre transform 
\begin{equation}	\label{def:lambdastar}
\LA^*(x)= \sup_{\beta>0} \{x\beta - \LA(\beta)\} .
\end{equation}
The supremum is attained at 
\begin{equation}	\label{def:betastar}
\beta=\beta_*(x) := (\LA')^{-1}(x) = \frac{d}{dx} \LA^*(x).
\end{equation}
Since $\lambda:\R^+\to \R^+$ is convex and increasing, and 
$\lambda(\beta)\sim \beta \log 2- (\log \beta)/2$ as $\beta\to \infty$, 
we have that $\lambda^*$ is a bijection from $[0,\log 2)$ to the positive real line.
We define $\xcrit>0$ as the unique solution to 
\begin{equation}	\label{def:xxc}
\LA^*(\xcrit)=1.
\end{equation}
Numerically solving the above equation gives 
$\xcrit\approx 0.6524$,
with the supremum in $\lambda^*(\xcrit)$ attained at $\betacrit=\beta_*(\xcrit)\approx 11.746$. See Figure \ref{fig:ratefunction} for a comparison of the Gaussian rate function with $\lambda^*$, which is in some sense the ``effective" rate function for $X_N(t)$.

\begin{figure}
\includegraphics[width = 12cm]{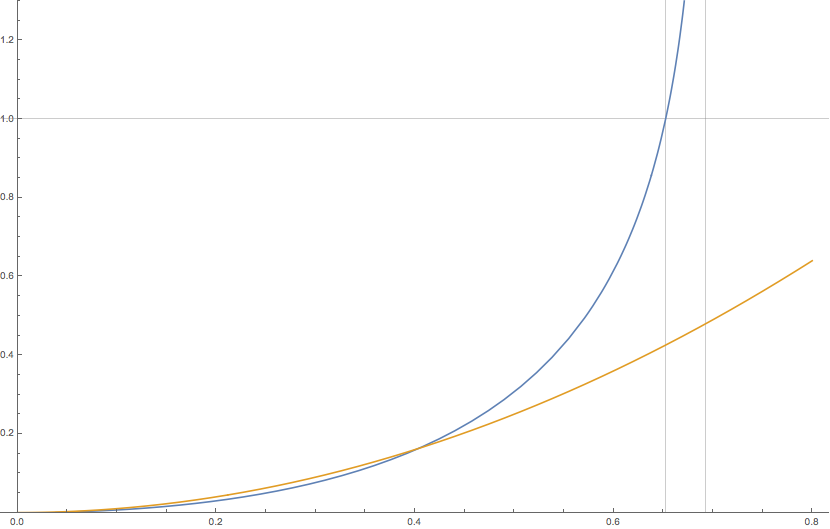}
\caption{
Numerically generated plot of the rate function $\lambda^*(x)$ from \eqref{def:lambdastar} (blue), with the Gaussian rate function $x^2$ (orange) for comparison (normalized as in the upper tail $N^{-x^2+o(1)}$ for the CUE field $X_N^{\cue}(t)$).
Included are the vertical lines $x=\xcrit=(\lambda^*)^{-1}(1)\approx 0.652$, and $x=\log 2$, where $\lambda^*$ diverges.
The lighter tail for $X_N(t)$ as compared to that of the CUE helps explain the smaller value for the maximum.
}

\label{fig:ratefunction}
\end{figure}

We briefly indicate the way in which the function $\lambda$ and the constant $\xcrit$ arise in the proof of Theorem \ref{thm:main}.
Roughly speaking, we will model the field $X_N(t)$ by a sequence $(\tY(t))_{t\in T_N}$ of weakly correlated variables indexed by a discrete set $T_N$ of roughly $N$ equally spaced points $t$ in $\tor$.
For each $t\in T_N$, $\tY(t)$ is a sum of roughly $\log N$ i.i.d.\ copies of the variable $\LL$ in \eqref{def:LL}.
(Such an approximation of $X_N(t)$ can only be justified for points $t\in \tor$ obeying a quantitative ``minor arc" condition similar in spirit to the finite type condition \eqref{def:finite-type}.)
Assuming the variables $\{\tY(t)\}_{t\in T_N}$ are sufficiently weakly correlated, one can show that $\max_{t\in T_N}\tY(t) =(\xcrit+o(1))\log N$ with high probability. 
Indeed, the expected number of points $t\in T_N$ for which $\tY(t)\ge x\log N$ is 
\[
|T_N|e^{-(\lambda^*(x)+o(1))\log N} = N^{1-\lambda^*(x)+o(1)}, 
\]
which is $o(1)$ for $x\ge \xcrit+\epp$.
This leads to an upper bound on the maximum of $\tY(t)$ by Markov's inequality. The lower bound is more involved and requires understanding the two-point correlations for the field $X_N$, which decay logarithmically with respect to a certain ``arithmetic" distance.

A more detailed overview of proof ideas is provided in Section \ref{sec:overview}.

\subsection{Organization of the paper}

The rest of the paper is organized as follows. After setting up notation in Section \ref{sec:notation}, we give a high-level description of the main ideas of the proof in Section \ref{sec:overview}.
In Section \ref{sec:prelim} we prove some preliminary estimates: Section \ref{sec:cyc} gives estimates on the distribution of cycle lengths for a random permuation, while Sections \ref{sec:phi} and \ref{sec:kronecker} provide lemmas that are used in Section \ref{sec:FL} to estimate Fourier--Laplace transforms of the Poisson field. In Section \ref{sec:FL} we also provide upper and lower bounds on large deviation events for the Poisson field (at one and two points). 
In Section \ref{sec:upper} we prove the upper bound in Theorem \ref{thm:main}, and we prove the matching lower bound in Sections \ref{sec:lower} and \ref{sec:late}.

\subsection{Notation}	\label{sec:notation}

We use the Vinogradov symbol $\ll$; thus, $f\ll g$, $g \gg f$, $f=O(g)$ all mean $|f|\le Cg$ for some universal constant $C<\infty$. $f\asymp g$ means $f\ll g \ll f$. We indicate dependence of the implied constant on parameters with subscripts, e.g.\ $f\ll_\alpha g$. 
$f=o_{\alpha;\, q\to q_0}(g)$ means that $f/g\to 0$ as $q\to q_0$ for any fixed value of the parameter $\alpha$ (so the rate of convergence may depend on $\alpha$). 
In the usual case that the asymptotic parameter is $N$ we suppress the subscript $N\to \infty$: thus, $f=o(g)$, $f=o_\alpha(g)$ mean $f=o_{N\to \infty}(g)$, $f=o_{\alpha;\, N\to \infty}(g)$. 
$f=\omega(g)$ means $g=o(f)$. 

For $x\in \tor$,
$\|x\|_\tor$ denotes the distance from $x$ to zero. 


We use $|S|$ to denote the cardinality of a set $S$, and $\Leb(E)$ to denote the Lebesgue measure of a Borel set $E\subset\T$. 

For an interval $J\subset\R$, by $\sum_{j\in J}$ we mean $\sum_{j\in J\cap \Z}$. 
Sums written $\sum_{j\le M}$ are understood to mean $\sum_{1\le j\le M}$. 

We denote the Fourier coefficients
of a function $f:(\tor)^d\to \C$ by
\begin{equation}	\label{def:fourierco}
\wh{f}(\bxi) = \int_\tor e(-\bxi \cdot \bt)f(\bt)dt_1\cdots dt_d,\quad \bxi\in \Z^d.
\end{equation}

\section{Proof overview}	
\label{sec:overview}

In this section we provide a high-level overview of proof ideas.

\subsection{Decomposition of $X_N$ into modes according to cycle structure}
\label{sec:overview.cycle}

A routine computation shows that if $Q$ is the $\ell\times \ell$ permutation matrix for a cycle of length $\ell$, then 
\begin{equation}
\det(1-zQ)= 1-z^\ell.
\end{equation}
Thus, letting $C_\ell(P_N)$ denote the number of cycles of length $\ell$ in the cycle decomposition of $P_N$, we have
\begin{equation}	\label{XN.cyc}
\XX_N(t) = \sum_{\ell=1}^N C_\ell(P_N)\log |1-e(\ell t)|.
\end{equation}

Before proceeding we make a couple of quick observations.
Since $\log|1-e(\cdot)|$ is uniformly bounded by $\log 2$ on $\tor$, we have
\begin{equation}
\sup_{t\in \tor}X_N(t)\le (\log 2) \sum_{\ell=1}^N C_\ell(P_N) \qquad \text{ almost surely.}
\end{equation}
Moreover, an easy first and second moment argument for the total number of cycles (similar to the proof of Lemma \ref{lem:Xtail.ub} below) yields 
\begin{equation}	\label{XN.crudeUB}
\sup_{t\in \tor}X_N(t) \le (\log 2 + \epp) \log N \qquad \text{ with probability $1-o(1)$}
\end{equation}
for any fixed $\epp>0$. Since $\log 2\approx 0.693$, this fails to achieve the sharp bound of Theorem \ref{thm:main}. 
However, we point out here that the same argument does achieve the sharp bound for the supremum of the imaginary part of $\log\chi_N(e(t))$. Indeed, similarly to \eqref{XN.cyc} we have
\begin{equation}
\Im\log\det(I_N-e(t)P_N) = \sum_{\ell=1}^N C_\ell(P_N) \arg(1-e(\ell t)),
\end{equation}
where we take the principal branch of $\arg$ so that $\arg(1-e(u))\in [-\pi/2,\pi/2)$ for $u\in \tor$. Now we can bound
\[
\sup_{t\in \tor}\Im\log\det(I_N-e(t)P_N)\le \frac\pi2\sum_{\ell=1}^N C_\ell(P_N) \le \Big(\frac\pi2+\epp\Big)\log N \qquad \text{with probability $1-o(1)$.}
\]
Conversely, one notes that for $t\in (1-\epp/N, 1)$ we have $\arg(1-e(\ell t))>\frac\pi2(1-\epp)$ for all $1\le \ell\le N$, so in fact
\[
\sup_{t\in \tor}\Im\log\det(I_N-e(t)P_N)= \Big(\frac\pi2+o(1)\Big)\log N \qquad \text{with probability $1-o(1)$.}
\]
The reason the trivial upper bound for the imaginary part is sharp is that the function $\arg(1-e(u))= \pi(u-\frac12)$ achieves its maximum in a neighborhood of zero, and for points $t$ very close to zero, the sequence $(\ell t)_{\ell\in [N]}$ stays within such a neighborhood. 
By contrast, the function $\log|1-e(u)|$ diverges to $-\infty$ near zero, while the maximum is achieved at $1/2$. 
This makes the upper bound \eqref{XN.crudeUB} impossible to achieve (any sequence $(\ell t)_{\ell\in [N]}$ visiting a neighborhood of $1/2$ very often will also visit a neighborhood of zero very often). 
Instead, as we will see below, the maximum tends to be attained at ``minor arc" points $t$, for which the sequence $(\ell t)_{\ell \in [N]}$ mixes rapidly in the torus, or equivalently, does not visit a small neighborhood of zero very often.

\subsection{Poisson approximation}

It is well known that for any fixed $k\in \N$, as $N\to \infty$ the joint distribution of the cycle counts $(C_1(P_N), \dots, C_k(P_N))$ converges in distribution to a sequence of independent Poisson variables $(Z_1,\dots, Z_k)$, where $\e Z_\ell = 1/\ell$. 
In fact, this still holds if $k$ grows with $N$ at any speed with $k=o(N)$, as shown by the following result of Arratia and Tavar\'e.

\begin{theorem}[Poisson approximation {\cite[Theorem 2]{ArTa92}}]		\label{thm:arta}
Let $(Z_\ell)_{\ell\in \N}$ denote a sequence of independent Poisson random variables with $\e Z_\ell=1/\ell$ for each $\ell\in \N$.
Let $1\le M\le N$.
Then then following bound on the total variation distance between the sequences $ (C_\ell(P_N))_{\ell\le M}$ and $(Z_\ell)_{\ell\le M}$ holds:
\begin{equation*}
\sup_{A\in \Z_{\ge0}^{M}} \Big| \pr\Big( (C_\ell(P_N))_{\ell\le M}\in A\Big) - \pr\Big((Z_\ell)_{\ell\le M}\in A\Big)\Big|  \le \expo{ -(1+o(1))\frac{N}{M}\log\frac{N}{M}}.
\end{equation*}
\end{theorem}

See \cite{ArTa92} for a more specific bound; for us it is only important that the total variation distance goes to zero if $M=o(N)$. 

This motivates splitting the permutation field $X_N$ as follows. 
We first truncate the sum in \eqref{XN.cyc}.
Let $W$ be a sowly growing parameter satisfying
\begin{equation}	\label{2W}
\omega(1)\le W\le N^{o(1)}
\end{equation}
and write
\begin{align}
X_N(t) &= \sum_{\ell\le N/W} C_\ell(P_N)\log|1-e(\ell t)| + \sum_{N/W<\ell\le N} C_\ell(P_N) \log |1-e(\ell t)|	=: X_{N}^\le(t) + X_{N}^>(t). 	\label{XNW}
\end{align}
Let $(Z_\ell)_{\ell\ge1}$ be a sequence of independent Poisson variables with $\e Z_\ell=1/\ell$, and define the \emph{Poisson field}
\begin{equation}	\label{YN}
Y_N(t) = \sum_{\ell\le N} Z_\ell \log|1-e(\ell t)|.
\end{equation}
From Theorem \ref{thm:arta}, we can understand the maximum of the field $X_{N}^\le$ by instead considering the Poisson field $Y_{N/W}$.
Then we have two problems: 
\begin{enumerate}[(I)]
\item \label{probi} Estimate the maximum of $Y_{N/W}(t)$.
\item \label{probii} Show that the high frequency tail $X_{N}^>$ has negligible impact.
\end{enumerate}

\subsection{Upper bound: Failure of the first moment method}

We begin the discussion of how to prove the upper bound in Theorem \ref{thm:main}, i.e.\ that for any fixed $\epp\in (0,1)$, 
\begin{equation}	\label{2goal.upper}
\sup_{t\in \tor} X_N(t) \le (\xcrit+\epp) \log N \qquad \text{ with probability $1-o_\epp(1)$}.
\end{equation}
For the upper bound, problem (\ref{probii}) is relatively simple:
a straightforward second moment computation shows that with probability $1-o(1)$,
\begin{equation}
\max_{t\in \tor}X_{N}^{>}(t) \le O(\log W) = o(\log N)\,;
\end{equation}
see Lemma \ref{lem:Xtail.ub}. The key point is that the functions $\log|1-e(\ell\,\cdot)|$ are bounded pointwise by $\log 2<1$ on $\tor$, so $X_{N}^>(t)$ is uniformly controlled by the number of cycles of length between $N/W$ and $N$.
(As we discuss below, problem (\ref{probii}) is significantly more challenging for proving the lower bound due to the singularities of $\log|1-e(\ell\,\cdot)|$.)
Thus, to prove \eqref{2goal.upper} our task is reduced to showing the same estimate for the Poisson field $Y_{N/W}$, which from \eqref{2W} is equivalent to showing
\begin{equation}	\label{2goal.upper2}
\sup_{t\in \tor} Y_N(t) \le (\xcrit+\epp) \log N \qquad \text{ with probability $1-o_\epp(1)$}.
\end{equation}

It turns out it is enough to control the maximum on a fine mesh for the torus. (Actually, in the proof it will be more convenient to pass to a mesh \emph{before} replacing the permutation field $X_N$ with the Poisson field $Y_N$, but we ignore this point here.)
Specifically, it will suffice to show
\begin{equation}	\label{2goal.upper3}
\max_{t\in T_N}Y_N(t) \le (\xcrit+\eps)\log N\qquad \text{with probability } 1-o(1)
\end{equation}
where $T_N\subset \tor$ is a finite mesh for $\tor$ of cardinality $O(N)$. 

Following the standard first moment paradigm, see e.g. 
\cite{ABB} for a similar approach,
we consider the super-level sets
\begin{equation}	\label{def:superY}
\super^Y_N(T,\xx) :=  \{ t\in T: Y_N(t) \ge \xx\log N\}
\end{equation}
with $T=T_N$. From Markov's inequality, for the upper bound in Theorem \ref{thm:main} it will suffice to establish the first moment estimate
\begin{equation}	\label{naive1mom}
\e |\super^Y_N(T, \xcrit+\epp) | = \sum_{t\in T} \pr( Y_N(t) \ge (\xcrit+\epp)\log N) = o_\epp(1).
\end{equation}
Our task thus reduces to a large deviations problem: to obtain accurate estimates for the tail events
$
\{Y_N(t)\ge \xx \log N\}
$
at different points $t\in \tor$. 
(Recall that, in contrast to the CUE field discussed around \eqref{conj:FHK}, the distribution of $Y_N(t)$ depends on $t$.)
As $Y_N(t)$ is a sum of independent variables, we can obtain accurate upper bounds on tail events by Laplace's method. We can easily evaluate the Laplace transform:
\begin{equation}
\e e^{\beta Y_N(t)} = \prod_{\ell\le N} \e \expo{ \beta Z_\ell \log|1-e(\ell t)|} = \exp\bigg( \sum_{\ell\le N} \frac1\ell \left( |1-e(\ell t)|^\beta - 1\right)\bigg).
\end{equation}
From considerations of Diophantine approximation that will be discussed in Section \ref{sec:overview.tails}, one can show that for \emph{generic} points $t\in \tor$ (points that are in a certain quantitative sense ``sufficiently irrational"), 
\begin{equation}	\label{mgf:1st}
\log \e e^{\beta Y_N(t)} = \sum_{\ell\le N} \frac1\ell \left( |1-e(\ell t)|^\beta - 1\right) = \bigg(\int_\tor( |1-e(u)|^\beta -1)du +o(1) \bigg)\log N,
\end{equation}
where, crucially, the error $o(1)$ is uniform over all generic points $t$. 
Applying Markov's inequality to an exponential moment $\e \exp(\beta Y_N(t))$ and optimizing $\beta$ one obtains a sharp upper tail bound for $Y_N(t)$ at such points. Assuming (optimistically) that our mesh $T_N$ contains only generic points, then upon inserting the resulting tail bounds into \eqref{naive1mom}, a computation shows
\[
\e |\super^Y_N(T,\xx)| \asymp \begin{cases} N^{-c(\eps)} & \text{for } \xx= 1+ \eps\\ N^{c'(\eps)} & \text{for } \xx=1-\eps\end{cases}
\]
for any sufficiently small $\eps>0$, where $c(\eps),c'(\eps)$ are positive constants depending only on $\eps$.
In particular, this first moment computation suggests (incorrectly) that $\xcrit=1$. From Markov's inequality we obtain the upper bound
\begin{equation}	\label{sketch:Ywrong}
\max_{t\in T_N}Y_N(t) \le (1+\eps)\log N\qquad \text{with probability } 1-o_\epp(1).
\end{equation}
Recalling that $\xcrit\approx 0.652$, this bound is consistent with \eqref{2goal.upper3} but not sharp 
-- in fact, it even fails to beat the trivial bound \eqref{XN.crudeUB}.

\subsection{Conditioning on a typical distribution of cycle lengths at coarse scale }

To obtain the sharp bound in Theorem \ref{thm:main} we need to modify our first and second moment computations to avoid a certain class of rare events.
Indeed, there is a class of small events on which the maximum of $Y_N(t)$ is atypically large (of size $\sim \xx\log N$ for some $\xx\in (\xcrit,1)$), but which are of 
probability $N^{-c}$ for some $c\in (0,1)$. When we are looking at height $\xx$ these events cause the first moment to blow up. However, the first moment (i.e.\ the union bound) is wasteful as these events have significant overlap for different values of $t\in \tor$. 

Roughly speaking, the problematic events occur when there is a ``clump" in the sequence of Poisson variables $Z_\ell$ (or cycle counts $C_\ell(P_N)$) -- specifically, when $\cN(I):=\sum_{\ell\in I} Z_\ell$ is atypically large for a short interval $I$. 
Recall that $\e Z_\ell=1/\ell$, so ``large" and ``short" in the preceding sentence depend on the position of $\ell\in [N]$ (for instance when $\ell$ is of order $N$ the event that $\cN(\{\ell\})=Z_{\ell}\ge2$ is atypical).
Since $\e \sum_{\ell\le N} Z_\ell =\log N + O(1)$, we expect gaps between consecutive nonzero variables $Z_\ell$ to be of order $\ell$. 
Thus, for a given expected size for $\cN(I)$ we should consider intervals of exponentially growing length. 
If it occurs that a large number of cycles have lengths in the interval $[\ell, (1+\delta)\ell]$ for some $\ell\sim N^c$ and small $\delta>0$, then the corresponding modes $t\mapsto \log|1-(\ell t)|$ in the expansion \eqref{XN.cyc} (or \eqref{YN}) will add constructively on large subsets of $\tor$ (roughly, on intervals of length $\sim N^{-c}$ spaced at distance $\sim1/\delta$ from each other), creating an atypically large number of high points (on the order of $\delta N^{1-c}$).

We deal with these rare clumping events by grouping the summands in \eqref{YN} into intervals with a lacunary sequence of endpoints, and conditioning on the event that (for most intervals) there is at most one nonzero summand. 
To state this more precisely we develop some additional notation. 
For $J\subset\R_+$ we let
\begin{equation}	\label{def:NJ}
\cN(J) = \sum_{\ell\in J\cap \Z} Z_\ell
\end{equation}
denote the number of ``cycles" with lengths in the set $J$ (recall the $Z_\ell$ are actually Poisson variables that model the cycle counts), 
and for $t\in \tor$ we write	
\begin{equation}	\label{def:YJ}
\YY_{J}(t) = \sum_{\ell\in J\cap \Z} Z_\ell\log|1-e(\ell t)|
\end{equation}
for the contribution of these cycles to the field $Y_N(t)$ (assuming $J\subset(0,N]$). 
We let $\varrho\in (0,1/2)$ be a small parameter (which may depend on $N$ and $\epp$) and for $\gen\ge 0$ denote
\begin{equation}	\label{def:micro}
I_\gen = [e^{\varrho \gen}, e^{\varrho (\gen+1)}),\qquad \rho_\gen= \e \cN(I_\gen) = \sum_{\ell\in I_\gen} \frac1\ell.
\end{equation}
Taking $\varrho, n$ depending on $N,\epp$ is such a way that
\begin{equation}
\varrho n = (1+o(1))\log N
\end{equation}
it now suffices to show
\begin{equation}	\label{2goal.upper4}
\max_{t\in T_N} Y_{[1,e^{\varrho n})}(t) = \max_{t\in T_N}\sum_{1\le \gen<n} \YY_{I_\gen}(t) \le (1+\epp) \varrho n\qquad \text{ with probability $1-o_\epp(1)$}.
\end{equation}

The variables $\cN(I_\gen)$ are independent Poisson variables with expectation $\rho_\gen$, and for $\gen$ reasonably large we have $\rho_\gen\approx \varrho$ (specifically, if $\gen=\omega(\frac1\varrho\log\frac1\varrho)$ then $\rho_\gen=(1+o(1))\varrho$).
To avoid the atypical clumping events described above 
we will condition on a ``typical" sequence of the partial sums $\cN(I_\gen)$. 
For $a\ge0$ let
\begin{equation}	\label{def:Qa}
\urn_a=\{\gen\ge 0: \cN(I_\gen) =a\}, \qquad \urn_{\ge a}= \{\gen\ge 0: \cN(I_\gen)\ge a\}
\end{equation}
denote the sets $\gen$ indexing intervals $I_\gen$ containing exactly $a$ cycles and at least $a$ cycles, respectively.
When $\varrho$ is sufficiently small, for a long interval $[m,n)$ we expect (roughly speaking)
\begin{equation}	\label{urn.approx}
|\urn_0\cap [m,n)|\approx (1-\varrho)(n-m), \qquad |\urn_1\cap [m,n)|\approx \varrho (n-m), \qquad |\urn_{\ge 2}\cap [m,n)|\approx 0.
\end{equation}
We will condition on fixed realizations $\urn,\turn$ for $\urn_1\cap [m,n)$ and $\urn_{\ge 2}\cap [m,n)$, respectively, satisfying the above.
For given $m,n,\urn,\turn$ we denote the event
\begin{equation}	\label{def:Qevent}
\cQ:=\cQ(m,n,\urn, \turn) := \left\{ \urn_1\cap [m,n)=\urn , \quad \urn_{\ge 2}\cap [m,n) = \turn\right\}.
\end{equation}
With a fixed choice of $m,n,\urn,\turn$ we will often abbreviate
\begin{equation}	\label{def:eHK}
\e^\cQ(\,\cdot\,) := \e\left( \, \cdot\ \middle| \ \cQ\right), \qquad 
\pr^\cQ(\,\cdot\,) := \e^\cQ\ind(\,\cdot\,).
\end{equation}
Now in place of \eqref{naive1mom} we want to show that for a typical choice of $\urn,\turn$, 
\begin{equation}	\label{new1mom}
\e^{\cQ}|\super^Y_{\lf e^{\varrho n}\rf}(T,\xcrit+\epp)| = \sum_{t\in T} \pr^{\cQ}\Big(Y_{[1,e^{\varrho n})} (t) \ge (\xcrit+\epp)\varrho n \Big) = o_\epp(1).
\end{equation}
\eqref{2goal.upper4} then follows from Markov's inequality and the fact that we conditioned on typical realizations of $\urn_1,\urn_{\ge2}$.

Note that
\eqref{urn.approx} suggests that for typical choices of $\urn,\turn$, on the event $\cQ(1,n,\urn,\turn)$, 
\[
Y_{[1,e^{\varrho n})}(t) \approx \sum_{\gen\in \urn} Y_{I_\gen}(t),
\]
where for $k\in Q$, 
\[
Y_{I_k}(t) = \log|1-e(\bs{\ell}_k t)|
\]
and $\bs{\ell}_k$ is a random element of $I_k$ with distribution $\pr(\bs{\ell}_k=\ell)\propto \frac1\ell1_{\ell\in I_k}$.
As will be discussed below, in the generic case that $t\in \tor$ is not too close to a rational with small denominator, the sequence $\{\ell t\}_{\ell \in I_k} \subset \tor$ rapidly equidistributes in the torus. One can use this together with \eqref{EV} to show that for such points $t$ and for $\gen\in \urn$,
\[
\e^{\cQ} Y_{I_\gen}(t) = \e\left( Y_{I_\gen}(t) \,\big| \, \cN(I_\gen)=1\right)\approx 0.
\] 
One can similarly show that at such points $t$ the variables $\{Y_{I_\gen}(t)\}_{\gen\in \urn}$ have comparable variance.
Thus we expect the variables $Y_{I_\gen}(t)$ to have comparable contribution to the fluctuations of $Y_{[1,e^{\varrho n})}(t)$ (at least for generic points $t\in \tor$).

\subsection{Upper bound: major and minor arcs}
\label{sec:overview.arcs}

To show \eqref{new1mom}
we need accurate estimates on the summands; the value for $\xcrit$ will then be the critical value of $\xx$ at which the tail probability
\begin{equation}	\label{mod.tails}
\pr^\cQ\Big(Y_{[1,e^{\varrho n})} (t) \ge \xx\varrho n \Big) 
\end{equation}
is of order $e^{-\varrho n +o(1)} = N^{-1+o(1)}$, balancing the entropy cost $|T|\asymp N$. 

A major difficulty is that the distribution of $Y_N (t)$, and in particular the tail probabilities \eqref{mod.tails}, depend on the choice of $t\in \tor$.
Indeed, a moment's thought reveals there is no hope of showing a
probability estimate of size $e^{-\varrho n+o(1)}$ for all $t$: consider for instance taking $t=0$, where $Y_N$ (and $X_N$) are $-\infty$ for all $N$. 
Moreover, for any rational $p/q\in \tor$ with $q=O(1)$ we have that $\pr(Z_q\ne 0)$ is of constant order, and from \eqref{YN} we see that $Y_N(p/q)=-\infty$ on this event. 

One thus sees that a general obstruction to obtaining nice tail estimates for $Y_N(t)$ is that $t$ is well-approximated by a rational with small denominator. 
We quantify this in terms of \emph{Bohr sets}. Recall that for an integer frequency $\xi\in \Z$ and $\kappa\in (0,1)$, the associated Bohr set is defined
\begin{equation}	\label{def:bohr}
B_\xi(\kappa) = \{t\in \tor: \|\xi t\|_\tor \le \kappa\}
\end{equation}
(recall our notation $\|s\|_\tor$ for the distance from $s$ to 0 in the torus). 
For a cuttoff frequency $\xi_0\in \N$ we denote the set of ``major arc points"
\begin{equation}	\label{def:Maj}
\Maj(\xi_0,\kappa) = \bigcup_{1\le \xi\le \xi_0} B_\xi(\kappa).
\end{equation}
Borrowing terms from the Hardy--Littlewood circle method, we refer to elements of the complement $\tor\setminus \Maj(\xi_0,\kappa)$ as ``minor arc points". 
(Viewed as a subset of the circle $\partial \mathbb{D}=e(\tor)$, $\Maj(\xi_0,\kappa)$ is a union of arcs with lengths in $\{2\kappa/\xi: 1\le\xi\le \xi_0\}$.) 
We will take $\xi_0,\kappa$ small enough that $\Maj(\xi_0,\kappa)$ has Lebesgue measure $o(1)$, so that \emph{generic} points in the torus are minor arc.

For major arc points $t$ the sequence $\{\ell t\}_{\ell\ge 1}\subset\tor$ is poorly mixing, and for small values of $\ell$ it returns often to a small neighorhood of $0$. 
Since $\log|1-e(\,\cdot\,)|$ is very negative there, from \eqref{YN} one can (correctly) guess it is possible to choose $\xi_0,\kappa$ in such a way that, with high probability, the values of $\ell$ corresponding to these visits cause a large deficit in the value of $Y_N(t)$ that cannot be compensated by all of the other terms in the sum. 
Thus, with high probability, the Poisson field does not even take positive values at major arc points, and it only remains to establish \eqref{2goal.upper4} with the supremum restricted to minor arcs.

While major arcs were fairly easy to handle in the proof of the upper bound, they will enter the proof of the lower bound in a more subtle way, as will be discussed in Section \ref{sec:overview.lower}.

\subsection{Large deviation estimates for the field on minor arcs}
\label{sec:overview.tails}

For $t\in \tor$ that are minor arc we can obtain fairly accurate estimates on the tails of partial sums $Y_{[e^{\varrho m}, e^{\varrho n})}(t)$, with uniform error bounds depending on $\xi_0,\kappa$. 
We briefly indicate why Bohr sets should play such a role. 
For the sake of discussion we ignore the contribution of cycles with lengths in $\turn$.
From large deviations theory we can deduce accurate tail estimates from estimates on the moment generating function
\begin{equation}	\label{sketch:mgf}
\e^\cQ \exp\bigg( \beta \sum_{\gen\in \urn} Y_{I_\gen}(t) \bigg) = \prod_{\gen\in \urn} \expo{ \beta Y_{I_\gen}(t) \,\middle|\, \cN(I_\gen) =1} = \prod_{\gen\in Q}\bigg( \frac1{\rho_\gen} \sum_{\ell\in I_\gen} \frac{|1-e(\ell t)|^\beta}{\ell}\bigg)
\end{equation}
(compare with \eqref{mgf:1st}).
Recalling that $\rho_\gen=\sum_{\ell\in I_\gen} 1/\ell$, we recognize the above as a logarithmic average of the function $\phi_\beta:= |1-e(\,\cdot\,)|^\beta$ from $\tor$ to $\R$ along the \emph{Kronecker sequence} $\{\ell t\}_{m\le \ell<n}\subset \tor$. 
Now it is well known that that if $t\in \tor$ is irrational and $f:\tor\to \C$ is continuous, we have
\[
\frac1n\sum_{\ell\le n} f(\ell t) \longrightarrow \int_\tor f(u)du \qquad \text{ as $n\to \infty$};
\]
from summation by parts we get the same limit for the logarithmic averages $\frac1{\log n}\sum_{\ell\le n} \frac{f(\ell t)}{\ell}$.
However, we require quantitative bounds depending on parameters $\xi_0,\kappa$, which we obtain in Section \ref{sec:prelim} by standard Fourier-analytic arguments.
These estimates show 
\[
\left|\frac1{\rho_\gen} \sum_{\ell\in I_\gen} \frac{|1-e(\ell t)|^\beta}{\ell} - \int_\tor |1-e(u)|^\beta du\right| \le \text{Error}(\gen,\varrho, \xi_0,\kappa)
\]
where the right hand side is $o_{\gen\to\infty}(1)$ for $\varrho$ sufficiently small and appropriate choices of $\xi_0,\kappa$.
From the above and \eqref{sketch:mgf}, and recalling \eqref{def:LA}, we can show 
\begin{equation}	\label{sketch:mgf2}
\e^\cQ\exp\bigg( \beta \sum_{\gen\in \urn} Y_{I_\gen}(t) \bigg) = (1+o_{n-m\to\infty}(1))e^{\lambda(\beta)|\urn| }
\end{equation}
for appropriate $\xi_0,\kappa$ depending on $n,m$ (see Proposition \ref{prop:Phi} for a precise statement).

The above estimate suggests comparing $Y_{[e^{\varrho m}, e^{\varrho n})}(t)$ with a sum of i.i.d.\ copies of the variable $V$ from \eqref{def:LL}. 
Let $(U_\gen)_{\gen\in \urn}$ be a sequence of i.i.d.\ uniform elements of $\tor$ and put 
\begin{equation}
V_\gen = \log|1-e(U_\gen)|, \qquad \tY_\urn = \sum_{\gen\in \urn} V_\gen.
\end{equation}
We recognize the main term on the right hand side of \eqref{sketch:mgf2} as the moment generating function for $\tY_\urn$. 
We can use this estimate to make a comparison of the form
\begin{equation}	\label{sketch:tail1}
\pr^\cQ\big( Y_{[e^{\varrho m}, e^{\varrho n})}(t) \ge \yy |\urn| \big) = (1+o(1))\pr\big( \tY_\urn \ge \yy |\urn| \big)
\end{equation}
(see Proposition \ref{prop:Yt.approx} for the quantitative estimate).
This comparison yields a tail bound that is sufficient to complete the proof of \eqref{2goal.upper4} for minor arc points, and hence the upper bound in Theorem \ref{thm:main} (in fact we can even get by with a weaker tail estimate given by Corollary \ref{cor:Yt.upper}). 
Note the comparison \eqref{sketch:tail1} and the form of the moment generating function for $\tY_\urn$ explain where the definition of the constant $\xcrit$ in \eqref{def:xxc} comes from.

For the proof of the lower bound by a second moment argument we will also need joint tail estimates for partial sums of the Poisson field at two points $s,t\in \tor$. We defer discussion of this to the following subsection.


\subsection{Lower bound: Logarithmic decay of correlations and structural dichotomy for high points}
\label{sec:overview.lower}

On a conceptual level we find it convenient to adopt some terminology that is suggestive of analogies with branching processes.
For $M\le N$, let us denote the partially summed field
\begin{equation}	\label{def:X.partial}
X_{\le M}(t) = \sum_{\ell \le M} C_\ell(P_N) \log|1-e(\ell t)|,
\end{equation}	\label{def:super.partial}
and for a level $x\in \R$, the associated set of ``survivors"
\begin{equation}
\cS_{\le M}(x) = \{ t\in T_N: X_{\le M}(t)\ge x\log M\},
\end{equation}
where, as before, $T_N$ is a fixed mesh for $\tor$ of size $O(N)$. (We will need to take $T_N$ to be a slight rotation of the set of $N'$-th roots of unity for some $N'=O(N)$, but we do not discuss this issue here.)
We think of the cutoff $M$ as a time parameter, and $\cS_{\le M}(x)$ as the portion of an initial ``population" $T_N$ that has attained a ``fitness level" $x$. 

Our goal is to show that the final set of survivors
$
\cS_{\le N}(\xcrit-\epp)
$
is non-empty with high probability. 
This task is considerably more delicate than for the upper bound.
We track the population across three epochs:
\begin{enumerate}
\item (Early generations): We show 
\begin{equation}	\label{bd:early}
\cS_{\le N^c}(-C)\gg N \qquad \text{ with high probability}
\end{equation}
for some constants $C,c>0$ sufficiently large and small, respectively. 
\item (Middle generations): Letting $W$ be a slowly growing parameter as in \eqref{2W}, we show 
\begin{equation}	\label{bd:middle}
|\cS_{\le N/W}(\xcrit-\epp)|\ge N^{c(\epp)}\qquad\text{ with high probability}.
\end{equation}
\item (Late generations): Finally, we argue 
\begin{equation}	\label{bd:late}
|\cS_{\le N}(\xcrit-2\epp)|\ge 1\qquad\text{ with high probability.}
\end{equation}
\end{enumerate}
For early and middle generations we can use Theorem \ref{thm:arta} to replace the cycle counts $C_\ell(P_N)$ with Poisson variables $Z_\ell$, as we did for the upper bound.
The reason to consider early generations separately is to gain some independence between different ``lineages" for the second moment argument for middle generations. 
The basic idea for establishing \eqref{bd:early} is that the partially summed Poisson field $Y_{\le M}(t)=\sum_{\le M} Z_\ell \log |1-e(\ell t)|$ is an average over samples from the mean zero function $\log|1-e(\cdot)|$; since we already have control on such random averages from above, we can argue that the truncated field can't be too much smaller than its average value at too many points $t\in \tor$.

For the middle generations we adapt a well known second moment argument for the maximum of branching random walk, the basic ideas of which go back to \cite{Bramson78}.
(The analogue of Theorem \ref{thm:main} for branching random walk goes back a bit further -- see \cite[Theorem 2]{Hammersley74}, as well as \cite{Kingman75, Kingman76, Biggins77}.)
The second moment argument is also used in the works \cite{ABB, PaZe, CMN} on the maximum for the CUE field.
We will not review the detailed argument here, and instead refer  the reader to 
\cite{Arguin_notes, Kistler_notes,Zeitouni_brw-notes} for an exposition.
The main point is that a straightforward application of the second moment method to the number of survivors $|\cS_{\le N/W}(\xcrit-\epp)|$ fails to show it is nonzero with high probability, despite the fact that its first moment is of size $N^{c(\epp)}$. 
Instead, we compute the first and second moments of a (smaller) number of points $t\in T_N$ for which the partially summed Poisson field $Y_{\le M}(t)$ makes steady progress towards the extreme level $\xcrit\log N$. We measure the notion of \textit{steady progress}, in terms advocated by Kistler \cite{Kistler_notes}, as follows.
Conditioning on an event $\cQ$ as in \eqref{def:Qevent}, for integers $n_1<n_2$ we denote the \emph{partial lineage} of a point $t\in T_N$
\begin{equation}
\label{eq-june11a}
\Delta_{n_1,n_2}(t):= \sum_{k\in \urn\cap[n_1,n_2)} Y_{I_k}(t) = \sum_{k\in Q\cap [n_1,n_2)} \sum_{e^{\varrho n_1}\le \ell < e^{\varrho n_2}} Z_\ell \log|1-e(\ell t)|.
\end{equation}
For a large parameter $K$ (which we eventually take to be of size $\gg 1/\epp$) we cut the set $\urn\cap [1,n)$ into $K$ legs $\urn_i=\urn\cap [(i-1)n/K, in/K)$, and aim to show that the set of ``rapidly rising" points
\begin{equation}
\cR_n(T_N, \xcrit-\epp) = \Big\{ t\in T_N: \Delta_{\frac{(i-1)n}{K}, \frac{in}{K}}(t) \ge (\xcrit-\epp) |\urn_i| \; \forall 2\le i\le K\Big\}
\end{equation}
is nonempty with high probability. This is enough to establish \eqref{bd:middle} since the set of rapidly rising points, intersected with the set of survivors $\cS_{\le N^c}(-C)$ from early generations, is a subset of the set of high points for the Poisson field (once $K$ is sufficiently large depending $c$). 

To show $|\cR_n(T_N, \xcrit-\epp)| \ge N^{c(\epp)}$ with high probability, we lower bound its expectation (first moment) and show it is concentrated (bound its second moment). 
The lower bound on the first moment can be obtained from the tail estimate \eqref{sketch:tail1} applied to each leg $\urn_i$ (the associated increments $\Delta_{\frac{(i-1)n}{K}, \frac{in}{K}}(t)$ are independent across $i$). 

The upper bound on the second moment is where we need to expose the logarithmic correlation structure of the Poisson field. Specifically, the key for a satisfactory bound on the second moment is to show that tail events for partial lineages
become approximately independent for points $s,t\in \tor$ that are sufficiently far apart, i.e.
\begin{equation}	\label{sketch:tail2}
\pr^{\cQ}\Big( \Delta_{n_1, n_2}(s) , \;  \Delta_{n_1, n_2}(t) \ge y|\urn_i| \Big) \le (1+o(1)) \pr\big(\tY_{\urn\cap[n_1,n_2)} \ge y|\urn\cap [n_1,n_2)|\big)^2,
\end{equation}
with $\tY_{\urn}$ as in \eqref{sketch:tail1}. 
It turns out that for the above to hold, it is not enough for $s$ and $t$ to be well separated under the usual metric $d(s,t)=\|s-t\|_\tor$ on the torus: we also need all linear combinations with sufficiently small integer weights to be well separated from zero. 
Specifically, for a parameter $\xi_0\in \N$ we define the ``distance" 
\begin{equation}	\label{def:dq}
d_{\xi_0}(s,t):= \min_{\xi,\xi'\in \{-{\xi_0},\cdots,{\xi_0}\}\setminus \{0\}} \|\xi s+\xi't\|_\tor,\qquad s, t\in \tor.
\end{equation}
Writing $A_{\xi_0}=\{-{\xi_0},\cdots,{\xi_0}\}\setminus \{0\}$, one sees this is the distance between the dilates $sA_{\xi_0}$ and $tA_{\xi_0}$ in $\tor$.
A lower bound on $d_{\xi_0}(s,t)$ amounts to a quantitative linear independence of $\{1,s,t\}$ over $\Q$, which can be seen as a two-dimensional minor arc condition 
(note that $t\notin \Maj(\xi_0,\kappa)$ enforces quantitative linear independence of $\{1,t\}$, i.e.\ quantitative irrationality).

Roughly speaking, we get \eqref{sketch:tail2} when $d_{\xi_0}(s,t) = \omega(e^{-\varrho m})$.
This is what we should expect, as $e^{-\varrho n_1}$ is the reciprocal of the lowest frequency summand appearing in \eqref{eq-june11a},
so the field should become decorrelated at separations that are large compared with this wavelength.

Finally, we discuss the argument for \eqref{bd:late}.
For the late stage the Poisson approximation is no longer available. 
Here we will condition on the cycles in $P_N$ of length $\ell\le N/W$, which fixes the set of survivors $\cS_{\le N/W}(\xcrit-\epp)$ up to generation $N/W$, and aim to show that they are not wiped out by the high frequency tail $X_N^> = X_N - X_{\le N/W}$. 

First we identify points in the torus which are ``at risk". Note that $\log|1-e(\ell t)|\le -\epp\log N$ for all $t$ in the Bohr set 
\begin{equation}
B_\ell(N^{-\epp}) = \{ t\in \tor: \|\ell t\|_\tor \le N^{-\epp}\}.
\end{equation}
If $\cS_{\le N/W}(\xcrit-\epp) \subset B_\ell(N^{-\epp})$ for some $\ell\in (N/W,N]$, then on the event that $P_N$ has a cycle of length $\ell$ we have $\max_{t\in T_N} \cS_{\le N}(\xcrit-2\epp) = \emptyset$, which we want to avoid. 
Note that $|B_\ell(N^{-\epp})\cap T_N|\approx N^{1-\epp}$, while we have only shown $|\cS_{\le N/W}(\xcrit-\epp)|\ge N^{c(\epp)}$, so we get no help from the cardinality of survivors alone. 

We will show \eqref{bd:late} using a structural dichotomy for the set of survivors $\cS_{\le N/W}(\xcrit-\epp)\subset \tor$: we say that $\cS_{\le N/W}(\xcrit-\epp)$ is \emph{structured} if $\cS_{\le N/W}(\xcrit-\epp)$ has large overlap with $B_\ell(N^{-\epp})$ for many different frequencies $\ell\in (N/W, N]$ (on the order of $N/W^{O(1)}$), and \emph{unstructured} otherwise. 
In the unstructured case, the number of bad frequencies $\ell\in (N/W,N]$ is small enough that we can argue it is unlikely that any of them are selected (i.e.\ that $P_N$ contains a cycle of that length).
For the structured case,  we can use a Vinogradov-type lemma (Lemma \ref{lem:vino}) and a pigeonholing argument to show that $\cS_{\le N/W}(\xcrit-\epp)$ must actually contain an element that is major arc -- specifically, an element of a thin, low-frequency Bohr set $B_\xi(\eta)$. But as we showed in the proof of the upper bound, such major arc points are unlikely to be in $\cS_{\le N/W}(\xcrit-\epp)$ in the first place.

\section{Preliminaries}	
\label{sec:prelim}

In this section we gather some preliminary lemmas.
In Section \ref{sec:cyc} we establish the heuristics \eqref{urn.approx} for the ``typical" distribution of cycle lengths among intervals of the form \eqref{def:micro}.
Sections \ref{sec:phi} and \ref{sec:kronecker} contain some Fourier-analytic preliminaries for Section \ref{sec:FL}.
In particular, in Section \ref{sec:kronecker} we provide quantitative estimates for logarithmic averages of a function $f:(\tor)^d\to \C$ along sequences of the form $(\ell t)_{\ell \in \N}$, in terms of the Fourier coefficients of $f$. 
In Section \ref{sec:FL} these will be applied with $f(t)=\phi_z(t):= |1-e(t)|^z$ (whose Fourier coefficients are estimated in Section \ref{sec:phi}) to obtain accurate tail estimates for the Fourier--Laplace transform of partial sums of the Poisson field $Y_N(t)$ at minor arc points $t\in \tor$.

\subsection{Typical distribution of cycle lengths at coarse scale}
\label{sec:cyc}

Recall the notation \eqref{def:NJ} and \eqref{def:Qa}.
In this subsection we prove the following:

\begin{lemma}	\label{lem:typical}
Let $m<n$ be integers and $\varrho\in (0,1)$. Assume
\begin{equation}	\label{assume:typical}
m\ge \frac{C}{\varrho}\log \frac1\varrho
\end{equation}
for a sufficiently large constant $C>0$.
Then for any $\eps\in (0,1)$,
\begin{equation}	\label{S.conc}
\cN([e^{\varrho m}, e^{\varrho n})) = \sum_{m\le\gen<n} \cN(I_\gen) = (1+O(\eps)) \varrho (n-m)
\end{equation}
with probability $1-O(\exp(-c\eps^2\varrho (n-m)))$, and 
\begin{equation}	\label{typical.Q1}
 |\urn_1\cap [m,n)| = (n-m)\big[\varrho (1-\varrho ) + O (\eps + \varrho^3) \big]
\end{equation}
with probability $1-O(\expo{ -c\eps\varrho (n-m)})$.
Moreover, 
\begin{equation}	\label{EQ2}
\e |Q_{\ge 2} \cap [m,n)| \ll \varrho^2(n-m)
\end{equation}
and for all $K\ge C'$ for a sufficiently large constant $C'>0$,
\begin{equation}	\label{typical.Q2}
\pro{  |\urn_{\ge 2}\cap [m,n)| \ge K\varrho^2 (n-m)} \le \expo{ -cK \varrho^2 (n-m)}.
\end{equation}
\end{lemma}

\begin{proof}
Recalling \eqref{def:micro},
we begin by noting that if $C>0$ is sufficiently large then we have $\rho_\gen\asymp \varrho$ for all $m\le\gen<n$. 

For \eqref{S.conc}, $\cN([e^{\varrho m}, e^{\varrho n}))$ has Poisson distribution with expectation
\[
\e \cN([e^{\varrho m}, e^{\varrho n})) = \sum_{e^{\varrho m}\le \ell<e^{\varrho n}} \frac1\ell = \varrho (n-m) + O(e^{-c\varrho m})
\]
and the claim follows from standard concentration for Poisson variables. 
For \eqref{typical.Q1},
\begin{align*}
\e |\urn_1\cap [m,n)| 
&= \sum_{m\le\gen<n} \pr(\cN(I_\gen)=1) 
= \sum_{m\le\gen<n} \rho_\gen e^{-\rho_\gen}
= \sum_{m\le\gen<n} \rho_\gen - \rho_\gen^2+O(\rho_\gen^3).
\end{align*}
Since $\rho_\gen = \varrho +O(e^{-\varrho\gen})$,
\begin{align}
\e |\urn_1\cap [m,n)| &=  \varrho (n-m)-\varrho^2(n-m) + O(\varrho^3(n-m)) + O\left(\sum_{\gen\ge m} e^{-\varrho\gen}\right)	\notag\\
&=\varrho (n-m)-\varrho^2(n-m) + O(\varrho^3(n-m) + \frac1\varrho e^{-\varrho m})	\notag\\
&= \varrho (n-m)-\varrho^2(n-m) + O(\varrho^3(n-m) ),	\label{EH1.fine}
\end{align}
where in the third line we used \eqref{assume:typical}.
\eqref{typical.Q1} is a direct consequence of the above estimate and Bernstein's inequality.
For \eqref{EQ2},
\begin{align*}
\e |\urn_{\ge2}\cap [m,n)| 
&= \sum_{m\le\gen<n} \pr(\cN(I_\gen)\ge2) 
= \sum_{m\le\gen<n} 1- e^{-\rho_\gen}(1+\rho_\gen)
\ll \varrho^2(n-m).
\end{align*}
For \eqref{typical.Q2}, 
for each $\gen\in [m,n)$ we have
\begin{align*}
\e \expo{ \ind(\cN(I_\gen)\ge2)} = e^{-\rho_\gen}(1+\rho_\gen +O(\varrho^2) ) = 1+O(\varrho^2).
\end{align*}
By independence,
\[
\e\expo{ |\urn_{\ge 2}\cap [m,n)| } = \expo{ O(\varrho^2(n-m)},
\]
and \eqref{typical.Q2} now follows from Markov's inequality. 
\end{proof}

\subsection{Estimates for $\wh{\phi_z}$}
\label{sec:phi}

In this section we record some preliminary estimates on the Fourier coefficients of the function
\begin{equation}	\label{def:phizt}
\phi:\C\times \tor\to \C, \quad \phi(z,t) =\phi_z(t)= |1-e(t)|^z.
\end{equation}
We denote the projection of $\phi$ to the subspace of functions that are mean-zero in the second argument by
\begin{equation}	\label{def:phi0}
 \phi_0(z,t) = \phi(z,t)- \int_\tor \phi(z,u)du.
\end{equation}
We consider points $z=\beta+\ii \tau$ with $\beta>0$. 
We note the pointwize bound
\begin{equation}	\label{ub:phi}
|\phi(z,u)|\le 2^\beta, \qquad u\in \tor. 
\end{equation}
For $\xi\in\Z$ we denote the Fourier coefficients
\[
\wh{\phi}_z(\xi) = \int_\tor \phi_z(t) e(-\xi t)dt.
\]

\begin{lemma}	\label{lem:phi.lb}
For $z=\beta+\ii\tau$ with $\beta\ge1$, we have 
\begin{equation}	\label{lb:intphi}
|\wh{\phi}_z(0)| = \left|\int_\tor \phi(z, u)du\right|  \gg(1+|\tau|)^{-O(1)}. 
\end{equation}



\end{lemma}

\begin{proof}
We have 
\[
\int_\tor \phi(z,u) du = \int_0^1 (2\sin\pi u)^z du = 2^{z+1} \int_0^{1/2}(\sin\pi u)^z du.
\]
With the change of variable $\sin\pi u = \sqrt{v}$, the last expression becomes
\[
\frac{2^z}{\pi} \int_0^1 u^{(z-1)/2} (1-u)^{-1/2} du = \frac{2^z}{\pi } \text{B}\Big( \frac{z+1}{2}, \frac12\Big),
\]
where B$(x,y)=\Gamma(x)\Gamma(y)/\Gamma(x+y)$ denotes the beta function. 
From Stirling's formula and the fact that the Gamma function has no zeros or poles on the right half plane, together with our assumption $\beta\ge1$, we have
\[
\Big|\text{B}\Big(\frac{z+1}{2}, \frac12\Big)\Big| \gg |1+z|^{-1/2},
\]
and the claim follows (in fact we get a lower bound on $|\int_\tor\phi(z,u)du|$ that grows exponentially with $\beta$, but we do not need such a refinement).
\end{proof}

\begin{lemma}	\label{lem:phi.fourier}
For $\beta>0$ and $\xi\in \Z\setminus \{0\}$, 
\begin{equation}	\label{phizhat0}
|\wh{\phi}_z(\xi)|  \ll \frac{10^\beta}{|\xi|^{1+\min(\beta,1)}}  , \quad \text{when } \tau=0
\end{equation}
and
\begin{equation}	\label{phizhat.gen}
|\wh{\phi}_z(\xi)|  \ll 
\frac{10^\beta|z|}{|\xi|^{1+\min(\beta,1)}}\left( 1+\frac1\beta + |z-1|\log|\xi|\right) , \quad \tau\in \R. 
\end{equation}\end{lemma}
In  
 particular, there exists a real-valued function $K^+$ defined on the complex upper half plane so that, with $z=\beta+i\tau$ with
 $\tau\in \R$ and $\beta\ge1$, and any  $\xi\in \Z\setminus \{0\}$,
\begin{equation}	\label{Kplus}
|\wh{\phi}_z(\xi)| \le \frac{K^+(z)}{|\xi|^{3/2}} ,\qquad K^+(z) \ll 10^\beta(1+\tau^2).
\end{equation}

\begin{remark}	\label{rmk:betage1}
The key point is that the bound \eqref{Kplus} is summable in $\xi$. Sharper decay rates holding for all $z$ in the right half-plane can be obtained by slightly longer arguments; in particular one can show $\phi_z(\xi)\ll_{\beta,\tau}|\xi|^{-2}$ for all $\beta>0$. 
For our purposes we only need a summable bound holding for all $\beta\ge1$.
Indeed, for the proof of Theorem \ref{thm:main} we will only need to study $\lambda(z)$ for $z=\beta+\ii \tau$ with $\beta$ in a small neighborhood of the critical value $\betacrit= \beta_*(\xcrit)$. 
From the relation
\[
1= \lambda^*(\xcrit) = \xcrit\betacrit - \lambda(\betacrit) \ge \xcrit\betacrit
\]
and the bound $\xcrit\le \log 2$ (see the discussion at \eqref{def:lambdastar}--\eqref{def:xxc}) we see that 
\begin{equation}
\betacrit\ge \frac1{\log 2} >1.
\end{equation}
Since $\beta_*=(\lambda')^{-1}$ and $\lambda$ is smooth and convex on $(0,\log 2)$, it follows that for some absolute constant $c_0>0$,
\begin{equation}	\label{beta.bounds}
1\le \beta_*(\xx) = O(1)\qquad \forall \xx\in [\xcrit-c_0,\xcrit+c_0].
\end{equation}
We emphasize that all of our arguments can be adapted to treat any fixed $\beta>0$ by using \eqref{phizhat.gen} in place of \eqref{Kplus} (only the resulting dependence on $\beta$ is somewhat messy).
\end{remark}

\begin{proof}
For $t\in (0,1)$ we have $\phi_z(t) = (2\sin \pi t)^z$, and
\[
\phi_z'(t) = \pi z 2^z (\sin\pi t)^{z-1} \cos \pi t=: \pi z2^z g(t). 
\]
Thus, for $\xi\in \Z\setminus\{0\}$, 
\begin{equation}
|\wh{\phi}_z(\xi)| = \frac1{2\pi|\xi|} |\wh{\phi_z'}(\xi)| \ll \frac{2^\beta|z|}{|\xi|} |\wh{g}(\xi)|.
\end{equation}
Letting $\eps:=1/(10|\xi|)$, we have
\[
\wh{g}(\xi) = \int_0^1 g(t) e(-\xi t) dt = \int_0^\eps + \int_{1-\eps}^1 + \int_\eps^{1-\eps}. 
\]
The first and second integrals are bounded in modulus by
\[
\int_0^\eps (\sin\pi t)^{\beta-1} dt \ll \pi^\beta \int_0^\eps t^{\beta-1}dt = \frac{(\pi \eps)^{\beta}}{\beta}. 
\]
For the third integral, integration by parts gives
\begin{align*}
 \int_\eps^{1-\eps} g(t) e(-\xi t)dt 
 &= \frac{g(t) e(-\xi t)}{-2\pi i \xi } \bigg|_\eps^{1-\eps} + \frac1{2\pi i\xi} \int_\eps^{1-\eps} g'(t) e(-\xi t) dt\\
 & \ll \frac{|g(\eps)|}{|\xi|} + \frac1{|\xi|} \int_\eps^{1-\eps} |g'(t)| dt,
\end{align*}
where in the second line we used that $g(\eps)=-g(1-\eps)$. 
For the integrand in the second term,
\[
|g'(t)| = \pi\left| (z-1)(\sin\pi t)^{z-2} \cos^2\pi t - (\sin\pi t)^z\right| \ll 1+ |z-1|(\sin\pi t)^{\beta-2}
\]
and hence
\begin{align*}
 \int_\eps^{1-\eps} |g'(t)| dt
 &\ll 1+ |z-1|\int_\eps^{1/2} (\pi t)^{\beta-2}dt.
\end{align*}
Substituting $\eps=1/(10|\xi|)$, for $\beta\ne 1$ we obtain
\[
|\wh{g}(\xi)| \ll \frac{\pi^\beta}{|\xi|^\beta} \left( 1+ \frac1\beta\right) + \frac1{|\xi|} \left( 1+ \pi^\beta |z-1| \left( \frac{1-(1/|\xi|)^{\beta-1}}{\beta-1}\right)\right),
\]
and hence
\begin{equation}	\label{phiz.hat.beta}
|\wh{\phi}_z(\xi)| \ll \frac{2^\beta |z|}{|\xi|^2} + \frac{(2\pi)^\beta |z|}{|\xi|^{1+\beta}} \left( 1+ \frac1\beta + |z-1| \left( \frac{|\xi|^{\beta-1}-1}{\beta-1}\right) \right).
\end{equation}
For $\beta=1$ we obtain
\begin{equation}	\label{phiz.hat.beta1}
|\wh{\phi}_z(\xi)| \ll \frac{|z|}{|\xi|^2} \left( 1+ |z-1|\log|\xi|\right).
\end{equation}
When $\tau=0$ we have $z=\beta$, and the above bounds become
\[
|\wh{\phi}_z(\xi)|
\ll \begin{cases}
\frac{1}{|\xi|^2} + \frac{\beta}{|\xi|^{1+\beta}}\left( \frac1\beta + 1-|\xi|^{\beta-1}\right) \ll \frac1{|\xi|^{1+\beta}} & \beta\in (0,1)\\
\frac1{|\xi|^2} & \beta = 1\\
\frac{\beta 2^\beta}{|\xi|^2} + \frac{\beta(2\pi)^\beta}{|\xi|^{1+\beta}} (1+|\xi|^{\beta-1}) \ll \frac{10^\beta}{|\xi|^2} & \beta>1
\end{cases}
\]
which in all cases is bounded by $10^\beta |\xi|^{-1-\min(\beta,1)}$, as desired. 
For the general case $\tau\in \R$ we can use convexity of the function $f(\alpha)= |\xi|^{\alpha}$ to bound
\[
\frac{|\xi|^{\beta-1} -1}{\beta-1} \le f'(\max(0,\beta-1)) = (\log |\xi|) \max(|\xi|^{\beta-1}, 1).
\]
Thus, for $\beta\le 1$, 
\[
|\wh{\phi}_z(\xi)| \ll \frac{|z|}{|\xi|^{1+\beta}}\left( \frac1\beta + |z-1|\log |\xi|\right)
\]
and for $\beta>1$,
\begin{align*}
|\wh{\phi}_z(\xi)| 
&\ll \frac{2^\beta |z|}{|\xi|^2} + \frac{(2\pi )^\beta |z|}{|\xi|^{1+\beta}} \left( 1+ |z-1||\xi|^{\beta-1}\log |\xi|\right)\\
&\ll \frac{(2\pi)^\beta |z|}{|\xi|^2} \left( 1+ |z-1|\log |\xi|\right)
\end{align*}
and the claimed bound follows from the previous two displays. 
\end{proof}

\subsection{Logarithmic averages over Kronecker sequences}
\label{sec:kronecker}

\    As an immediate consequence of 
 Lemma \ref{lem:phi.fourier}, we have that for $\Re(z)>0$ the Fourier coefficients of $\phi_z$ are summable, and hence $\phi_z$ has the uniformly convergent Fourier series expansion
\begin{equation}	\label{phiz.fourier}
\phi_z(t) = \sum_{\xi\in \Z} \wh{\phi}_z(\xi) e(\xi t). 
\end{equation}
In this subsection we show how to use the Fourier series expansion of a function $f:(\tor)^d\to \C$ and decay estimates for the Fourier coefficients (as in \eqref{Kplus} for $\phi_z$) to estimate logarithmically-weighted averages of $f$ over arithmetic progressions in $(\tor)^d$. 
Such averages (with $f=\phi_z$) will arise in Section \ref{sec:FL} in the analysis of the Fourier--Laplace transform for the field $(Y_N(t_1), \dots,Y_N(t_d))$ at fixed points $t_1,\dots, t_d\in \tor$ (we will only need to consider $d=1,2$). 
The proofs follow standard ideas from the classical subject of discrepancy theory -- see for instance \cite{Montgomery_10lectures}.

\begin{lemma}	\label{lem:fvec.error}
Let $f:\T^d\to \C$ be a continuous function with absolutely summable Fourier coefficients $\wh{f}(\bxi)$
(as in \eqref{def:fourierco}).
Let $\varrho\in (0,1/2)$ and $M\ge C/\varrho$ for a sufficiently large absolute constant $C>0$.
Then for any $\bt\in \T^d$, 
\begin{equation}
\sum_{M< \ell \le e^{\varrho} M} \frac{f(\ell \bt)}{\ell } \ll \sum_{\bxi\in \Z^d} |\wh{f}(\bxi)| \min\left(\varrho, \frac1{M\|\bxi\cdot \bt\|_\tor}\right).
\end{equation}
(Here we adopt the convention $\min(\varrho,1/x)=\varrho $ when $x=0$.)
\end{lemma}

We will use the following standard fact:

\begin{lemma}[Summation by parts]	\label{lem:sbp}
Let $f, g:[1,N]\to \C$ and assume $f$ is continuous and $g$ is continuously differentiable. 
Then
\begin{equation}	\label{sbp.general}
\sum_{\ell\le N} f(\ell)g(\ell) = g(N)\sum_{\ell\le N}f(\ell) - \int_1^N \bigg( \frac1x \sum_{\ell\le x} f(\ell)\bigg)g'(x)dx. 
\end{equation}
In particular, taking $g(x) = 1/x$ we have
\begin{equation}	\label{LtoA}
\sum_{\ell\le N} \frac{f(\ell)}{\ell} = \frac{1}N\sum_{\ell\le N} f(\ell) + \int_1^N\frac{1}{x}\sum_{\ell\le x} f(\ell) \frac{dx}{x}.
\end{equation}
\end{lemma}

\begin{proof}[Proof of Lemma \ref{lem:fvec.error}]
For any $0\ne u\in \tor$ and $y\ge 1$,
\begin{align*}
\bigg| \sum_{\ell \le y} e(\ell u)\bigg| 
&= \left| \frac{e((\lf y\rf + 1)u)-1}{e(u)-1}\right| \le \frac2{|e(u)-1|} \le \frac2{\|u\|_\tor}.
\end{align*}
By factoring out $e((\ell -x)u)$ we obtain the same bound for the sum over $x<\ell \le y$ for any $0<x<y$ with $y-x\ge 1$. Combined with the trivial bound $|\sum_{x<\ell \le y} e(\ell u)|\le \lf y-x\rf$ this gives
\begin{equation}	\label{exp.sum}
\bigg| \sum_{x<\ell \le y} e(\ell u) \bigg| \le \min\left( \lf y-x\rf, \frac{2}{\|u\|_\tor}\right), \quad \forall \, u\in \tor,\;0<x<y, \;\text{with } y-x\ge 1.
\end{equation}
Now for any $x,y$ as in \eqref{exp.sum},
\begin{equation}	\label{f.sumxy}
\bigg| \sum_{x<\ell \le y} f(\ell \bt)\bigg| = \bigg| \sum_{\bxi\in \Z^d} \wh{f}(\bxi) \sum_{x<\ell \le y}e(\ell \bxi\cdot \bt)\bigg|
\le \sum_{\bxi\in \Z^d} |\wh{f}(\bxi)| \min\left( \lf y-x\rf, \frac{2}{\|\bxi\cdot \bt\|_\tor}\right).
\end{equation}
From Lemma \ref{lem:sbp},
\begin{align*}
\sum_{\ell \in I} \frac{f(\ell \bt)}{\ell } 
&= \frac1{e^\varrho M}\sum_{\ell \le e^\varrho M} f(\ell \bt)  - \frac{1}{M}\sum_{\ell \le M} f(\ell \bt) + \int_{M}^{e^\varrho M}  \frac1x\sum_{\ell \le x} f(\ell \bt) \frac{dx}{x}\\
&=  \frac1{e^\varrho M}\sum_{M<\ell \le e^\varrho M} f(\ell \bt) -( 1-e^{-\varrho})\frac{1}{M}\sum_{\ell \le M} f(\ell \bt)
+ \int_{M}^{e^\varrho M}  \frac1x\sum_{\ell \le x} f(\ell \bt) \frac{dx}{x}.
\end{align*}
Substituting the bound \eqref{f.sumxy} (noting the conditions on $x,y$ are satisfied by our assumption $M\ge C/\varrho$),
\begin{align*}
&\bigg|\sum_{\ell \in I} \frac{f(\ell \bt)}{\ell } \bigg|\\
&\le  \frac1{e^\varrho M}\bigg| \sum_{M<\ell \le e^\varrho M} f(\ell \bt)\bigg| +( 1-e^{-\varrho})\frac{1}{M}\bigg|\sum_{\ell \le M} f(\ell \bt)\bigg|
+ \int_{M}^{e^\varrho M} \bigg| \frac1x\sum_{\ell \le x} f(\ell \bt)\bigg| \frac{dx}{x}\\
&\le \sum_{\bxi\in \Z^d} |\wh{f}(\bxi)| \bigg[  \frac1{e^\varrho M} \min\left( (e^\varrho -1)M, \frac2{\|\bxi\cdot \bt\|_\tor}\right) \\
&\qquad\qquad\qquad\qquad\qquad\qquad+ \frac{1-e^{-\varrho}}{M} \min\left( M, \frac{2}{\|\bxi\cdot \bt\|_\tor}\right) + \int_{M}^{e^\varrho M} \min\left( 1, \frac2{x\|\bxi\cdot \bt\|_\tor}\right) \frac{dx}{x}\bigg]\\
&\le \sum_{\bxi\in \Z^d} |\wh{f}(\bxi)| \bigg[ \min\left( \varrho , \frac2{M\|\bxi\cdot \bt\|_\tor}\right)\! + \!\varrho \min\left( 1 , \frac2{M\|\bxi\cdot \bt\|_\tor}\right) \!+\! \min\left( 1 , \frac2{M\|\bxi\cdot \bt\|_\tor}\right) \int_{M}^{e^\varrho M} \frac{dx}{x}\bigg] \\
&\le 3\sum_{\bxi\in \Z^d} |\wh{f}(\bxi)| \min\left( \varrho , \frac2{M\|\bxi\cdot \bt\|_\tor}\right)  .
\!\!\!\!\!\!\qedhere
\end{align*}
\end{proof}

\begin{cor}	\label{cor:fg.error}
Let $f, g:\tor\to \C$ be mean-zero continuous functions, that is, satisfying
$$\int_\tor f(u)du=\int_\tor g(u)du=0,$$ with absolutely summable Fourier coefficients $\wh{f}(\xi), \wh{g}(\xi)$. Let $M,\varrho$ be as in Lemma \ref{lem:fvec.error}. 
Then for any $t\in \tor$,
\begin{equation}	\label{bd:f.error}
\sum_{M\le \ell < e^\varrho M} \frac{f(\ell t)}{\ell } \ll \sum_{\xi\ge 1} |\wh{f}(\xi)| \min\left( \varrho, \frac1{M\|\xi t\|_\tor}\right),
\end{equation}
and for any $s,t\in \tor$, 
\begin{equation}	\label{bd:fg.error}
\sum_{M\le \ell < e^\varrho M} \frac{f(\ell s)g(\ell t)}{\ell } \ll \sum_{\xi,\xi'\in \Z\setminus \{0\}} |\wh{f}(\xi)||\wh{g}(\xi')|\min\left(\varrho, \frac1{M\|\xi s+ \xi't\|_\tor}\right).
\end{equation}
\end{cor}

\begin{proof}
By perturbing $M$ slightly we may replace the range for $\ell $ in the sums \eqref{bd:f.error}, \eqref{bd:fg.error} with $M<\ell \le e^\varrho M$.
For \eqref{bd:f.error}, from Lemma \ref{lem:fvec.error} we have that the left hand side is bounded by
\[
\sum_{\xi\in \Z} |\wh{f}(\xi)| \min\left( \varrho, \frac1{M \|\xi t\|_\tor}\right) = \varrho |\wh{f}(0)| +2\sum_{\xi\ge 1} |\wh{f}(\xi)| \min\left( \varrho, \frac1{M \|\xi t\|_\tor}\right) .
\]
By our assumption that $f$ is mean-zero we have $\wh{f}(0)=0$, and the bound \eqref{bd:f.error} follows. 
For \eqref{bd:fg.error} we apply Lemma \ref{lem:fvec.error} with $d=2$, taking $f\otimes g$ in place of $f$, and note that the Fourier coefficients of $f\otimes g$ are the products of the Fourier coefficients of $f$ and $g$, which can be bounded in the same way as we did for \eqref{bd:f.error}.
\end{proof}

In Section \ref{sec:FL} we will need to control logarithmic averages as in \eqref{bd:f.error} and \eqref{bd:fg.error}, but averaged over a sequence of intervals $I_\gen=[e^{\varrho\gen}, e^{\varrho(\gen+1)})$ (so taking $M=M_\gen=e^{\varrho\gen}$). 
The following will be applied to control such averages of the right hand sides of \eqref{bd:f.error} and \eqref{bd:fg.error}.
Write  $\log_+(x)= \max(0, \log x)$ for $x\in \R_+$.
\begin{lemma}	\label{lem:isolate}
Let $1\le m<n$ and $\varrho>0$. Let $R:\N\to \R_+$ such that for sone $\kappa\in (0,1)$ we have $R(\xi)\ge \kappa$ for all $\xi\le \xi_0$.
Then for any $\beta_0>0$,
\begin{equation}	\label{bd:isolate1}
 \sum_{m\le\gen<n} \sum_{\xi\ge 1} \frac1{\xi^{1+\beta_0}}\min\left( \varrho, \frac{e^{-\varrho\gen}}{R(\xi)}\right)  \ll 
 \frac1{\beta_0}\left(1 +\log_+\left(\frac{e^{-\varrho m}}{\varrho \kappa}\right)+ \frac{\varrho (n-m)}{\xi_0^{\beta_0}}\right).
\end{equation}
If we additionally assume $\kappa\ge \frac1\varrho e^{-\varrho m}$, then
\begin{equation}	\label{bd:isolate2}
\sum_{\xi\ge 1} \frac1{\xi^{1+\beta_0}} \sum_{m\le\gen<n} \min\left( \varrho, \frac{e^{-\varrho\gen}}{R(\xi)}\right) \ll \frac1{\beta_0}\left(\frac1{\varrho\kappa} e^{-\varrho m}  + \frac{1+ \varrho(n-m)}{\xi_0^{\beta_0}}\right).
\end{equation} 
\end{lemma}

\begin{proof}
For fixed $\xi\ge1$, 
\begin{equation}	\label{isolate:express}
 \sum_{m\le\gen<n}\min\left( \varrho, \frac{e^{-\varrho\gen}}{R(\xi)}\right) =\varrho \left|\left\{ \gen\in [m,n): e^{\varrho\gen}\le\frac{1}{\varrho R(\xi)}\right\}\right|+\frac1{R(\xi)}\sum_{\substack{m\le\gen<n\\ e^{-\varrho\gen}<\varrho R(\xi)}} e^{-\varrho\gen} .
\end{equation}
For the second term on the right hand side we can sum the geometric series, whose largest summand is at most $\min(e^{-\varrho m}, \varrho R(\xi))$, to obtain 
\[
\frac1{R(\xi)}\sum_{\gen: e^{-\varrho\gen}<\varrho R(\xi)} e^{-\varrho\gen}  \ll \frac1{\varrho R(\xi)}\min(e^{-\varrho m}, \varrho R(\xi)) =\min\left( 1, \frac{e^{-\varrho m}}{\varrho R(\xi)}\right).
\]
The first term on the right hand side of \eqref{isolate:express} is bounded by 
\[
\varrho\left| \left[ m, \min\left( n, \frac1\varrho \log\frac{1}{\varrho R(\xi)}\right)\right]\right| 
= \min\left( \varrho(n-m), \log_+\left(\frac{e^{-\varrho m}}{\varrho R(\xi)}\right)\right).
\]
Under the hypothesis $\kappa\ge \frac1\varrho e^{-\varrho m}$ these bounds give
\begin{align}
&\sum_{\xi\ge 1} \frac1{\xi^{1+{\beta_0}}} \sum_{m\le\gen<n} \min\left( \varrho, \frac{e^{-\varrho\gen}}{R(\xi)}\right)	\notag \\
&\qquad\ll \frac1\varrho e^{-\varrho m} \sum_{\xi\le \xi_0} \frac1{\xi^{1+{\beta_0}}R(\xi)} + \sum_{\xi>\xi_0} \frac1{\xi^{1+{\beta_0}}} \left( 1+ 
 \min\left( \varrho(n-m), \log_+\left(\frac{e^{-\varrho m}}{\varrho R(\xi)}\right)\right)\right)	\notag\\
 &\qquad \ll \frac1{{\beta_0}\xi_0^{\beta_0}}+\frac1\varrho e^{-\varrho m} \sum_{\xi\le {\xi_0}} \frac1{\xi^{1+{\beta_0}}R(\xi)} + \sum_{\xi>{\xi_0}} \frac1{\xi^{1+{\beta_0}}} 
 \min\left( \varrho(n-m), \log_+\left(\frac{e^{-\varrho m}}{\varrho R(\xi)}\right)\right)	\notag\\ 
&\qquad\ll \frac1{\beta_0}\left(\frac1{\varrho\kappa} e^{-\varrho m}  + \frac1{\xi_0^{\beta_0}}\left(1+ \varrho(n-m)\right)\right),	\notag
\end{align}
which gives \eqref{bd:isolate2} as desired. 
Without the assumption $\kappa\ge \frac1\varrho e^{-\varrho m}$ we have
\begin{align*}
\sum_{\xi\ge 1} \frac1{\xi^{1+{\beta_0}}} \sum_{m\le\gen<n} \min\left( \varrho, \frac{e^{-\varrho\gen}}{R(\xi)}\right)	
& \ll \sum_{\xi\le {\xi_0}} \frac1{\xi^{1+{\beta_0}}}\left( 1+ \log_+\left( \frac{e^{-\varrho m}}{\varrho \kappa}\right)\right) + \sum_{\xi>{\xi_0}} \frac1{\xi^{1+{\beta_0}}} \varrho (n-m) \\
&\ll \frac1{\beta_0}\left(1 + \log_+\left(\frac{e^{-\varrho m}}{\varrho \kappa}\right)+ \frac{\varrho (n-m)}{\xi_0^{\beta_0}}\right),
\end{align*}
which yields \eqref{bd:isolate1}.
\end{proof}

\section{Fourier--Laplace transform and tail bounds for the Poisson field}
\label{sec:FL}

Recall the Poisson field $\YY_J(t)$ defined for an interval $J\subset\R_+$ in \eqref{def:YJ}. 
In this section we let $\varrho\in (0,1/2)$ be a small parameter, $m,n\in \N$ with $m<n$, and set $J=[e^{\varrho m}, e^{\varrho n})$.
With the notation \eqref{def:micro} we have
\begin{equation}
\YY_J(t) = \sum_{m\le\gen<n} \YY_{I_\gen}(t).
\end{equation}
Our aim is to estimate the probability of tail events
\begin{equation}	\label{tailevents}
\{\YY_J(t)\ge y\}, \qquad \{\YY_J(s), \YY_J(t)\ge y\}
\end{equation}
for fixed $s,t\in \tor$ and $y>0$ (we will take $y$ at scale $\varrho(n-m)$). 
We do this under conditioning on an event $\cQ=\cQ(m,n,\urn,\turn)$ of the form \eqref{def:Qevent}, which fixes the intervals $I_\gen$ for which the corresponding partial sum of the Poisson variables $Z_\ell$ is 0, 1, or $\ge 2$. 
Thus, we fix $\urn, \turn$ disjoint subsets of $[m,n)\cap \Z$.

For the duration of this section we view $m, n, \varrho, \urn,\turn$ as fixed; we will place additional assumptions on these parameters in the lemmas and propositions. In our estimates on the tail events \eqref{tailevents} we will quantify the dependence on parameters sufficiently that we can later take parameters such as $\varrho$ to depend on $n,m$.
We will assume throughout that $\urn\ne \emptyset$ and
\begin{equation}	\label{m:lb}
m\ge \frac{C_0}\varrho\log\frac1\varrho
\end{equation}
for a sufficiently large constant $C_0>0$.
Note that since $\rho_\gen=\varrho + O(e^{-\varrho m})$, under \eqref{m:lb} we have $\rho_\gen\asymp \varrho$.
In Section \ref{sec:onept} we allow $\turn\ne \emptyset$, while in Sections \ref{sec:refined}--\ref{sec:decorr} we prove refined estimates under the assumption $\turn=\emptyset$.

\subsection{Fourier--Laplace transform for the field at one point}		\label{sec:onept}

Throughout this section, $\beta$ and $\tau$ denote the real and imaginary parts of a point $z$ in the right half-plane. 
For $t\in \tor$, $\beta>0$ and $\tau\in \R$, denote the (conditional) Fourier--Laplace transform of $Y_J(t)$ by
\begin{equation}	\label{def:Phi}
\Phi_t(z) = \Phi_t(z; m, n, \urn, \turn) :=\e \big(\exp(zY_J(t)) \mid \cQ(m,n,\urn,\turn)\big).
\end{equation}
By the independence of the variables $\{Y_{I_\gen}(t)\}_{m\le\gen<n}$ (which survives under conditioning on the event $\cQ$), the Fourier--Laplace transform factorizes as 
\begin{align}	
\Phi_t(z) &= \prod_{\gen\in\urn} \e\big[\exp(zY_{I_\gen}(t)) \, |\, \cN(I_\gen)=1\big] \prod_{\gen\in \turn} \e \big[\exp(zY_{I_\gen}(t))\,|\, \cN(I_\gen)\ge 2\big]	\notag\\
 &=: \prod_{\gen\in\urn} \vphi_{\gen,t}(z) \prod_{\gen\in \turn} \tphi_{\gen,t}(z).	\label{def:phi}
\end{align}
We also denote the cumulant generating functions
\begin{equation}	\label{def:Lambda}
\Lambda_t(z) = \Lambda_t(z; m,n,\urn,\turn) := \log\Phi_t(z)
\end{equation}
and 
\begin{equation}	\label{def:lambda.kt}
\lambda_{\gen,t}(z):= \log \vphi_{\gen,t}(z),
\end{equation}
where we take the principal branch of the logarithm with branch cut along the negative real axis and imaginary part constrained to $[-\pi,\pi)$. 

For fixed $\gen\ge1$ we have
\begin{align}
\vphi_{\gen,t}(\zz) 
&= \sum_{\ell\in I_\gen} \pr(Z_\ell=1|\cN(I_\gen)=1) \expo{ \zz\log|1-e(\ell t)|}	\notag\\
&=\sum_{\ell\in I_\gen} \frac{\pr(Z_\ell=\cN(I_\gen)=1)}{\pr(\cN(I_\gen)=1)}  |1-e(\ell t)|^\zz	\notag\\
&= \sum_{\ell\in I_\gen} \frac{(1/\ell)e^{-1/\ell} \prod_{i\in I_\gen\setminus\{\ell\}}e^{-1/i}}{\rho_\gen e^{-\rho_\gen}}  |1-e(\ell t)|^\zz\notag\\
&= \frac1{\rho_\gen} \sum_{\ell\in I_\gen} \frac{\phi(\zz,\ell t)}{\ell} ,	\label{vp.log}
\end{align}
where $\phi$ was defined in \eqref{def:phizt}.
For $\tphi_{\gen,t}(\zz)$,
\begin{align*}
\tphi_{\gen,t}(\zz) &= \frac{\e e^{\zz \YY_{I_\gen}(t)} \ind(\cN(I_\gen)\ge 2)}{\pr(\cN(I_\gen)\ge 2)}\\
&=\frac{ \e e^{\zz \YY_{I_\gen}(t)} - \pr(\cN(I_\gen)=1) \e[e^{\zz  \YY_{I_\gen}(t)}| \cN(I_\gen)=1] - \pr(\cN(I_\gen)=0)}{1-\pr(\cN(I_\gen) = 1)-\pr(\cN(I_\gen)=0)}\\
&=\frac{ \e e^{\zz \YY_{I_\gen}(t)} - \rho_\gen e^{-\rho_\gen} \vphi_{\gen,t}(\zz) - e^{-\rho_\gen}}{1-\rho_\gen e^{-\rho_\gen}-e^{-\rho_\gen}}.
\end{align*}
Recalling that for $Z\sim \text{Poi}(\rho)$ and $\zz\in \C$ we have $\e \expo{\zz Z} = \expo{ \rho(e^\zz -1)}$, 
\begin{align*}
\e e^{\zz \YY_{I_\gen}(t)} &= \prod_{\ell\in I_\gen} \e \expo{ Z_\ell \log(|1-e(\ell t)|^\zz)} \\
&= \exp\bigg( \sum_{\ell\in I_\gen} \frac1\ell\left( |1-e(\ell t)|^\zz-1\right)\bigg) \\
&= \expo{ \rho_\gen(\vphi_{\gen,t}(\zz)-1)}.
\end{align*}
Thus,
\begin{align}
\tphi_{\gen,t}(\zz) &= \frac{ \expo{ \rho_\gen(\vphi_{\gen,t}(\zz)-1)}- \rho_\gen e^{-\rho_\gen} \vphi_{\gen,t}(\zz) - e^{-\rho_\gen}}{1-\rho_\gen e^{-\rho_\gen}-e^{-\rho_\gen}}	\notag\\
&= \frac{\expo{\rho_\gen\vphi_{\gen,t}(\zz)}-1-\rho_\gen\vphi_{\gen,t}(\zz)}{e^{\rho_\gen} - 1 - \rho_\gen}.	\label{tp.log}
\end{align}
From Taylor expansion and \eqref{ub:phi},
\begin{equation}	\label{tphi.crude}
|\tphi_{\gen,t}(z)| = |\vphi_{\gen,t}(z)|^2(1+O(2^\beta\varrho)) = O(4^\beta).
\end{equation}
Combining the above with \eqref{def:phi}, we have that for $\varrho\le c10^{-\beta}$, 
\begin{equation}	\label{Phit.phi}
\Phi_t(z) = e^{O(\beta|\turn|)} \prod_{\gen\in\urn} \vphi_{\gen,t}(z).
\end{equation}
Thus, for estimating $\Lambda_t(z)=\log\Phi_t(z)$ our task reduces to estimating the terms $\vphi_{\gen,t}(z)$, which we will do through the expression \eqref{vp.log} and the Fourier-analytic bound \eqref{bd:f.error} applied with $f=\phi(z,\cdot)$. 

Our first result of this section is an upper bound on the Laplace transform of $Y_J(t)$ holding under relatively mild conditions on the point $t\in \tor$. 
Recall the function $\lambda(\cdot)$, see  \eqref{def:LA}.

\begin{prop}	\label{prop:Phi_ub}
Let $\beta>1$ and assume $\varrho\le c2^{-\beta}$ for a sufficiently small constant $c>0$.
Let $t\in \tor \setminus \Maj(\xi_0,\kappa)$ (see \eqref{def:Maj}) for some $\xi_0\in \N$ and $\kappa\in (0,1)$.
We have
\begin{align*}
\Lambda_t(\beta) &\le  \LA(\beta)|\urn| +  O\big(\beta|\turn| + K^+(\beta)\Err_0\big),
\end{align*}
 where 
\begin{equation}		\label{def:Ups0}	
\Err_0=\Err_0(\varrho,m,n,\xi_0,\kappa):=  \frac{n-m}{\xi_0^{1/2}}   +  \frac1\varrho \left( 1+ \log \frac1{\varrho\kappa}\right)  
\end{equation}
and $K^+$ is as in \eqref{Kplus}.
\end{prop}

\begin{proof}
Taking logs in \eqref{Phit.phi} we have
\begin{equation}	\label{Lt1.above}
\Lambda_t(\beta) = \sum_{\gen\in\urn}  \lambda_{\gen,t}(\beta) + O(\beta|\turn|).
\end{equation}
To estimate the terms $\lambda_{\gen,t}(\beta) = \log\vphi_{\gen,t}(\beta)$, from \eqref{vp.log} we can split each $\vphi_{\gen,t}(\beta)$ as 
\begin{equation}	\label{vphi.split}
\vphi_{\gen,t}(\beta) = \int_\tor \phi(\beta,u)du + \frac1{\rho_\gen} \sum_{\ell\in I_\gen} \frac{\phi_0(\beta, \ell t)}{\ell},
\end{equation}
where $\phi_0$ was defined in \eqref{def:phi0}. 
Since $\phi_0(\beta,\cdot)$ is mean-zero with nonzero Fourier coefficients the same as those of $ \phi(\beta,\cdot)$, from Lemma \ref{lem:phi.fourier} and \eqref{bd:f.error} we have
\begin{equation}
 \sum_{\ell\in I_\gen} \frac{\phi_0(\beta, \ell t)}{\ell} \ll K^+(\beta)\sum_{\xi\ge 1} \frac1{\xi^{3/2}} \min\left( \varrho, \frac{e^{-\varrho\gen}}{\|\xi t\|_\tor}\right).
\end{equation}
Under our assumption \eqref{m:lb} we have $\rho_\gen \gg\varrho$ for all $\gen\ge m$, so 
\begin{align*}
\vphi_{\gen,t}(\beta) 
&\le \int_\tor \phi(\beta,u)du + O \bigg( \frac{K^+(\beta)}{\varrho}  \sum_{\xi\ge 1} \frac1{\xi^{3/2}} \min\left( \varrho, \frac{e^{-\varrho\gen}}{\|\xi t\|_\tor}\right) \bigg)\\
&\le \int_\tor \phi(\beta,u)du \bigg( 1 + O \bigg( \frac{K^+(\beta)}{\varrho}  \sum_{\xi\ge 1} \frac1{\xi^{3/2}} \min\left( \varrho, \frac{e^{-\varrho\gen}}{\|\xi t\|_\tor}\right) \bigg)\bigg),
\end{align*}
for all $m\le\gen<n$, where in the second line we used that fact that $\int_\tor \phi(\beta,u)du = e^{\lambda(\beta)} \ge e^{\lambda(0)} = 1$. 
Taking logs and substituting the bound into \eqref{Lt1.above} we obtain
\begin{align*}
\Lambda_t(\beta)
&\le \lambda(\beta)|\urn| + O(\beta|\turn|) + O\bigg( \frac{K^+(\beta)}{\varrho} \sum_{m\le\gen<n}  \sum_{\xi\ge 1} \frac1{\xi^{3/2}} \min\left( \varrho, \frac{e^{-\varrho\gen}}{\|\xi t\|_\tor}\right)\bigg) \\
&\le  \lambda(\beta) |\urn| + O(\beta|\turn|) + O\bigg( \frac{K^+(\beta)}{\varrho} \left( 1+ \log_+\left( \frac{e^{-\varrho m}}{\varrho \kappa}\right) + \frac{\varrho (n-m)}{\xi_0^{1/2}}\right)\bigg),
\end{align*}
where in the second line we applied \eqref{bd:isolate1}. The desired upper bound follows.
\end{proof}

From Proposition \ref{prop:Phi_ub} with $\beta=\beta_*(\xx)$ and Markov's inequality we deduce the following upper tail bound.

\begin{cor}[Upper bound for the upper tail]	
\label{cor:Yt.upper}
Assume $\varrho$ is at most a sufficiently small constant and satisfies \eqref{m:lb}.
Let $\yy\in [\xcrit-\eps_0,\xcrit+\eps_0]$ with $\eps_0$ as in \eqref{beta.bounds}. 
Let $t\in \tor \setminus \Maj(\xi_0,\kappa)$ for some $\xi_0\in \N$ and $\kappa\in (0,1)$.
Then
\begin{align}
&\pro{ \YY_{J}(t) \ge \yy |\urn| \,\Big|\, \cQ(m,n,\urn,\turn) } \le \exp\Big( -\LA^*(\yy)|\urn| +O(|\turn|+\Err_0)\Big).	\label{YJt.upper}
\end{align}
\end{cor}

The above will be sufficient for proving the upper bound in Theorem \ref{thm:main} (in particular, the reader may now safely skip to Section \ref{sec:upper}; in the remainder of the section we establish bounds that will be used in Section \ref{sec:lower}).

\subsection{Refined estimates for the field at one point}	\label{sec:refined}

For the proof of the lower bound on $\max_{t\in \tor}Y_N(t)$ we need more refined estimates on $\Lambda_t(z)$, 
holding both from above and below. For such bounds we need to make stronger Diophantine assumptions on the point $t\in \tor$. 
It turns out that to have sufficiently strong error estimates, in Section \ref{sec:lower} we will also have to take $\varrho$ small enough that $\turn=\emptyset$ with high probability. 
Thus, for the remainder of the section we condition on the event $\cQ(m,n,\urn,\emptyset)$ for some fixed nonempty set $\urn\subset \Z\cap [m,n)$.

\begin{prop}	\label{prop:Phi}
Fix $z=\beta+\ii\tau$ with $\beta\ge1$ and $\tau\in \R$.
Assume $\varrho\le c2^{-\beta}$ for a sufficiently small constant $c>0$. 
There exists $K^*(z)\ge1$ with $K^*\ll e^{O(\beta)}(1+|\tau|^{O(1)})$ such that the following holds. 
Let $t\in \tor\setminus \Maj(\xi_0,\kappa)$ with 
\begin{equation}	\label{ass:xi0kappa}
\xi_0\ge (K^*)^{2},\qquad \kappa \ge K^* \frac1\varrho e^{-\varrho m}.
\end{equation}
Letting
\begin{equation}	\label{def:Ups1}
\Upsilon_1 = \Upsilon_1(\varrho, n, m, \xi_0, \kappa) := \frac{n-m}{\xi_0^{1/2}} + \frac1{\varrho \xi_0^{1/2}} + \frac1{\varrho^2\kappa} e^{-\varrho m},
\end{equation}
we have
\begin{equation}	\label{bd:Phi}
\Phi_t(z) =\Phi_t(z;m,n,\urn,\emptyset)=\expo{\lambda(z)|\urn| + O(  K^*(z)\Upsilon_1)} .
\end{equation}
\end{prop}

\begin{proof}
We begin with \eqref{bd:Phi}. 
From \eqref{Phit.phi} it suffices to estimate $\vphi_{\gen,t}(z)$ for $m\le\gen<n$. 
We decompose $\vphi_{\gen,t}(z)$ as in \eqref{vphi.split} (with $\beta$ replaced by $z$). For the error term, by Lemma \ref{lem:phi.fourier}, \eqref{Kplus}, 
\eqref{bd:f.error}, and our assumption \eqref{m:lb} we have
\begin{equation}	\label{logPhi2.first}
\frac1{\rho_\gen} \sum_{\ell\in I_\gen} \frac{\phi_0(z,\ell t)}{\ell} \ll \frac{K^+(z)}{\varrho} \sum_{\xi\ge1} \frac1{\xi^{3/2}} \min\left( \varrho, \frac{e^{-\varrho\gen}}{\|\xi t\|_\tor}\right) .
\end{equation}
Denoting the right hand side by $\theta_\gen$, we have 
\begin{equation}	\label{thetagen}
\vphi_{\gen,t}(z) = \int_\tor \phi(z,u)du + O(\theta_\gen) = (1+ O((1+|\tau|)^{O(1)}\theta_\gen)) \int_\tor \phi(z,u)du \quad \forall m\le\gen<n,
\end{equation}
where in the second equality we applied Lemma \ref{lem:phi.lb}. 
By our assumption that $t\notin \Maj(\xi_0,\kappa)$,
\[
\theta_\gen \le \frac{K^+}{\varrho} \left( \sum_{\xi\le \xi_0} \frac1{\xi^{3/2}}\frac{e^{-\varrho\gen}}{\kappa} + \sum_{\xi>\xi_0} \frac{ \varrho}{\xi^{3/2}} \right) \ll K^+ \left( \frac{e^{-\varrho\gen}}{\varrho \kappa} + \frac1{\xi_0^{1/2}}\right). 
\]
Thus, by taking $K^*\ge CK^+(1+|\tau|)^C$ for a sufficiently large absolute constant $C>0$, from our assumption \eqref{ass:xi0kappa} we can make the error term $O((1+|\tau|)^{O(1)}\theta_\gen)$ in \eqref{thetagen} smaller than $1/2$, say, for all $m\le\gen<n$. 
In particular we can replace the multiplicative errors $1+O((1+|\tau|)^{O(1)}\theta_\gen)$ with $\exp(O((1+|\tau|)^{O(1)}\theta_\gen))$. Doing this for each $\gen$ and substituting these bounds into \eqref{Phit.phi} we obtain
\begin{equation}
\Phi_t(z) = \left( \int_\tor \phi(z,u)du\right)^{|\urn|} \exp\bigg( O\bigg( (1+|\tau|)^{O(1)}\sum_{m\le\gen<n}\theta_\gen\bigg)\bigg). 
\end{equation}
Summing the bound \eqref{logPhi2.first} over $\gen$, by our assumption on $\kappa$ we can apply \eqref{bd:isolate2} to obtain
\[
\sum_{m\le\gen<n} \theta_\gen \ll \frac{K^+}{\varrho} \left( \frac{e^{-\varrho m}}{\varrho \kappa} + \frac{1+\varrho(n-m)}{\xi_0^{1/2}}\right)
\]
and the desired bound \eqref{bd:Phi} follows by substituting the above into the previous line.
\end{proof}

\begin{cor}[Refined upper tail estimate]	
\label{cor:Yt.upper2}
Assume $\varrho$ is at most a sufficiently small constant and satisfies \eqref{m:lb}.
Let $\yy\in [\xcrit-\eps_0,\xcrit+\eps_0]$ with $\eps_0$ as in \eqref{beta.bounds}. 
Let $t\in \tor\setminus \Maj(\xi_0,\kappa)$, where $\xi_0,\kappa$ satisfy \eqref{ass:xi0kappa} with $K^*=K^*(\beta_*(\yy))=O(1)$. 
Then
\begin{equation}
\pro{ Y_J(t) \ge \yy |\urn| \,\Big|\, \cQ(m,n,\urn, \emptyset) } \le \expo{ -\lambda^*(\yy)|\urn| + O(\Upsilon_1)}.
\end{equation}
\end{cor}

\subsection{Fourier--Laplace transform for the field at two points}	\label{sec:2points}

In this section we let $s, t\in \tor$ and consider the mixed Fourier--Laplace transform for the Poisson field $Y_J$ at two points:
\begin{equation}
\Psi_{s,t}(w,z) = \Psi_{s,t}(w,z; m,n,\urn) := \e \left( \exp( wY_J(s) + zY_J(t)) \mid \cQ(m,n,\urn,\emptyset) \right), 
\end{equation}
where $w=\alpha+\ii\sigma$ and $z=\beta+\ii\tau$ are points in the right half-plane. 
Whereas in the previous subsection we showed $\Phi_t(z)$ is well approximated by $\exp(\lambda(z)|\urn|)$, assuming $t$ does not lie in a low-frequency Bohr set, here we will need to rule out atypical arithmetic relationships between $s$ and $t$. 
Recall that ``distance" $d_{\xi_0}(s,t)$ defined in \eqref{def:dq}. 
In this subsection we establish the following:

\begin{prop}	\label{prop:Psi}
Let $w=\alpha+\ii\sigma$ and $z=\beta+\ii\tau$ with $\alpha,\beta\ge1$, and 
assume $ (n-m)^{-1}\ll \varrho\le c2^{-\alpha-\beta}$ for a sufficiently small constant $c>0$. 
There exists $\wt{K}(w,z)\ge1$ with $\wt{K}\ll e^{O(\alpha+\beta)}(1+|\sigma|+|\tau|)^{O(1)})$ such that the following holds. 
Let $s,t\in \tor\setminus \Maj(\xi_0,\kappa)$ with 
\begin{equation}	\label{ass:xi0kappabis}
\xi_0\ge \wt{K}^{4},\qquad \kappa \ge \wt{K} \frac1\varrho e^{-\varrho m}.
\end{equation}
Additionally assume 
\begin{equation}
d_{\xi_0}(s,t) \ge \Delta \frac1\varrho e^{-\varrho m} \qquad \text{ for some} \quad \Delta \ge \wt{K}^2.
\end{equation}
Letting
\begin{equation}	\label{def:Ups2}
\Upsilon_2 = \Upsilon_2(\varrho, n, m, \xi_0, \kappa,\Delta) := \frac{n-m}{\xi_0^{1/2}} + \frac1{\varrho\Delta} + \frac{e^{-\varrho m}}{\varrho^2\kappa},
\end{equation}
we have
\begin{equation}	\label{bd:Psi}
\Psi_{s,t}(w,z) = \expo{(\lambda(w)+\lambda(z))|\urn|+ O(  \wt{K}^2\Upsilon_2 )}.
\end{equation}
\end{prop}

\begin{proof}
By similar lines as in \eqref{def:phi}--\eqref{tp.log} we can express
\begin{equation}	\label{psik.st}
\Psi_{s,t}(w,z) = \prod_{\gen\in\urn} \bigg( \frac1{\rho_\gen} \sum_{\ell\in I_\gen} \frac{\phi(w,\ell s)\phi(z,\ell t)}{\ell}\bigg) =: \prod_{\gen\in\urn} \psi_\gen^{s,t}(w,z) .
\end{equation}
Splitting $\phi(w,\cdot),\phi(z,\cdot)$ into their average and mean-zero components, we have
\begin{align}
\psi_\gen^{s,t}(w,z) &= \left( \int_\tor \phi(w, u)du\right)\left( \int_\tor \phi(z,u)du\right)	\notag\\
& + \left( \int_\tor \phi(w, u)du\right)\frac1{\rho_\gen} \sum_{\ell\in I_\gen} \frac{\phi_0(z, \ell t) 
}{\ell} +\left( \int_\tor \phi(z,u)du\right)\frac1{\rho_\gen} \sum_{\ell\in I_\gen} \frac{\phi_0(w, \ell s) 
}{\ell}   \notag\\
&+ \frac1{\rho_\gen} \sum_{\ell\in I_\gen} \frac{\phi_0(w, \ell s)\phi_0(z, \ell t)}{\ell}.	\label{psi.split}
\end{align}
From \eqref{bd:f.error} and Lemma \ref{lem:phi.fourier},
\begin{equation}
 \sum_{\ell\in I_\gen} \frac{\phi_0(w, \ell s) }{\ell} \ll K^+(w) \sum_{\xi\ge 1} \frac1{\xi^{3/2}} \min\left(\varrho, \frac{e^{-\varrho\gen}}{\|\xi s\|_\tor}\right)
\end{equation}
and similarly with $w,\alpha,\sigma,s$ replaced by $z,\beta,\tau,t$.
From \eqref{bd:fg.error} and Lemma \ref{lem:phi.fourier},
\begin{equation}
 \sum_{\ell\in I_\gen} \frac{\phi_0(w, \ell s) \phi_0(z,\ell t) }{\ell} \ll K^+(w)K^+(z) \sum_{|\xi|,|\xi'|\ne 0} \frac1{|\xi\xi'|^{3/2}} \min\left(\varrho, \frac{e^{-\varrho\gen}}{\|\xi s+\xi' t\|_\tor}\right).
\end{equation}
Inserting these estimates into \eqref{psi.split} and applying Lemma \ref{lem:phi.lb} we have
\begin{equation}	\label{psik.st.above}
\psi_\gen^{s,t}(w,z) = \left( \int_\tor \phi(w, u)du\right)\left( \int_\tor \phi(z,u)du\right) \left( 1+ O(\check{K} \theta_\gen'+\check{K}^2\theta_\gen'')  \right)
\end{equation}
where
\[
\theta_\gen':= \frac1{\varrho}\sum_{\xi\ge 1} \frac1{\xi^{3/2}} \min\left(\varrho, \frac{e^{-\varrho\gen}}{\|\xi s\|_\tor\wedge\|\xi t\|_\tor}\right)
\]
and
\[
\theta_\gen'':=\frac1\varrho\sum_{|\xi|,|\xi'|\ne 0} \frac1{|\xi\xi'|^{3/2}} \min\left(\varrho, \frac{e^{-\varrho\gen}}{\|\xi s+\xi' t\|_\tor}\right)
\]
and we have taken
\begin{equation}
\check{K}(w,z) = C(1+|\sigma|+|\tau|)^C(K^+(w)+K^+(z))
\end{equation}
for a sufficiently large absolute constant $C>0$.
Following similar arguments as in the proof of Proposition \ref{prop:Phi}, by taking $\wt{K}=C'\check{K}$ for a sufficiently large constant $C'>0$ and using our assumptions on $\xi_0,\kappa$ and $\Delta$, we can make the multiplicative error $1+O(\check{K}\theta_\gen' + \check{K}^2\theta_\gen'')$ in \eqref{psik.st.above} bounded between $1/2$ and $2$ for all $m\le\gen<n$. In particular, this term can be replaced by $\expo{ O(\check{K}\theta_\gen' + \check{K}^2\theta_\gen'')}$. 
Substituting the resulting expression for $\psi_\gen^{s,t}(w,z)$ into \eqref{psik.st} we obtain
\begin{equation}	\label{Psi.above2}
\Psi_{s,t}(w,z) = e^{(\lambda(w)+ \lambda(z))|\urn|} \exp\bigg( O\bigg(  \wt{K} \sum_{m\le\gen<n} \theta_\gen' + \wt{K}^2 \sum_{m\le\gen<n} \theta_\gen''\bigg)\bigg).
\end{equation}
Applying \eqref{bd:isolate1},
\[
\sum_{m\le\gen<n} \theta_\gen' \ll \frac{n-m}{\xi_0^{1/2}} + \frac1{\varrho \xi_0^{1/2}} + \frac1{\varrho^2\kappa}e^{-\varrho m} \ll \frac{n-m}{\xi_0^{1/2}} + \frac1{\varrho^2\kappa}e^{-\varrho m},
\]
where we noted the middle term in the first bound is controlled by the first term by our assumed lower bound on $\varrho$.
We express the sum over $\theta_\gen''$ as
\begin{align*}
\sum_{m\le\gen<n} \theta_\gen'' 
&= \frac1\varrho \sum_{|\xi|\ge1} \frac1{\xi^{3/2}} \sum_{m\le\gen<n} \sum_{|\xi'|\ge1}\frac1{|\xi'|^{3/2}} \min\left(\varrho, \frac{e^{-\varrho\gen}}{\|\xi s+ \xi't\|_\tor}\right) .
\end{align*}
For fixed $\xi$ with $|\xi|\le \xi_0$, the inner sum over $\gen$ and $\xi'$ is bounded by
\[
\frac1\Delta + \frac{\varrho (n-m)}{\xi_0^{1/2}}
\]
by \eqref{bd:isolate2}, our assumption on $\Delta$, and the lower bound on $\varrho$. 
For $|\xi|>\xi_0$ we can bound the inner sum by $\varrho(n-m)$. 
Combining these and summing over $\xi$ gives
\[
\sum_{m\le\gen<n} \theta_\gen'' \ll \frac1\varrho\left( \frac1\Delta + \frac{1+\varrho(n-m)}{\xi_0^{1/2}} \right).
\]
Substituting the bounds on the sums over $\theta_\gen'$ and $\theta_\gen''$ into \eqref{Psi.above2} we obtain
\begin{align*}
&\Psi_{s,t}(w,z) = \exp\bigg\{ (\lambda(w)+ \lambda(z))|\urn| 
+O\bigg( \wt{K} \frac{e^{-\varrho m}}{\varrho^2 \kappa} + \wt{K}^2 \bigg( \frac{n-m}{\xi_0^{1/2}}  + \frac1{\varrho\Delta}  \bigg) \bigg)\bigg\}.
\end{align*}
The desired bound follows from replacing the first instance of $\wt{K}$ above with $\wt{K}^2$ (which we can do since $\wt{K}\ge 1$). 
\end{proof}

\subsection{Approximation and decorrelation for tail events}	\label{sec:decorr}

Throughout this subsection we drop the subscript $J$ from $\YY_J(t)$.
We will also use our notation \eqref{def:eHK} (as in the previous two subsections we take $\cQ=\cQ(m,n,\urn,\emptyset)$).

In this section we approximate the upper tail of $Y(t)$ (conditional on the event $\cQ$) for fixed $t\in \tor$ by the upper tail of a sum of i.i.d.\ random variables $\tY=\sum_{p\le q}\LL_p$, and also the joint upper tail of $Y(s),Y(t)$ for fixed $(s,t)\in \tor^2$ to the joint upper tail of two independent copies of $\tY$.
Of course, for the approximation to be accurate we will need to make arithmetic assumptions on the points $s,t$.

Let $U_1,\dots, U_q$ be i.i.d.\ uniform random elements of $\tor$, put $\LL_p=\log|1-e(U_p)|$ for $1\le p\le q$, and denote $\tY=\sum_{p\le q}\LL_p$.
From \eqref{def:LA}, $\tY$ has Fourier--Laplace transform
\begin{equation}
\e e^{\zz \tY} = e^{\lambda(z)q}.
\end{equation}
From the Bahadur--Rao theorem \cite{BaRa60} (see also \cite[p.\ 110]{DeZe_book}), for any $\yy>0$, letting $\beta>0$ such that $\yy=\lambda'(\beta)$, we have
\begin{equation}	\label{BaRa}
\pr(\tY\ge \yy q)=(1+ o_{\yy;\, q\to \infty}(1))\frac{e^{-\lambda^*(\yy) q}}{\beta\sqrt{\lambda''(\beta) q}} .
\end{equation}

\begin{prop}	\label{prop:Yt.approx}
Let $\yy\in [\xcrit-\eps_0,\xcrit+\eps_0]$ with $\eps_0$ as in \eqref{beta.bounds}. 
Assume $\varrho\asymp (n-m)^{-b}$ for some fixed constant $b\in (0,1)$, and that $q:=|\urn| \asymp \varrho (n-m)$.
Let $\xi_0\ge (n-m)^{C_1}$ for a sufficiently large constant $C_1>0$, and let $\kappa\ge e^{-\varrho m/2}$.
Let $t\in \tor \setminus \Maj(\xi_0,\kappa)$ (see \eqref{def:Maj}).
Then, with $m$ satisfying \eqref{m:lb} with a sufficiently large absolute constant $C_0$,
\begin{equation}
\pr^\cQ(\YY(t)\ge \yy q) = (1+ o_{\yy;\,n-m\to \infty}(1)) \pr(\tY\ge \yy q). 
\end{equation}
\end{prop}

We will prove Proposition \ref{prop:Yt.approx} after performing some preliminary steps.
In the sequel we denote the Fourier transform of a function $f:\R^d\to \C$ by 
\[
\wh{f}(\bet) = \int_{\R^d} f(\bs{x}) e^{-\ii\bx\cdot \bet}d\bx
\]
and the inverse Fourier transform of a measure $\mu$ on $\R^d$, that is, its characteristic function, by 
\[
\cmu(\bet) = \frac1{(2\pi)^d} \int_{\R^d} e^{\ii \bx\cdot\bet} d\mu(\bx).
\]

\begin{lemma}[Fourier inversion]	\label{lem:fourier}
Let $d\in \N$ and let $\mu,\nu$ be probability measures on $\R^d$.
Let $f\in L^1(\R)$ be an absolutely continuous function whose derivative $f'$ has bounded variation, and let $K_0,K_1$ be constants such that $\|f\|_{L^1(\R)}\le K_0$, $\|f'\|_{\TV}\le K_1$.
Suppose that for some $\ET>0$ and $\delta\in (0,1)$ we have
\begin{equation}	\label{ass:fourier}
\sup_{\|\bet\|_\infty \le \ET} |\wch{\mu}(\bet) - \wch{\nu}(\bet)|\le \delta.
\end{equation}
Put $F=f^{\otimes d}$. 
Then
\begin{equation}	\label{fourier:goal}
\left| \int_{\R^d} F\, d\mu - \int_{\R^d} F\,d\nu\right| \ll_d (K_0K_1)^{d/2} \left( \delta + \frac{1}{\ET} \sqrt{\frac{K_1}{K_0}}\right).
\end{equation}
\end{lemma}

\begin{proof}
We express
\begin{equation}	\label{fourier:express}
\int_{\R^d} F d(\mu-\nu)
= \int_{\R^d} \wh{F}(\bet) (\cmu(\bet)-\cnu(\bet))d\bet.
\end{equation}
First we note that our assumptions on $f$ imply the estimate
\begin{equation}	\label{fourier:pointwise}
|\wh{f}(\eta)| \ll \min\left( K_0, K_1\eta^{-2}\right).
\end{equation}
Indeed, the first bound is immediate from the pointwise bound $|\wh{f}(\eta)|\le \|f\|_{L^1(\R)}$. The second follows from integrating by parts twice and noting that for a smooth function $f$ we have $\|f'\|_{\TV}=\|f''\|_{L^1(\R)}$ (note that we may assume $f$ is smooth by convolving with a mollifier -- since the estimate is independent of the support of the mollifier we can take it to be arbitrarily small). 
Integrating the above pointwise bound yields
\begin{equation}	\label{whf:L1}
\int_{-\infty}^\infty |\wh{f}(\eta)|d\eta \ll \sqrt{K_0K_1}. 
\end{equation}

Let $\omega:\R\to \R$ be a non-negative 
Schwartz function that integrates to $1$ and
whose Fourier transform $\wh{\omega}$ is supported on $[-1,1]$. 
Set $\eps=1/\ET$. 
We partition the identity as a telescoping sum
\[
1 = \prod_{j=1}^d \wh{\omega}(\eps\eta_j) + \sum_{k=1}^d(1-\wh{\omega}(\eps\eta_k)) \prod_{j=1}^{k-1} \wh{\omega}(\eps\eta_j).
\]
Inserting this into the integrand on the right hand side of \eqref{fourier:express} yields
\begin{align*}
&\left| \int_{\R^d} \wh{F}(\bet) (\cmu(\bet)-\cnu(\bet))d\bet\right|\\
&\qquad=\left| \int_{\R^d} \wh{F}(\bet) \!\prod_{j=1}^d \wh{\omega}(\eps\eta_j) (\cmu(\bet)-\cnu(\bet))d\bet 
\!+ \!\sum_{k=1}^d\int_{\R^d} \wh{F}(\bet)(1-\wh{\omega}(\eps\eta_k)) \!\prod_{j=1}^{k-1} \wh{\omega}(\eps\eta_j)(\cmu(\bet)-\cnu(\bet)) d\bet \right|\\
&\qquad\le \int_{\|\bet\|_\infty\le 1/\eps} |\wh{F}(\bet)||\cmu(\bet)-\cnu(\bet)|d\bet  + \sum_{k=1}^d \int_{\R^d} |\wh{F}(\bet)| |1-\wh{\omega}(\eps\eta_k)| |\cmu(\bet)-\cnu(\bet)|d\bet \\
&\qquad\le \delta\int_{\R^d}|\wh{F}(\bet)|d\bet  + \frac{2}{(2\pi)^d}\sum_{k=1}^d \int_{\R^d} \prod_{j=1}^d |\wh{f}(\eta_j)| |1-\wh{\omega}(\eps\eta_k)| d\bet \\
&\qquad= \delta\left(\int_{-\infty}^\infty |\wh{f}(\eta)|d\eta\right)^d  + \frac{2d}{(2\pi)^d}\left(\int_{-\infty}^\infty |\wh{f}(\eta)|d\eta\right)^{d-1}
\int_{-\infty}^\infty |\wh{f}(\eta)| |1-\wh{\omega}(\eps\eta)| d\eta,
\end{align*}
where in the third line we applied our assumption \eqref{ass:fourier} and the fact that $\eps=1/\ET$.
Inserting the estimate \eqref{whf:L1}, the claim will follow if we can show
that
\begin{equation}	\label{fourier:goal1}
\int_{-\infty}^\infty |\wh{f}(\eta)| |1-\wh{\omega}(\eps \eta)|d\eta \ll K_1\eps.
\end{equation}

Note that $|\wh{\omega}(\eta)|\ll 1$ on $\R$ and $|1-\wh{\omega}(\eta)| \ll \eta^2$ on $[-1,1]$ (with implied constants depending on the choice of $\omega$).  
Thus, the left hand side in \eqref{fourier:goal1} is bounded by
\[
\int_{|\eta|\le 1/\eps} |\wh{f}(\eta)|\eps^2\eta^2d\eta + \int_{|\eta|>1/\eps} |\wh{f}(\eta)|d\eta.
\]
Inserting the estimate $|\wh{f}(\eta)|\le K_1\eta^{-2}$ from \eqref{fourier:pointwise} and integrating yields \eqref{fourier:goal1} as desired.
\end{proof}

We will also apply the following normal approximation estimate. 
We denote by $\gamma$ the standard Gaussian measure on $\R$. 

\begin{lemma}[cf.\ {\cite[p.\ 538]{Feller_vol2}}]	\label{lem:clt3}
Let $X_1,\dots, X_q$ be i.i.d.\ centered variables of law $\mu$, with unit variance and finite third moment $m_3:= \e X_1^3$, and let $\nu$ be the law of the normalized sum $(X_1+\cdots+X_q)/\sqrt{q}$. Assume further that the distribution of $X_1$ is non-lattice.
Then
\[
\sup_{x\in \R} \left| \nu((-\infty,x]) - \gamma((-\infty,x]) - \frac{m_3}{6\sqrt{2\pi q}} (1-x^2)e^{-x^2/2}\right| = o_{\mu,q\to\infty}(q^{-1/2}). 
\]
\end{lemma}

The key point is that the above lemma refines the Berry--Ess\'een bound $\nu((-\infty,x]) - \gamma((-\infty,x])  = O(q^{-1/2}\e|X_1|^3)$.

\begin{proof}[Proof of Proposition \ref{prop:Yt.approx}]
Let $\we_\yy,\e_\yy^{(t)}$ denote expectation with respect to the probability 
measures $\P$, $\P^\cQ$ tilted by $e^{\beta\tY}$ and $e^{\beta \YY(t)}$, respectively,
where $\beta$ is chosen so that
\begin{equation}
\label{eq-beta}
\lambda^*(y)=\beta y-\lambda(\beta) \quad (\mbox{\rm and therefore, $\yy=\lambda'(\beta)$}).
\end{equation}
Write $\sigma:=\sqrt{\lambda''(\beta)}$
and define the normalized variables
\begin{equation}
\tW = \frac{\tY - \yy q}{\sigma \sqrt{q}},\qquad W(t) = \frac{\YY(t)-\yy q}{\sigma\sqrt{q}}.
\end{equation}
Let $\mu_\yy,\nu_\yy^t$ denote the laws of $\tW,W(t)$ under $\we_\yy,\e_\yy^{(t)}$, respectively -- that is, for any bounded continuous function $f:\R\to \R$, 
\begin{equation}
\int_\R f(u) d\mu_\yy(u) = \we_\yy f(\tW),
\end{equation}
and similarly for $\nu_\yy^t$. 
We express
\begin{align}
\pr(\tY \ge \yy q)  
&= e^{\lambda(\beta)q} \we_\yy e^{-\beta \tY} \ind(\tY\ge \yy q)	\notag\\
& = e^{(\lambda(\beta) - \beta\yy)q} \we_\yy e^{-\beta \sigma\sqrt{q} \tW} \ind(\tW\ge 0)	\notag\\
&= e^{-\lambda^*(\yy) q} \int_\R f_{\beta \sigma\sqrt{q}} d\mu_\yy 	\label{express:tY}
\end{align}
where we used \eqref{eq-beta} and the notation
\[
f_B(u) := e^{ -B u}1_{[0,\infty)}(u).
\]
For $Y(t)$ we have
\begin{align}
\pr^\cQ(\YY(t)\ge \yy q) 
&=  \Phi_t(\beta) e^{-\beta \yy q} \e_\yy^{(t)} e^{-\beta\sigma\sqrt{q} W(t)}\ind(W(t)\ge 0)  \notag\\
&= \Phi_t(\beta) e^{-\beta \yy q} \int_\R f_{\beta\sigma \sqrt{q}} d\nu_\yy^t.	\label{express:Yt}
\end{align}

In order to compare the last expression with \eqref{express:tY} we use Lemma \ref{lem:fourier}.
To apply the lemma we need to regularize the jump discontinuity in $f$, which we do as follows (note the same regularization was used for a similar purpose in \cite{ABB}). 
Let $\eps>0$ to be chosen sufficiently small depending on $n-m$, and denote 
\begin{equation}	\label{def:fpm}
f^+_{B,\eps} (u)
= \begin{cases}
e^{-B u} & u\ge 0\\
0 & u\le -\eps,
\end{cases}
\qquad
f^-_{B,\eps} (u)
= \begin{cases}
e^{-B u} & u\ge \eps\\
0 & u\le 0
\end{cases},
\end{equation}
with $f^\pm_{B, \eps}$ linearly interpolated on $(-\eps, 0)$ and $(0,\eps)$, respectively.
We will assume 
\begin{equation}	\label{assume:epsB}
\eps \le c/B
\end{equation}
for a sufficiently small absolute constant $c>0$.
For later reference we note that by a straightforward computation,
\begin{equation}	\label{fpm.est}
\|f^\pm_{B,\eps}\|_{L^1(\R)} \ll \frac1B, \qquad \|(f^\pm_{B,\eps})'\|_{\TV} \ll \frac1\eps.
\end{equation}
We compare the inverse Fourier transforms $\wch{\nu_\yy^{t}}$ and $\wch{\mu_\yy}$ using Proposition \ref{prop:Phi}.
First, recalling \eqref{m:lb},
we note that under our assumptions, the error summary parameter $\Upsilon_1$ from that proposition satisfies
\begin{equation}	\label{Ups1.bound}
\Upsilon_1 \ll   (n-m)^{-C_1/4}
\end{equation}
for all $C_1>0$ sufficiently large, if $C_0$ in \eqref{m:lb} is chosen sufficiently large as function of $C_1$. 
Now for $\eta\in \R$,
\begin{align}
\wch{\nu_\yy^t} (\eta)
&= \e_\yy^{(t)} \expo{ \ii\eta W(t)}	\notag\\
&= \expo{-\Lambda_t(\beta) -\frac{\ii \yy \sqrt{q}}{\sigma }\eta} 
\e \expo{ \left(\beta + \frac{\ii\eta}{\sigma \sqrt{q}}\right)\YY(t)}	\notag\\
&= \expo{ -\frac{\ii \yy\sqrt{q}}{\sigma}\eta + \Lambda_t\Big(\beta+ \frac{\ii \eta}{\sigma\sqrt{q}}\Big) - \Lambda_t(\beta)} .\label{nut.check}
\end{align}
By Proposition \ref{prop:Phi} and our assumptions on parameters, 
\begin{align}
\wch{\nu_\yy^t} (\eta)
&= \expo{ 	-\frac{\ii \yy\sqrt{q}}{\sigma}\eta + \left(\lambda\Big(\beta+ \frac{\ii \eta}{\sigma\sqrt{q}}\Big) - \lambda(\beta)\right) q + O\Big( K^*\Big( \beta+ \frac{\ii\eta}{\sigma\sqrt{q}}\Big)\Upsilon_1\Big)}	\notag\\
&= \wch{\mu_\yy}(\eta) \expo{ O\Big( K^*\Big( \beta+ \frac{\ii\eta}{\sigma\sqrt{q}}\Big)\Upsilon_1\Big)}	\\
&= \wch{\mu_\yy} (\eta) \expo{ O\big(e^{O(\beta)} (n-m)^{-C_1/8}\big) }	\qquad \text{for }\; |\eta|\le \sigma\sqrt{q}(n-m)^{c_1C_1}  \label{nus.est}
\end{align}
for a sufficiently small absolute constant $c_1>0$,
where in the last line we used \eqref{Ups1.bound} and the fact that
\[
K^*\Big( \beta + \frac{\ii \eta}{\sigma\sqrt{q}}\Big) \ll e^{O(\beta)} \left( 1+  \frac{|\eta|}{\sigma\sqrt{q}}\right)^{O(1)}.
\]
Thus,
\begin{equation}	\label{nus.diff}
\left| \wch{\nu_\yy^t}(\eta) - \wch{\mu_\yy}(\eta)\right| \ll e^{O(\beta)}(n-m)^{-C_1/8} |\wch{\mu_\yy}(\eta)| \le e^{O(\beta)}(n-m)^{-C_1/8}
\end{equation}
uniformly for $|\eta|\le \sigma\sqrt{q}(n-m)^{c_1C_1}$. 
The above estimate combines with \eqref{fpm.est} and Lemma \ref{lem:fourier} to give
\begin{equation}	\label{fpm.numu}
\left|\int_\R f_{B,\eps}^\pm d(\nu_\yy^t-\mu_\yy)\right| 
\ll e^{O(\beta)} (n-m)^{-c_1'C_1} \left( \frac1{\sqrt{B\eps}}+ \frac{1}{\eps\sqrt{\sigma^2 q}}\right)
\ll \frac{ e^{O(\beta)} (n-m)^{-c_1'C_1} }{ B \eps}
\end{equation}
for an absolute constant $c_1'>0$, where in the second bound we used the definition of $B$ and our assumption \eqref{assume:epsB}.

Now from \eqref{express:Yt} and \eqref{express:tY} we have
\begin{align*}
\pr^\cQ(Y(t)\ge \yy q)
&\ge \Phi_t(\beta)e^{-\beta \yy q} \int_\R f_{B,\eps}^- d\nu_\yy^t\\
&= \Phi_t(\beta) e^{-\beta \yy q} \int_\R f_{B,\eps}^- d\mu_\yy \left( 1+ \frac{\int_\R f_{B,\eps}^- d(\nu_\yy^t-\mu_\yy)}{\int_\R f_{B,\eps}^- d\mu_\yy }\right)\\
&= \pr^\cQ(\tY\ge \yy q) e^{\Lambda_t(\beta)-\lambda(\beta)q} \left( 1- \frac{\int_\R (f_B-f_{B,\eps}^-)d\mu_\yy}{\int_\R f_B d\mu_\yy}\right) \left( 1+ \frac{\int_\R f_{B,\eps}^- d(\nu_\yy^t-\mu_\yy)}{\int_\R f_{B,\eps}^- d\mu_\yy }\right).
\end{align*}
From Proposition \ref{prop:Phi} and \eqref{Ups1.bound}, 
\begin{equation}
\Lambda_t(\beta) - \lambda(\beta)q \ll e^{O(\beta)} \Upsilon_1 \ll e^{O(\beta)} (n-m)^{-C_1/4}.
\end{equation}
Applying Lemma \ref{lem:clt3} with $\mu_\yy$ in place of $\nu$ (with $(\LL_p-y)/\sigma$, under the tilted expectation $\we_\yy$, playing the role of the $X_p$ for $1\le p\le q$), we have
\[
\mu_\yy( (-x, x]) = \gamma((-x,x]) + o_{\yy;\,q\to \infty}(q^{-1/2})
\]
for any $x>0$, where the error term is uniform in the choice of $x$.
Taking $x=2\eps$ and using monotonicity and the boundedness of the Gaussian density, we deduce
\begin{equation}	\label{muy.anti}
\int_\R (f_B-f_{B,\eps}^-)d\mu_\yy \le \int_\R (f^+_{B,\eps}-f^-_{B,\eps}) d\mu_\yy \le \mu_\yy([-\eps,\eps]) \ll \eps + o_{\yy;\,q\to \infty}(q^{-1/2}).
\end{equation}
Note that by our assumptions we have $q\to \infty$ when $n-m\to \infty$. 
We next note the lower bound
\begin{equation}	\label{fB.LB}
\int_\R f_B d\mu_\yy \gg 1/B,
\end{equation}
which follows from (the proof of) the Bahadur--Rao theorem (which uses Lemma \ref{lem:clt3}) -- see \cite[pp. 110--111]{DeZe_book}.
Together with \eqref{muy.anti} this gives
\begin{equation}	\label{fBminus.LB}
\int_\R f_{B,\eps}^-d\mu_\yy \gg 1/B - O(\eps) \gg 1/B,
\end{equation}
recalling our assumption \eqref{assume:epsB}.
Combining the previous seven displays, we have
\begin{align*}
\pr^\cQ(Y(t)\ge \yy q)
&\ge \pr^\cQ(\tY\ge \yy q) \left( 1 - O_\beta\left( (n-m)^{-C_1/4}\right) \right)
\left( 1- O(\eps B)\right) \left( 1- O_\beta\left( \frac1\eps (n-m)^{-c_1'C_1}\right)\right)\\
&= (1+ o_{\yy;\,n-m\to \infty} (1))\pr^\cQ(\tY\ge \yy q),
\end{align*}
where we took 
\begin{equation}
\eps = (n-m)^{-c_0C_1}
\end{equation}
for a suitable constant $c_0>0$. 
Following similar lines with $f_{B,\eps}^-$ replaced by $f_{B,\eps}^+$ we can show a matching upper bound, and the claim follows.
\end{proof}

By following a similar approach with the two-dimensional case of Lemma \ref{lem:fourier} we can show the following:

\begin{prop}	\label{prop:decorr}
Let $\yy, \varrho,Q, \xi_0, C_0, C_1$ and $\kappa$ be as in Proposition \ref{prop:Yt.approx}, and let
$s,t\in \tor \setminus \Maj(\xi_0,\kappa)$ (see \eqref{def:Maj})
Let $\Delta\ge (n-m)^{C_1}$, 
and assume $d_{\xi_0}(s,t)\ge \Delta \frac1{\varrho}e^{-\varrho m}$ (see \eqref{def:dq}).
Then
\begin{equation}	\label{YsYt.asymp}
\pr^\cQ(\YY(s), \YY(t)\ge \yy q ) = (1+o_{\yy;\,n-m\to \infty}(1))\pr(\tY\ge \yy q)^2.
\end{equation}
\end{prop}

We remark that for the proof of Theorem \ref{thm:main} we only need the upper bound in \eqref{YsYt.asymp}. For the proof, we will use Fourier inversion to compare the joint law of $(Y(s),Y(t))$ under a tilted probability measure to the product of the tilted laws of $Y(s),Y(t)$. 

\begin{proof}
Let $\beta$ be as in \eqref{eq-beta}.
Let $\e_\yy^{(s)}, \e_\yy^{(t)}, \e_\yy^{(s,t)}$ denote expectation with respect to the probability measure $\P^\cQ$ tilted by $e^{\beta \YY(s)}, e^{\beta \YY(t)},$ and  $e^{\beta(\YY(s)+  \YY(t))}$, respectively. 
We continue the notation $\we_\yy$, $\sigma= \sqrt{\lambda''(\beta)}$, $\tW, W(s), W(t), \mu_\yy$ and $\nu_\yy^t$ from the proof of Proposition \ref{prop:Yt.approx}, let $\nu_\yy^s$ be defined analogously to $\nu_\yy^t$, and also let $\nu_\yy^{s,t}$ be the law of $(W(s),W(t))\in \R^2$ under $\e_\yy^{(s,t)}$.
By similar lines to the proof of Proposition \ref{prop:Yt.approx}, 
\begin{align}
\pr^\cQ(\YY(s),\YY(t)\ge \yy q) 
&= \Psi_{s,t}(\beta,\beta)e^{-2\beta \yy q} \e_\yy^{(s,t)} e^{-\beta\sigma (W(s) + W(t))} \ind(W(s),W(t)\ge 0)	\notag\\
&= \Psi_{s,t}(\beta,\beta)e^{-2\beta \yy q}\int_{\R^2} f_{B}^{\otimes 2}d\nu_\yy^{s,t}	\notag\\
&= \big(\Psi_{s,t}(\beta,\beta) e^{-2\lambda(\beta)q} \big) e^{-2\lambda^*(\yy)q} \left( \int_{\R} f_B d\mu_\yy\right)^2 \left( \frac{ \int_{\R^2} f_B^{\otimes 2} d\nu_\yy^{s,t}}{ \left( \int_{\R} f_B d\mu_\yy \right)^2}\right) 	\notag\\
&= \pr(\tY\ge \yy q)^2 \big(\Psi_{s,t}(\beta,\beta) e^{-2\lambda(\beta)q} \big)  \frac{ \int_{\R^2} f_B^{\otimes 2} d\nu_\yy^{s,t}}{ \left( \int_{\R} f_B d\mu_\yy \right)^2} ,
\label{YsYt.split}
\end{align}
where $B=\beta\sigma\sqrt{q}$ and $f_B$ and are as in the proof of Proposition \ref{prop:Yt.approx}.

For the middle factor in \eqref{YsYt.split},
by our assumptions and \ref{prop:Psi},
\[
\Psi_{s,t}(\beta,\beta) e^{-2\lambda(\beta)q}  \le \expo{ O( \wt{K}(\beta,\beta)^2 \Upsilon_2)} = \expo{ e^{O(\beta)} \Upsilon_2}.
\]
Note that under our assumptions, the error summary parameter $\Upsilon_2$ from Proposition \ref{prop:Psi} satisfies
\begin{equation}	\label{Ups2.bound}
\Upsilon_2 \ll (n-m)^{-C_1/4}
\end{equation}
for all $C_1>0$ sufficiently large and $C_0$ sufficiently large depending on $C_1$. 
Hence,
\begin{equation}	\label{YsYt.middle}
\Psi_{s,t}(\beta,\beta) e^{-2\lambda(\beta)q} \le \expo{ e^{O(\beta)}  (n-m)^{-C_1/4} } .
\end{equation}

For the last factor in \eqref{YsYt.split}, as in the proof of Proposition \ref{prop:Yt.approx} we let
\begin{equation}	\label{eps2.def}
\eps = (n-m)^{-c_0'C_1}
\end{equation}
for a constant $c_0'>0$ to be chosen later. 
Now letting $f_{B,\eps}^+$ as in \eqref{def:fpm} we have
\begin{equation}	\label{YsYt.upper1}
\frac{ \int_{\R^2} f_B^{\otimes 2} d\nu_\yy^{s,t}}{ \left( \int_{\R} f_B d\mu_\yy \right)^2}
\le 
\frac{ \int_{\R^2} (f_{B,\eps}^+)^{\otimes 2} d\nu_\yy^{s,t}}{ \left( \int_{\R} f_B d\mu_\yy \right)^2} 
= \left( 1+ \frac{\int_\R (f_{B,\eps}^+ - f_B) d\mu_\yy }{ \int_\R f_Bd\mu_\yy} \right)^2 + \frac{\int_{\R^2} (f_{B,\eps}^+)^{\otimes 2} d(\nu_\yy^{s,t} - \mu_\yy^{\otimes 2})}{\left( \int_\R f_B d\mu_\yy\right)^2}.
\end{equation}
Applying \eqref{muy.anti} and \eqref{fB.LB}, 
\begin{equation}	\label{fo2.above}
\frac{ \int_{\R^2} f_B^{\otimes 2} d\nu_\yy^{s,t}}{ \left( \int_{\R} f_B d\mu_\yy \right)^2}
\le \big( 1 + O(B(\eps + o_{\yy;\,q\to\infty}(q^{-1/2})))\big)^2 + O\left( B^2 \int_{\R^2} (f_{B,\eps}^+)^{\otimes 2} d(\nu_\yy^{s,t} - \mu_\yy^{\otimes 2})\right).
\end{equation}
To bound the integral on the right hand side above we use Lemma \ref{lem:fourier}. First we compare the inverse Fourier transforms $\wch{\nu_\yy^{s,t}} $ and  $\wch{\mu_\yy^{\otimes 2}}=\wch{\mu_\yy}^{\otimes 2}$ using Proposition \ref{prop:Psi}.
Following similar lines as in the proof of Proposition \ref{prop:Yt.approx}, from Proposition \ref{prop:Psi} and our assumptions, for any $\eta_1,\eta_2\in \R$,
\begin{align*}
&\wch{\nu_\yy^{s,t}}  (\eta_1,\eta_2)\\
&\quad=\exp\bigg[ -\frac{\ii \yy\sqrt{q}}{\sigma}(\eta_1+\eta_2) + \Lambda_s\Big(\beta+ \frac{\ii \eta_1}{\sigma\sqrt{q}}\Big)+ \Lambda_t\Big(\beta+ \frac{\ii \eta_2}{\sigma\sqrt{q}}\Big) - \Lambda_s(\beta)- \Lambda_t(\beta) \bigg] \\
&\quad= \exp\bigg[ 	-\frac{\ii \yy\sqrt{q}}{\sigma}(\eta_1+\eta_2) + \left(\lambda\Big(\beta+ \frac{\ii \eta_1}{\sigma\sqrt{q}}\Big)+ \lambda\Big(\beta+ \frac{\ii \eta_2}{\sigma\sqrt{q}}\Big) - 2\lambda(\beta)\right) q \\
&\qquad\qquad\qquad\qquad\qquad\qquad\qquad\qquad\qquad\qquad + O\left(\wt{K}\Big( \beta+ \frac{\ii\eta_1}{\sigma\sqrt{q}},\; \beta + \frac{\ii \eta_2}{\sigma\sqrt{q}}\Big)^2 \Upsilon_2\right)   \bigg]\\
&\quad = \wch{\mu_\yy^{\otimes 2}} (\eta_1,\eta_2) 
\expo{ O\left(\wt{K}\Big( \beta+ \frac{\ii\eta_1}{\sigma\sqrt{q}},\; \beta + \frac{\ii \eta_2}{\sigma\sqrt{q}}\Big)^2 \Upsilon_2\right) }.
\end{align*}
Applying \eqref{Ups2.bound} and the fact that
\[
\wt{K}\Big( \beta + \frac{\ii \eta_1}{\sigma\sqrt{q}}, \; \beta+ \frac{\ii\eta_2}{\sigma\sqrt{q}}\Big) \ll e^{O(\beta)} \left( 1+ \frac{|\eta_1|+ |\eta_2|}{\sigma\sqrt{q}}\right)^{O(1)}
\]
we conclude
\begin{equation}	\label{nust.diff}
\left| \wch{\nu_\yy^{s,t}}(\eta_1,\eta_2) - \wch{\mu_\yy^{\otimes 2}}(\eta_1,\eta_2)\right| \ll e^{O(\beta)}(n-m)^{-C_1/8} |\wch{\mu_\yy}(\eta_1)| |\wch{\mu_\yy}(\eta_2)| \le e^{O(\beta)}(n-m)^{-C_1/8}
\end{equation}
uniformly for $|\eta_1|+|\eta_2| \le \sigma\sqrt{q}(n-m)^{c_2C_1}$, where $c_2>0$ is a sufficiently small absolute constant.
Applying Lemma \ref{lem:fourier} with the above estimate along with \eqref{fpm.est}, we have
\begin{align*}
\left|\int_{\R^2} (f^+_{B,\eps})^{\otimes 2} d( \nu_\yy^{s,t} - \mu_\yy^{\otimes 2}) \right|
&\ll \frac{e^{O(\beta)}(n-m)^{-c_2'C_1}}{B\eps}\left(  1+ \sqrt{\frac{B}{\sigma^2q \eps}}\right) \\
&\ll \frac{e^{O(\beta)}(n-m)^{-c_2'C_1}}{B^{1/2}\eps^{3/2}}	
\end{align*}
for some absolute constant $c_2'>0$, where in the second line we used the definition of $B$ and took $C_1$ sufficiently large depending on $c_0'$ that $\eps = o_{ n-m\to \infty} (1/B)$.
Inserting the above estimate in \eqref{fo2.above}, and combining the resulting bound with \eqref{YsYt.middle} and \eqref{YsYt.split} yields
\begin{align*}
\pr^\cQ(\YY(s),\YY(t)\ge \yy q) 
&\le \pr(\tY\ge \yy q)^2 \expo{ e^{O(\beta)}  (n-m)^{-C_1/4} }   \\
&\qquad \;
\left[  \big( 1 + O(B(\eps + o_{\yy;\,q\to\infty}(q^{-1/2})))\big)^2 + O\left( (B/\eps)^{3/2}e^{O(\beta)}(n-m)^{-c_2'C_1}	  \right)  \right] \\
&=  \pr(\tY\ge \yy q)^2  (1+o_{\yy;\, n-m\to \infty}(1))
\end{align*}
for a sufficiently small choice of $c_0'$ in \eqref{eps2.def}.
This establishes the upper bound in \eqref{YsYt.asymp}.
The lower bound is established by similar lines, replacing $f_B$ with $f_{B,\eps}^-$ instead of $f_{B,\eps}^+$ in \eqref{YsYt.upper1}.
\end{proof}

\section{Proof of Theorem \ref{thm:main} (upper bound)}	
\label{sec:upper}

In this section we prove the upper bound in Theorem \ref{thm:main}. 
Specifically, we establish the following:

\begin{theorem}	\label{thm:upper}
For any $\epp>0$, 
\begin{equation}
\sup_{t\in \tor} X_N(t) \le (\xcrit+\epp)\log N\qquad \text{ with probability $1-o_\epp(1)$.}
\end{equation}
\end{theorem}

We now begin the proof of Theorem \ref{thm:upper} with some preliminary reductions.
Fix $\epp>0$. We may assume $\epp\in (0,1)$.
As a first step we pass from a supremum over the continuum $\tor$ to a maximum over a sufficiently dense net, via the following lemma from \cite{CMN}. 
Here and in the sequel, for $q\in \N$ we write
\begin{equation}	\label{def:torq}
\torr_{q}=\{j/q: 0\le j\le q-1\}.
\end{equation}

\begin{lemma}[{\cite[Lemma 4.3]{CMN}}]		\label{lem:CMN}
For any polynomial $\psi$ of degree at most $N_0\ge 1$, one has
\[
\sup_{t\in \tor} |\psi(e(t))| \le 14 \sup_{t\in \torr_{2N_0}}|\psi(e(t))|.
\]
\end{lemma}

From the above, since $X_N$ is the logarithm of the modulus of a polynomial of degree $N$, then for any $N_0\ge N$,
\[
\sup_{t\in \tor} X_N(t)\le \sup_{t\in \torr_{2N_0}}X_N(t) + O(1).
\]
Hence,
letting
\begin{equation}	\label{upper.N0}
N\le N_0 \ll N,
\end{equation}
it is enough to show 
\begin{equation}	\label{upper.goal1}
\sup_{t\in \torr_{2N_0}} X_N(t) \le (\xcrit+\epp)\log N	\qquad \text{ with probability $1-o_\epp(1)$}.
\end{equation}

Our second step is to truncate the sum in \eqref{XN.cyc} so that we may pass to consideration of the Poisson field $Y_N$ via Theorem \ref{thm:arta}. 
Let $W$ be a slowly growing parameter:
\begin{equation}	\label{upper.W}
\omega(1)\le W\le N^{o(1)}
\end{equation}
and decompose $X_N$ as in \eqref{XNW}.
We easily obtain uniform control from above on the high frequency tail $X_{N}^>$ via a second moment argument:

\begin{lemma}	\label{lem:Xtail.ub}
Let $2\le W\le N$. Then
\[
\pro{ \sup_{t\in\tor} X_{N}^>(t)\ge 2\log W} \ll \frac1{\log W}. 
\]
\end{lemma}

\begin{proof}
Denote
\[
\cN_{N,W}^>:= \sum_{N/W<\ell\le N} C_\ell(P_N).
\]
Since $\log|1-e(u)|\le \log 2<1$, $X_{N}^>(t)$ is bounded pointwise by $\cN_{N,W}^>$. 
We have
\[
\e \cN_{N,W}^> = \sum_{N/W<\ell\le N} \frac1\ell = \log W + O(1).
\]
Moreover, by an elementary computation we have $\e C_\ell(P_N)^2 = \frac1\ell + \frac1{\ell^2}$ and $\e C_k(P_N)C_\ell(P_N) = 1/(k \ell)$ for  $k\ne \ell$, from which it follows follows that 
\[
\Var(\cN_{N,W})= \e \cN_{N,W}^>.
\]
By Chebyshev's inequality,
\[
\pro{ \sup_{t\in \tor} X_{N}^>(t) \ge 2\log W}\le \pro{ \cN_{N,W}^>\ge 2\log W} \ll \frac1{\log W} .\qedhere
\]
\end{proof}

From the above lemma and the upper bound in \eqref{upper.W}, in order to show \eqref{upper.goal1} it is enough to show 
\[	
\sup_{t\in \torr_{2N_0}} X_{N}^{\le}(t) \le (\xcrit+\epp) \log N \qquad \text{ with probability $1-o_\epp(1)$.}
\]
Now from the lower bound in \eqref{upper.W} and Theorem \ref{thm:arta}, it suffices to show 
\begin{equation}	\label{upper.goal2}
\sup_{t\in \torr_{2N_0}} Y_{N/W}(t) \le (\xcrit+\epp) \log N \qquad \text{ with probability $1-o_\epp(1)$.}
\end{equation}

We will establish \eqref{upper.goal2} by controlling the maximum on complementary subsets of the unit circle -- informally called the ``major arcs" and ``minor arcs" -- by different arguments (see Section \ref{sec:overview.arcs} for a high-level discussion discussion). 
The following provides uniform control for $Y_{N/W}$ on major arcs. 
Recall our notation \eqref{def:bohr}. 

\begin{prop}[Upper bound for the Poisson field on major arcs]	\label{prop:upper.major}
Let $\alpha\in (0,1)$ and $\xi_0\ge1$ such that 
\begin{equation}	\label{assume:xi0alpha}
\xi_0\le C\alpha^2\log N 
\end{equation}
for a sufficiently large constant $C>0$.
Then except with probability $O(\xi_0 N^{-c\alpha/\xi_0})$, 
\begin{equation}
\sup_{t\in \Maj( \xi_0, N^{-\alpha}) } Y _N(t) \le  -\frac{\alpha^2\log^2N}{128\xi_0}.
\end{equation}
\end{prop}

(Note that we do not need the reduction to the net $\torr_{2N_0}$ for the major arc case.)
We defer the proof of Proposition \ref{prop:upper.major} to Section \ref{sec:major}. 

For the control on minor arcs it will be convenient to group terms as in \eqref{def:micro}--\eqref{2goal.upper4}.
Fix a constant $c_0>0$ to be chosen sufficiently small, and let $\varrho(N)\in (0,1)$, $n(N)\in \N$ be sequences such that
\begin{equation}	\label{assume:rhon.upper}
\varrho = (1+o(1))c_0\epp, \qquad \varrho n = (1+o(1))\log N.
\end{equation}
Note that by our assumption \eqref{upper.W} on $W$ the above allow $\varrho, n$ to be taken so that $Y_{[1,e^{\varrho n})} = Y_{N/W}$. 

\begin{prop}[Upper bound for the Poisson field on minor arcs]	\label{prop:upper.minor}
Let $\xi_0\ge C/\epp^4$ and $e^{-c\epp^3n}\le \kappa<1$ for constants $C,c>0$ sufficiently large and small, respectively. 
For $\varrho, n$ satisfying \eqref{assume:rhon.upper}, if $T\subset\tor$ with $|T|\le e^{(1+c_0\epp )\varrho n}$ for a sufficiently small constant $c_0>0$, then
\begin{equation}
\pro{\max_{t\in T\setminus \Maj(\xi_0, \kappa)}\YY_{[1,e^{\varrho n})}(t) \ge (\xcrit+\epp )\varrho n} \ll_{\epp } \expo{-\epp \varrho n} .
\end{equation} 
\end{prop}

We defer the proof of Proposition \ref{prop:upper.minor} to Section \ref{sec:minor}, and complete the proof of Theorem \ref{thm:upper} on the above propositions. 
We take $\alpha = c\epp^2$ and $\xi_0= C/\epp^4$ for suitable constants $c,C>0$. We take $\varrho , n$ such that $Y_{[1,e^{\varrho n})} = Y_{N/W}$, and put 
\[
\kappa = (N/W)^{-\alpha} = e^{-c\epp^2\varrho n} \ge e^{-c'\epp^3n}.
\]
From Proposition \ref{prop:upper.major} we have that with probability $1-O_\epp(N^{-c'\epp^6})$, 
\[
\sup_{t\in \Maj(\xi_0,\kappa)} Y_{N/W}(t) \le 0
\]
(say). 
From Proposition \ref{prop:upper.minor} with $T=\torr_{2N_0}$ (which from \eqref{upper.N0} has size at most $e^{(1+c_0\epp)\varrho n}$ for all $n$ sufficiently large depending on $\epp$), with probability $1-O_\epp(N^{-c'\epp })$, 
\[
\sup_{t\in \torr_{2N_0}\setminus \Maj(\xi_0,\kappa)} Y_{N/W}(t) \le (\xcrit+2\epp)\log N.
\]
Now \eqref{upper.goal2}, and hence Theorem \ref{thm:upper}, follows from the previous two displays and the union bound.

\subsection{Control on major arcs}	\label{sec:major}

In this subsection we prove Proposition \ref{prop:upper.major}.
The main idea is that for points $t$ in major arcs, the sequence $\ell t\in \tor$ returns often to the vicinity of the singularity of $\log|1-e(\,\cdot\,)|$, which (crucially) points in the direction of $-\infty$. 

We begin with a crude upper bound on the maximum of the Poisson field.

\begin{lemma}	\label{lem:upper.crude}
For any $M\ge 1$,
\[
\sup_{t\in \tor} Y_M(t) \le 3\log M \qquad \text{ with probability $1-O(M^{-c})$}. 
\]
\end{lemma}

\begin{proof}
We note the following standard concentration bound (easily established by bounding the moment generating function): if $Z$ is Poisson-distributed with expectation $\rho>0$, then
\begin{equation}	\label{poi.conc}
\pr(Z\ge 2\rho)\,, \; \pr(Z\le \rho/2) \ll e^{-c\rho}.
\end{equation}
The claim follows from the above the upper tail estimate above and the pointwise bound $Y_{M}(t) \le \sum_{\ell\le M}Z_\ell\eqd \Poi(\rho)$ with $\rho = \log M+O(1)$.
\end{proof}

The following shows that the Poisson field $Y_N$ is in fact very \emph{negative} on individual low frequency Bohr sets.

\begin{lemma}[Control of the maximum over Bohr sets]	\label{lem:bad}
Let $N,\xi\ge 1$, $\alpha\in (0,1)$ and 
$0<M\le \xi N^{\alpha/4}$.
There is an event $\cB(\xi,\alpha)$ with 
\begin{equation}	\label{bd:cB}
\pro{ \cB(\xi,\alpha)} \ll N^{-c\alpha/\xi}
\end{equation}
such that on $\cB(\xi,\alpha)^c$,
\[
\sup_{t\in B_{\xi}(N^{-\alpha})} \YY_{(M,N]}(t) \le -\frac{\alpha^2}{64\xi}\log^2N + 4\log N .
\]
\end{lemma}

\begin{proof}
Let
\[
\YY'(t) = \sum_{\substack{M<\ell\le \xi N^{\alpha/2}\\ \xi |\ell}} Z_\ell\log|1-e(\ell t)|, \qquad \YY''(t) = \YY_{(M,N]}(t) - \YY'(t)
\]
and
\[
\cN = \sum_{\ell\le N} Z_\ell, \qquad \cN' = \sum_{\substack{M<\ell\le \xi N^{\alpha/2}\\ \xi |\ell}} Z_\ell.
\]
Then $\cN$ and $\cN'$ are Poisson random variables with means 
\[
\rho =\e \cN = \log N + O(1)
\]
and
\begin{align*}
 \rho' &= \e \cN' =\frac1{\xi }\left(\log(N^{\alpha/2}) - \log\left(\frac{M}\xi\vee 1\right)\right) + O\left(\frac1{\xi\wedge M}\right) \\
 &\ge \frac1\xi\left(\frac{\alpha}{4}\log N - O(1)\right) \ge \frac{\alpha}{8\xi}\log N. 
\end{align*}
(For the last bound, note we are free to assume $\frac{\alpha}\xi\log N$ is larger than any fixed constant as the claim is trivial otherwise.)
Let $\cB(\xi,\alpha)$ be the event that either $\cN> 2\log N$ or $\cN'< \frac{\alpha}{16{\xi }}\log N$.
Then \eqref{bd:cB} follows from \eqref{poi.conc} and the union bound.

For the remainder of the proof we restrict to $\cB(\xi,\alpha)^c$.
Let $t\in B_{ \xi }(N^{-\alpha})$ be arbitrary, and let $a\in \Z$ and $\theta$ be such that $t=(a+\theta)/\xi$, with $|\theta|\le N^{-\alpha}$. 
Since $\log|1-e(\cdot)| \le 2$ pointwise, we have
\[
\YY''(t) \le 2 \cN\le 4\log N.
\]
Thus,
\begin{align*}
\YY'(t) &= \sum_{\frac{M}\xi<\ell\le N^{\alpha/2}} Z_{\ell\xi}\log|1-e(\ell\xi t)|\\
&= \sum_{\frac{M}\xi<\ell\le N^{\alpha/2}} Z_{\ell\xi}\log|1-e(\ell\theta)|\\
&\le  \sum_{\frac{M}\xi<\ell\le N^{\alpha/2}} Z_{\ell\xi}\log(2\|\ell\theta\|_\tor)\\
&\le  \sum_{\frac{M}\xi<\ell\le N^{\alpha/2}} Z_{\ell\xi}\log(2N^{-\alpha/2})\\
&= \cN' \log(2N^{-\alpha/2})\\
& \le -\frac{\alpha^2}{64\xi} \log^2N ,
\end{align*}
where in the final bound we assumed $\alpha\log N\ge 4\log 2$.
\end{proof}

Now we conclude the proof of Proposition \ref{prop:upper.major}.
Applying Lemma \ref{lem:upper.crude} and Lemma \ref{lem:bad} with $M=N^{\alpha/4}$, along with the union bound,
\[
\sup_{t\in \Maj(\xi_0, N^{-\alpha})} Y_N(t) \le \frac{-\alpha^2}{64\xi_0} \log^2 N + \left( 4+ \frac{\alpha}4\right)\log N \le \frac{-\alpha^2}{64\xi_0} \log^2 N + 5\log N
\]
except with probability $O(\xi_0N^{-c\alpha/\xi_0} + N^{-c\alpha}) = O(\xi_0N^{-c\alpha/\xi_0} )$. 
The claim now follows by taking the constant $C$ in \eqref{assume:xi0alpha} sufficiently large.

\subsection{Control on minor arcs}	\label{sec:minor}

In this subsection we prove Proposition \ref{prop:upper.minor}.
 
Let $\epp\in (0,1)$.
Up to adjusting the constant $c_1$ we are free to assume $\epp$ is smaller than any fixed constant. 
We may also assume $n$ is sufficiently large depending on $\epp$.

Recall our notation \eqref{def:Qa}.
Let $c_1,c_2>0$ be sufficiently small constants, 
put $m=\lf c_1\epp n\rf$, and
let $\cG$ be the event that
\begin{equation}	\label{Q1Q2}
\big||\urn_1\cap [m,n)| -\varrho (n-m)\big|, \; |\urn_{\ge 2}\cap [m,n)|\;\le c_2\epp\varrho (n-m).
\end{equation}
From our assumption \eqref{assume:rhon.upper} we have that \eqref{assume:typical} holds for any fixed constant $C>0$ and all $n$ sufficiently large. 
From Lemma \ref{lem:typical} we have 
\begin{equation}	\label{Geps2}
\pr(\cG) = 1- O(e^{ -c\epp\varrho^2(n-m)})  = 1-O(e^{-c\epp^3n}).
\end{equation}

Let $\xi_0\ge1$ and $\kappa\in (0,1)$ to be chosen later. 
Fix arbitrary disjoint sets $\urn, \turn\subset[m,n)\cap \Z$ satisfying the bounds for $\urn_1\cap [m,n)$ and $\urn_{\ge2}\cap [m,n)$, respectively, in \eqref{Q1Q2}. 
From Corollary \ref{cor:Yt.upper}, and assuming $\epp$, and hence $\varrho$, is at most a sufficiently small constant, then
for any $t\in \tor\setminus \Maj(\xi_0,\kappa)$,
\begin{align*}
&\pr \left( \YY_{[e^{\varrho m}, e^{\varrho n})}(t) \ge (\xcrit+\epp/2) \varrho (n-m) \,\middle|\, \cQ(m,n,\urn,\turn) \right) \\
&\qquad\qquad \le \pr \left( \YY_{[e^{\varrho m}, e^{\varrho n})}(t) \ge \frac{\xcrit+\epp/2}{1+c_2\epp} |\urn| \,\middle|\, \cQ(m,n,\urn,\turn) \right) \\
&\qquad\qquad \le \expo{ -\lambda^*\left( \frac{\xcrit+\epp/2}{1+c_2\epp}\right) (1-c_2\epp) \varrho (n-m) + O\left( \epp \varrho n + \Err'_0\right)},
\end{align*}
where 
\[
\Err'_0=\frac{n}{\xi_0^{1/2}} + \frac1\varrho \left( 1+ \log\frac1{\varrho \kappa}\right)
\]
and in the final line we took $\epp$ smaller than the constant $\eps_0$ in \eqref{beta.bounds}.
Hence,
\begin{equation}
\pr \left( \YY_{[e^{\varrho m}, e^{\varrho n})}(t) \ge (\xcrit+\epp/2) \varrho (n-m) \,\middle|\, \cQ(m,n,\urn,\turn) \right)
\le  \expo{ - (1+c\epp)\varrho (n-m) + O(\Err'_0)}
\end{equation}
where we have taken the constant $c_2$ sufficiently small and used that $\lambda^*$ is strictly increasing on $\R_+$ and equal to 1 at $\xcrit$.
Averaging over the choices of $\urn, \turn$, on $\cG$ we have
\begin{equation}	\label{onGe1}
\pr\left(\cG\cap \Big\{\, \YY_{[e^{\varrho m}, e^{\varrho n})}(t) \ge (\xcrit+\epp/2) \varrho (n-m) \,\Big\} \right)
\le \expo{ - (1+c\epp)\varrho (n-m) + O(\Err'_0)}
\end{equation}
for any $t\in \tor\setminus \Maj(\xi_0, \kappa)$. 

Let
\begin{equation}
\cA_m= \left\{ \sup_{t\in \tor} Y_{[1,e^{\varrho m})}(t) \le 4\varrho m\right\}.
\end{equation}
From Lemma \ref{lem:upper.crude},
\begin{equation}	\label{Amc}
\pro{\cA_m^c }  \ll e^{-c\varrho m} \ll e^{-c\epp^2n}.
\end{equation}
Applying the union bound,
\begin{align}
&\pr\Big(\cA_m\cap \cG\cap \Big\{\, \exists t\in T\setminus \Maj(\xi_0,\kappa): \YY_{[1,e^{\varrho n})}(t) \ge (\xcrit+\epp)\varrho n \,\Big\}\Big)	\notag\\
&\qquad\qquad\qquad\qquad\le 
\sum_{t\in T\setminus \Maj(\xi_0,\kappa)} 
\pro{ \cG\cap \Big\{\, \YY_{[e^{\varrho m}, e^{\varrho n})}(t) \ge (\xcrit+\epp-4c_1\epp)\varrho (n-m) \,\Big\}}	\notag\\
&\qquad\qquad\qquad\qquad\le 
|T|\expo{ - (1+c\epp)\varrho (n-m) + O(\Err'_0)}	\notag\\
&\qquad\qquad\qquad\qquad\le 
|T|\expo{ - (1+c\epp/2)\varrho n + O(\Err'_0)}	\notag\\
&\qquad\qquad\qquad\qquad \le \expo{ -c\epp\varrho n/4 + O(\Err'_0)}	\label{onAonG}
\end{align}
where in the third line we applied \eqref{onGe1}, in the third and fourth lines we took the constant $c_1$ sufficiently small, and in the final line we applied the assumption on $|T|$ with $c_0$ sufficiently small. 
Combining \eqref{Geps2}, \eqref{Amc} and \eqref{onAonG} we conclude
\begin{align*}
\pr\Big( \exists t\in T\setminus \Maj(\xi_0,\kappa): \YY_{[1,e^{\varrho n})}(t) \ge (\xcrit+\epp)\varrho n\Big) 
&\ll e^{-c\epp^3 n} + e^{-c\epp^2 n}   + \expo{ -c\epp \varrho n/4 + O(\Err'_0)}\\
&\ll  \expo{ -c'\epp\varrho n + O(\Err'_0)}.
\end{align*}
The claim now follows from our assumptions on $\xi_0$ and $\kappa$ and the definition of $\Err'_0$ (recalling that $\varrho\gg \epp$).

\section{Lower bound: Early and middle generations}	
\label{sec:lower}

In this and the following sections we complete the proof of Theorem \ref{thm:main} by establishing the following complement to Theorem \ref{thm:upper}.
For high-level motivation of ideas and some notation we refer to Section \ref{sec:overview.lower}.

\begin{theorem}	\label{thm:lower}
For any $\epp>0$, 
\begin{equation}
\sup_{t\in \tor} X_N(t) \ge (\xcrit-\epp)\log N \qquad \text{ with probability $1-o_\epp(1)$.}
\end{equation}
\end{theorem}

For this section it will be convenient to work with slight rotations of the nets defined in \eqref{def:torq}. For $\theta\in \R$ and $q\in \N$ we denote
\begin{equation}	\label{def:torq2}
\torr_{q,\theta} = \frac{\theta}{q^2} +\torr_q = \left\{ \frac{j}q+ \frac{\theta}{q^2}: 0\le j\le q-1\right\}.
\end{equation}
(Thus, $\torr_{q,0}=\torr_q$.)

Our goal in this section is to prove the following proposition, which treats the Poisson field. From Theorem \ref{thm:arta} this implies a similar result for the truncated field $X_{N}^\le$ from \eqref{XNW} (see Corollary \ref{cor:early.middle}). We defer treatment of the high frequency tail $X_{N}^>$ to the next section. 
Recall the notation \eqref{def:superY} for the super-level sets of the Poisson field $Y_N$.

\begin{prop}[Many high points for the Poisson field]	\label{prop:poi.highpts}
Let $N^{-50}\le \theta \le 1$, $N_0= N^{1+o(1)}$, and $T_0\subset \torr_{N_0,\theta}$ with $|T_0|\ge (1-o(1))N_0$. 
For all $\epp\in (0,1)$ there exists $c(\epp)>0$ such that, with probability $1-o_\epp(1)$, 
\begin{equation}
\big| \super^Y_N(T_0,\xcrit-\epp)  \big| \ge N^{c(\epp)}.
\end{equation}
\end{prop}

For the proof we take sequences $\varrho=\varrho(N)\in (0,1)$ and $n=n(N)\in \N$ such that
\begin{equation}	\label{assume:rhon.lower}
\varrho = (1+o(1))n^{-2/3}
\qquad \text{ and } \qquad \varrho n = (1+o(1))\log N.
\end{equation}
Note that here we take $\varrho=o(1)$, in contrast to \eqref{assume:rhon.upper} in the proof of the upper bound.
The reason for taking $\varrho$ so small is to ensure that with high probability, $\cN(I_\gen)\le 1$  for all $\gen$ in a certain range. This allows us to restrict to an event $\cQ$ as in \eqref{def:Qevent} with $\turn=\emptyset$, which then gives access to the stronger tail comparison results from Section \ref{sec:FL}.
The specific choice of $-2/3$ for the exponent of $\varrho$ is not important -- any fixed constant in $(-1,-1/2)$ would do just as well.

It will be convenient to parametrize the sets of survivors $\cS^Y$ with $n$ rather than $N$. 
For $n\ge1$, $T\subset \tor$ and $\yy\in \R$ we abbreviate
\begin{equation}	\label{def:SY}
\duper_n(T, \yy):= \super^Y_{\lf e^{\varrho n}\rf}(T,\yy)= \{t\in T: Y_{[1,e^{\varrho n})}(t) \ge \yy\varrho n\}.
\end{equation}

In Section \ref{sec:early} we prove the following.

\begin{prop}[Early generations]	\label{prop:early}
For any $m\ge1$, $N_0\ge e^{3\varrho m}$ and $\theta\in \R$ with $N_0^{-100}\le |\theta|\le 1$, we have
\begin{equation}
\left|\duper_m(\torr_{N_0,\theta}, -2) \right| \gg N_0\qquad \text{ with probability $1-O(e^{-c\varrho m})$},
\end{equation}
where $c>0$ is a sufficiently small absolute constant. 
\end{prop}

The core of the proof of Proposition \ref{prop:poi.highpts}
is a second moment argument, which is encapsulated by Proposition \ref{prop:rapid} below.
See Section \ref{sec:overview} for a high-level motivation of the ideas.

First we need to set up some notation.
Recall the notation \eqref{def:micro}. Let $\Eons \in \N$ be a large integer (which we will later take to infinity) and let
\begin{equation}
J_\eon= \bigcup_{\frac{\eon-1}{\Eons }n \le \gen< \frac{\eon}\Eons n} I_\gen,
\end{equation}
so that
\[
\YY_{[1,e^{\varrho n})}(t) = \sum_{1\le \gen<n} Y_{I_\gen}(t) = \sum_{\eon=1}^\Eons  \YY_{J_\eon}(t).
\]
We consider nonempty sets $\urn^{\eon}\subset [(\eon-1)n/\Eons, \eon n/\Eons)$ for $2\le \eon\le \Eons$ and denote $\urn=\bigcup_{2\le \eon \le \Eons} \urn^{\eon}$.
We write $q_\eon:=|\urn^{\eon}|$. 
Recalling our notation \eqref{def:Qa}--\eqref{def:eHK}, we denote the events
\begin{equation}	\label{def:good.urn}
\cQ_\eon := \cQ\left( \frac{\eon-1}{\Eons}n, \frac{\eon}{\Eons}n, \urn^{\eon}, \emptyset\right),\qquad \cQ^*:= \bigcap_{2\le \eon\le \Eons} \cQ_\eon
\end{equation}
and let $\pr^{\cQ^*}, \e^{\cQ^*}$ and $\Var^{\cQ^*}$ denote probability, expectation and variance conditional on $\cQ^*$. 
For a finite set $T\subset\tor$ and $\yy>0$ we define the associated (random) set of ``rapid" points
\begin{equation}
\cR_n(T,\yy)=\cR_n(T,\yy; \Eons, \urn) = \big\{ t\in T: \YY_{J_\eon}(t) \ge \yy q_\eon \;\; \forall \eon\in [2,\Eons]\big\}.
\end{equation}
Note that for $\epp>0$ small,  $\cR_n(T,\xcrit-\epp)$ is the set of points in $T$ at which the sequence $\{\YY_{J_\eon}(t)\}_{2\le \eon\le \Eons}$ rises at the near-maximum rate allowed by Proposition \ref{prop:upper.minor}. 

In Section \ref{sec:middle} we prove the following.

\begin{prop}[Middle generations]	\label{prop:rapid}
Let $\epp \in (0,1/2)$ and $\Eons\in \N$ sufficiently large. 
Let $\varrho, n$ be as in \eqref{assume:rhon.lower} and
assume 
\begin{equation}	\label{assume:qell}
q_\eon = (1+o(1)) \varrho  n/\Eons, \quad 2\le \eon \le \Eons.
\end{equation}
Let $N_0\in \N$ with $N_0=N^{1+o(1)}$. 
For any $\theta\in \R$ and any $T\subset \torr_{N,\theta}$ with $|T|=N^{1+o(1)}$, we have
\begin{equation}	\label{manyrapid}
\pr^{\cQ^*} \big(|\cR_n(T,\xcrit-\epp)| \ge e^{c(\epp)\varrho n} \big) = 1- o_{\epp,\Eons}(1).
\end{equation}

\end{prop}

\begin{remark}
The proof shows we can take $c(\epp)$ to be any constant strictly smaller than $1-\lambda^*(\xcrit-\epp))$.
\end{remark}

Now we conclude the proof of Proposition \ref{prop:poi.highpts} on Propositions \ref{prop:early} and \ref{prop:rapid}.
We may select $\varrho, n$ satisfying \eqref{assume:rhon.lower} such that $[1,e^{\varrho n})\cap \Z = [N]$. 
Now it suffices to show
\begin{equation}	\label{highpts.goal1}
|\duper_n(T_0, (\xcrit-\epp)\varrho n)| \ge \expo{ c(\epp) \varrho n} \qquad \text{ with probability $1-o_\epp(1)$}.
\end{equation}
Let $\Eons\ge 1$ to be taken sufficiently large depending on $\epp$, and set 
\[
m= \lf n/\Eons\rf, \qquad T_1= \duper_{m}(T_0, -2).
\]
From Proposition \ref{prop:early} and our assumption $|T_0|\ge (1-o(1))N_0$ we have
\begin{equation}
|T_1|\gg N_0 \qquad \text{ with probability $1-O(e^{-c\varrho n/\Eons})$}. 
\end{equation}
Let us condition on a realization of the Poisson variables $(Z_\ell)_{\ell<n/\Eons}$ such that the above holds, thus fixing the set $T_1$. 
By independence this does not affect the distribution of the variables $(Z_\ell)_{\ell\ge n/\Eons}$.

Recall the notation \eqref{def:Qa}. For $2\le \eon \le \Eons$ denote
\[
Q_1^{\eon} = Q_1\cap \left[ \frac{\eon - 1}\Eons n, \frac{\eon}\Eons n\right).
\]
For $\eps\in (0,1)$ let $\cG_\Eons (\eps)$ be the event that 
\begin{equation}	\label{GLeps1}
Q_{\ge2}\cap [n/\Eons, n) = \emptyset 
\end{equation}
and
\begin{equation}	\label{GLeps2}
\left| |Q_1^\eon| -\frac{\varrho n}{\Eons } \right| \le \eps \frac{\varrho n}{\Eons } \quad \forall 2\le \eon\le \Eons .
\end{equation}
We now argue $\cG_\Eons (\eps)$ is a likely event using the estimates in Lemma \ref{lem:typical}. 
First note that the assumption \eqref{assume:typical} amount to assuming $\varrho \ge e^{-c\varrho n/\Eons}$ for a sufficiently small constant $c>0$, which holds for all $n$ sufficiently large depending on $\Eons$ by our assumption \eqref{assume:rhon.lower}.
From \eqref{EQ2}, Markov's inequality and our assumption \eqref{assume:rhon.lower}, \eqref{GLeps1} holds with probability $1-O(n^{-1/3})$. 
Conditional on the event that \eqref{GLeps1} holds, we have 
$
|Q_1^\eon| = \cN(J_\eon)
$ (recall the notation \eqref{def:NJ}), and so $\cG_\Eons (\eps)$ is contained in the event that \eqref{GLeps1} holds and $\cN(J_\eon) = (1+O(\eps))\varrho n/\Eons$ for $2\le \eon\le \Eons $. 
The latter event holds with probability $1-O(\Eons  \exp(-c\eps^2\varrho n/\Eons))$ by \eqref{S.conc}, so by the union bound,
\[
\pr(\cG_\Eons (\eps)) = 1-O(n^{-1/3}) -O(\Eons  \exp(-c\eps^2\varrho n/\Eons)).
\]
Taking $\eps=n^{-1/10}$ and recalling $\varrho = (1+o(1))n^{-2/3}$, we conclude
\begin{equation}	\label{goodL}
\pr(\cG_\Eons (n^{-1/10})) = 1-o_{\Eons }(1).
\end{equation}

Now we condition on a realization of the sets $Q_1^\eon$ satisfying the conditions of $\cG_\Eons $ -- that is, we condition on the event $\cQ^*$ from \eqref{def:good.urn} with $Q^\eon := Q_1^\eon$ for $2\le \eon\le \Eons $. 
Under this conditioning, note that for any $t\in \cR_n(T_1,\xcrit-\epp)$, 
\[
Y_{[1,e^{\varrho n})}(t) \ge -2\frac{\varrho n}{\Eons } + (\xcrit-\epp)\sum_{\eon=2}^\Eons  q_\eon
\ge \varrho n\left( \xcrit-\epp - o(1)- O(1/\Eons )\right).
\]
Thus, fixing $\Eons =\Eons (\eps)$ sufficiently large, we have that on $\cG_\Eons $,
\begin{equation}	\label{RS.contained}
\cR_n(T_1,\xcrit-\epp) \subseteq \duper_n(T_0, \xcrit-2\epp)
\end{equation}
for all $n$ sufficiently large.
From Proposition \ref{prop:rapid},
\begin{equation*}	
\pr^{\cQ^*} \big( \cG_\Eons  \cap \big\{ |\cR_n(T_1,\xcrit-\epp)| \ge e^{c(\epp)\varrho n} \big\}\big) = 1- o_{\epp}(1).
\end{equation*}
Together with \eqref{goodL} this implies
\begin{equation}
|\cR_n(T_1,\xcrit-\epp)| \ge e^{c(\epp)\varrho n} \qquad \text{ with probability $1-o_{\epp}(1)$.}
\end{equation}
\eqref{highpts.goal1} now follows from the above and \eqref{RS.contained} (and replacing $\epp$ with $\epp/2$).
This concludes the proof of Proposition \ref{prop:poi.highpts} on Propositions \ref{prop:early} and \ref{prop:rapid}.

\subsection{Early generations: Many survivors}	
\label{sec:early}

In this section we prove Proposition \ref{prop:early}.
First we need two lemmas.

\begin{lemma}	\label{lem:easyLB}
Let $T$ be a finite set and $y:T\to \R$. Suppose that for some $\xx,\delta>0$, 
\begin{equation}
y(t)\le y_0\quad \forall t\in T
\end{equation}
and
\[
\frac{1}{|T|}\sum_{t\in T} y(t) \ge -\delta.
\]
Let $T_\ge = \{t\in T: y(t)\ge -2\delta \}$. Then
\begin{equation}
|T_\ge |\ge \frac{\delta|T|}{y_0+2\delta} .
\end{equation}
\end{lemma}

\begin{proof}
We have
\begin{align*}
-\delta |T| 
&\le \sum_{t\in T_\ge}y(t) + \sum_{t\in T\setminus T_\ge }y(t)\le y_0|T_\ge| -2\delta(|T|-|T_{\ge}|)
\end{align*}
and the result follows from rearranging.
\end{proof}

Thus, to establish Proposition \ref{prop:early} we can combine a uniform upper bound on $\YY_{[1,e^{\varrho m})}(t)$ 
with a lower bound on the average of $\YY_{[1,e^{\varrho m})}$ over $\torr_{N_0,\theta}$. 
The former is provided by Lemma \ref{lem:upper.crude}, while the latter can be obtained from the following.

\begin{lemma}	\label{lem:avgFell}
For $\ell\in \N$ write
\[
F_\ell(t) = \log|1-e(\ell t)|.
\]
Let $0\le a<\ell$ and $M\ge 2\ell$ be integers and put $I=[\frac{a}\ell, \frac{a+1}\ell)$. 
Let $\theta\in [-1,1]$ with $|\theta|\ge M^{-b}$ for some $b>0$. 
Then
\begin{equation}	\label{bd:avgFell}
\sum_{t\in \torr_{M,\theta}\cap I} F_\ell(t) \ll (1+b)\log M.
\end{equation}
\end{lemma}


The key reason we obtain so much cancellation in the sum \eqref{bd:avgFell} is the fact that for any interval $I\subset\tor$ of length $1/\ell$,
\begin{equation}	\label{Fell:mean0}
\int_I F_\ell(s) ds =0,
\end{equation}
which follows from \eqref{EV} and change of variable.

\begin{proof}
We will abbreviate $\theta'=\theta/M^2$ and $T=\torr_{M,\theta}$. 
Let $t^-, t^+$ be the smallest and largest elements of $T\cap I$, respectively. 
For $t\in T$ write $I(t) = [t-1/2M, t+ 1/2M)$, and denote $\tilde{I} = I\setminus (I(t^-)\cup I(t^+))$. 
From \eqref{Fell:mean0} we have
\begin{equation}	\label{avgFell.1}
\sum_{t\in T\cap I} F_\ell(t) = F_\ell(t^+) + F_\ell(t^-) + \left( \sum_{t\in T\cap \tilde{I}} F_\ell(t) - M\int_{\tilde{I}} F_\ell(s)ds\right) -M\int_{I\setminus \tilde{I}} F_\ell(s)ds.
\end{equation}
For the first two terms above, note that
\begin{equation}	\label{avgFell.2}
F_\ell(t) = \log\frac1{\|\ell t\|_\tor} + O(1) .
\end{equation}
Suppose $\|\ell t\|_\tor=\delta$ for some $t\in T\cap I$ and  $\delta>0$. Write $t= \theta' +\frac{k}{N}$.  Then there are $d\in \Z$ and $\psi\in \R$ with $|\psi|\le \delta$ such that 
\[
\ell\left(\theta'+ \frac{k}M\right) = d+ \psi.
\]
Thus,
\[
|\ell\theta' M - cM + \ell k| =|\psi M|\le \delta M,
\]
which implies $\|\ell\theta' M\|_\tor\le \delta M$. Hence
\[
\|\ell t\|_\tor=\delta \ge \frac1{M} \|\ell\theta' M\|_\tor = \frac1M|\ell\theta' M| \ge |\theta'| \ge M^{-2-b}.
\]
Together with \eqref{avgFell.2} this implies
\begin{equation}	\label{avgFell.3}
  \max_{t\in T\cap I} F_\ell(t) \ll (1+b)\log M.
\end{equation}

We next consider the term in parentheses in \eqref{avgFell.1}.
We have
\begin{align*}
\frac1M \sum_{t\in T\cap \tilde{I}} F_\ell(t) - \int_{\tilde{I}} F_\ell(s)ds
&= \sum_{t\in T\cap \tilde{I}} \int_{I(t)} F_\ell(t) - F_\ell(s) ds.
\end{align*}
From Taylor expansion, for each $t\in T\cap \tilde{I}$ and $s\in I(t)$,
\[
F_\ell(s)-F_\ell(t) = (s-t)F_\ell'(t) + O\left( \frac1{M^2}\|F_\ell''1_{I(t)}\|_\infty\right)
\]
so
\[
\int_{I(t)}F_\ell(t) - F_\ell(s) ds \ll \frac1{M^3} \|F''_\ell1_{I(t)}\|_\infty. 
\]
We have
\[
F_\ell''(t) = \frac{-\pi^2 \ell ^2}{\sin^2(\pi \ell t)} \ll \frac{\ell ^2}{\|\ell t\|_\tor^2}.
\]
Now for $t\notin\{t^-,t^+\}$ we have $\|\ell t\|_\tor \gg \frac{\ell }{M} \min(k-k^-, k^+-k)$, where we write $t= \theta'+ \frac{k}M$ and $t^\pm = \theta' + \frac{k^\pm}{M}$. 
Combining the previous displays thus gives
\begin{equation}	\label{avgFell.4}
\frac1M\sum_{t\in T\cap \tilde{I}} F_\ell(t) - \int_{\tilde{I}} F_\ell(s)ds \ll \frac1M \sum_{k^-<k<k^+} \frac1{\min(k-k^-, k^+-k)^2} \ll \frac1M .
\end{equation}

Finally, for the last term on the right hand side of \eqref{avgFell.1}, bounding $|1-e(\ell t)|\ll \ell t$ we have
\[
\int_{I\setminus \tilde{I}} F_\ell(s) ds \ll \int_0^{1/M} |\log(\ell x)|dx + O(1/M) \ll \frac{\log M}{M}.
\]

The estimate \eqref{bd:avgFell}
follows by substituting the above bound along with \eqref{avgFell.3} and \eqref{avgFell.4} into \eqref{avgFell.1}.
\end{proof}

\begin{proof}[Proof of Proposition \ref{prop:early}]
Write $T=\torr_{N_0,\theta}$ and $M=e^{\varrho m}$. 
From Lemma \ref{lem:avgFell}, 
\begin{equation}
\frac{1}{|T|}\sum_{t\in T} \YY_{[1,M )}(t) \ll \frac{\log N_0 }{N_0 } \sum_{\ell \le M} \ell Z_\ell. 
\end{equation}
In particular,
\begin{equation}	\label{avgT:expected}
\e\frac{1}{|T|}\sum_{t\in T} \YY_{[1,M )}(t) \ll \frac{M}{N_0 }\log N_0 .
\end{equation}
From Markov's inequality,
\begin{equation}
\pro{ \left|\frac{1}{|T|}\sum_{t\in T} \YY_{[1,M )}(t)\right|\ge K\frac{M}{N_0 }\log N_0  }  = O(1/K)
\end{equation}
for any $K>0$. 
From Lemma \ref{lem:upper.crude} we have
\[
\sup_{t\in \tor} Y_{[1,M)}(t) \ll \log M
\]
except with probability $O(M^{-c})$. 
Applying Lemma \ref{lem:easyLB} we conclude that 
with probability
$
1-O(1/K) - O(M^{-c})
$,
\[
\left|\left\{ t\in T: \YY_{[1,M )}(t) \ge -2K\frac{M}{N_0 }\log N_0 \right\}\right|  \gg \frac{|T|}{1+ \frac{1}{K}\frac{N_0 }M\frac{\log M}{\log N_0 }}.
\]
Taking
\[
K = \frac{N_0 \log M}{M\log N_0 },
\]
by our assumption $N_0\ge e^{3\varrho m} = M^3$ we have $K\gg M$, and so with probability $1-O(M^{-\min(c,1)})$,
\[
|\{ t\in T: \YY_{[1,M )}(t) \ge -2\log M\}|  \gg |T|.
\]
\end{proof}

\subsection{Middle generations: Many descendants are large}		\label{sec:middle}

In this section we prove Proposition \ref{prop:rapid}.
We will apply the following simple consequence of the Cauchy--Schwarz inequality.

\begin{lemma}[Paley--Zygmund inequality]	\label{lem:PZ}
Let $Z$ be a non-negative random variable with nonzero mean. Then for any $\delta\in (0,1)$, 
\[
\pr(Z\ge \delta \e Z) \ge (1-\delta)^2\left(1+ \frac{\Var Z}{(\e Z)^2}\right)^{-1}.
\]
\end{lemma}

\begin{proof}
We have
\[
(1-\delta)\e Z \le \e Z\ind(Z\ge \delta \e Z) \le (\e Z^2)^{1/2} \pr(Z\ge \delta \e Z)^{1/2},
\]
where we applied Cauchy--Schwarz. The result follows from rearranging terms. 
\end{proof}

We also need the following elementary estimate on the density of the nets $\torr_{q,\theta}$ in Bohr sets (recall \eqref{def:bohr} and \eqref{def:torq2}). 

\begin{lemma}	\label{lem:bohr.mesh}
Let $q,\xi\ge 1$, $\theta\in \R$, and $\kappa\in (0,1/2)$. 
Then
\[
|\torr_{q,\theta }\cap B_\xi(\kappa)| = 2\kappa q + O(\xi).
\]
\end{lemma}

\begin{proof}
Since $\kappa<1/2$, $B_\xi(\kappa)$ is the disjoint union of $\xi$ intervals in $\tor$ of length $2\kappa/\xi$. The number of points in $\torr_{q,\theta}$ inside each interval is $q\cdot 2\kappa/\xi + O(1)$, and the claim follows after summing this estimate over the $\xi$ intervals. 
\end{proof}

\begin{proof}[Proof of Proposition \ref{prop:rapid}]
For ease of writing we drop the superscript $\cQ^*$ from $\e, \pr$ for the duration of the proof. 
We may assume $\epp<\eps_0$, where the constant $\eps_0$ was defined in \eqref{beta.bounds}.

We introduce parameters $\kappa\in (0,1)$, $\xi_0,\Delta>0$ with
\begin{equation}	\label{KXD}
\kappa = \expo{-\varrho n/(2\Eons )}, \qquad   n^{C_0} \le \xi_0,\Delta \le e^{o(\varrho n)}
\end{equation}
for a sufficiently large absolute constant $C_0>0$.
We begin by arguing we may assume
\begin{equation}	\label{TB.reduc}
T\subseteq T_{N_0,\theta} \setminus \Maj(\xi_0,\kappa).
\end{equation}
Indeed, from Lemma \ref{lem:bohr.mesh}, for any $|\xi|\le \xi_0$, 
\[
|T_{N_0, \theta} \cap B_\xi(\kappa)| = 2\kappa N_0 + O(\xi_0) 
\]
so
\[
|T_{N_0, \theta} \cap \Maj(\xi_0, \kappa)|
\le 2\kappa \xi_0N_0 + O(\xi_0^2) = \expo{ \varrho n\left( 1- \frac1{2\Eons } + o(1)\right)}.
\]
Since we are assuming $|T|= \exp( \varrho n(1+o(1)))$, we can replace $T$ with $T\setminus \Maj(\xi_0,\kappa)$ without affecting the hypothesis. We henceforth assume \eqref{TB.reduc} holds.

As a first step we establish the first moment estimate
\begin{equation}	\label{ER.est}
\e^{\cQ^*}|\cR_n(T,\xcrit-\epp)|  \ge \exp\big[(a(\epp)+o_{\epp,\Eons }(1)) \varrho n\big].
\end{equation}
Since the conditioning on $\cQ^*$ does not affect the independence of the families of Poisson variables $\{Z_\ell\}_{\ell\in J_\eon}$ across $\eon$ we have
\begin{equation}	\label{Rn:1mom}
\e |\cR_n(T,\xcrit-\epp)|= \sum_{t\in T} \prod_{2\le \eon \le \Eons } \pr\big(Y_{J_\eon}(t)\ge (\xcrit-\epp) q_\eon \big).
\end{equation}

Let $(U_p)_{p\in \urn}$ be a sequence of i.i.d.\ uniform elements of $\tor$ indexed by the set $\urn$ and let $\LL_p=\log|1-e(U_p)|$ for each $p\in \urn$. 
For $\urn'\subset \urn$ we write 
\[
\tY_{\urn'} = \sum_{p\in \urn'} \LL_p.
\] 
Consider an arbitrary $\yy\in [\xcrit-\eps_0, \xcrit+\eps_0]$ with $\eps_0$ as in \eqref{beta.bounds}, and abbreviate
\[
p_\eon(\yy) := \pr( \tY_{\urn^{\eon}}\ge \yy q_\eon),
\]
where we recall that $q_i=|Q^i|$, see the notation above \eqref{def:good.urn} .
From \eqref{BaRa} (or just Cram\'er's theorem) and our assumptions on $\varrho, \Eons $ and \eqref{assume:qell}, for each $2\le \eon\le \Eons $ we have
\begin{equation}	\label{ply}
p_\eon(\yy) =  \expo{-(1+o_{y,\Eons }(1))\lambda^*(\yy)\frac{\varrho n}\Eons  }.
\end{equation}
From Proposition \ref{prop:Yt.approx} and our assumptions on $\varrho,q_\eon,\kappa,\xi_0$, 
\begin{equation}	\label{YJs.tail}
\pr (Y_{J_\eon}(s)\ge \yy q_\eon) = (1+o_{\yy,\Eons }(1)) p_\eon(\yy) ,\qquad \forall s\in \tor\setminus \Maj(\xi_0,\kappa).
\end{equation}
Substituting the estimates
\eqref{YJs.tail} and \eqref{ply} into \eqref{Rn:1mom} we have
\begin{align}	
\e |\cR_n(T, \xcrit-\epp)| 
&= |T| \expo{o_{\epp,\Eons }(1)} \prod_{\eon=2}^\Eons  p_\eon(\xcrit-\epp)	\label{eR.1} \\
& = |T| \expo{ -(1+o_{\epp,\Eons }(1)) \lambda^*(\xcrit-\epp)\varrho n \left( 1- \frac1\Eons \right)}.	\notag
\end{align}
Now by definition, $\lambda^*(\xcrit) =1$, and since $\lambda^*$ is strictly increasing and continuous on $\R_+$,
\begin{equation}	\label{def:aepp}
a(\epp):= 1-\lambda^*(\xcrit-\epp) >0.
\end{equation}
By our assumption on $|T|$ we thus have
\[
\e |\cR_n(T, \xcrit-\epp)|\gg \expo{ (1+o_{\epp,\Eons }(1)) \left( a(\epp)+ \frac{1-a(\epp)}{\Eons }\right)\varrho n} \ge \expo{ (a(\epp)+o_{\epp,\Eons }(1)) \varrho n}
\]
which gives \eqref{ER.est}.

Next we prove the following second moment estimate, showing that $|\cR_n(T,\xcrit-\epp)|$ is concentrated around its mean:
\begin{equation}	\label{R.conc}
\frac{\Var^{\cQ^*} |\cR_n(T,\xcrit-\epp)|}{\big(\e^{\cQ^*}|\cR_n(T,\xcrit-\epp)| \big)^2} = o_{\epp, \Eons }(1).
\end{equation}
The claim \eqref{manyrapid} follows from by combining the above with \eqref{ER.est} and Lemma \ref{lem:PZ} (taking $\delta=e^{-a(\epp)\varrho n/2}$, say). Thus, it only remains to establish \eqref{R.conc}.

Consider again an arbitrary $y\in [\xcrit-\eps_0,\xcrit+\eps_0]$. 
As we argued for \eqref{ER.est}, by the independence of the families of Poisson variables $\{Z_\ell\}_{\ell\in J_\eon}$ across $\eon$, 
\begin{equation}	\label{Rn:2mom}
\e|\cR_n(T,\yy)|^2 
=\e|\cR_n(T,\yy)| + \sum_{s\ne t\in T} \prod_{2\le \eon\le \Eons } \pr \big( Y_{J_\eon}(s), Y_{J_\eon}(t) \ge \yy q_\eon \big).
\end{equation}

We stratify the sum over distinct points $s,t\in T$ in \eqref{Rn:2mom} according to the size of $d_{\xi_0}(s,t)$. This distance will determine the structure of correlations in the second moment computation below.
For $s\in \tor$ and $r\ge0$ define the ``$r$-neighborhoods" of $s$:
\begin{equation}
D_r(s) = 
\begin{cases}
\{s\} & r=0,\\
\big\{ t\in \tor: d_{\xi_0}(s,t) < \delta_r \big\} & r\ge 1
\end{cases},
\qquad \text{ where } \delta_r=\frac{\Delta}\varrho  \expo{-\varrho n\frac{\Eons -r}{\Eons }}.
\end{equation}
Note that for fixed $s\in \tor$ these sets are increasing in $r$; moreover, since $d_{\xi_0}(s,t)\le 1$ for all $s,t\in \tor$ we have $D_r(s) = \tor$ for $r\ge \Eons $. 
For \eqref{Rn:2mom} we now have
\begin{align*}
&\e |\cR_n(T,\yy)|^2\\
&= \e |\cR_n(T,\yy)| + \sum_{s\ne t\in T } \prod_{\eon=2}^\Eons  \pr (Y_{J_\eon}(s), Y_{J_\eon}(t) \ge \yy q_\eon)\\
&= \e  |\cR_n(T,\yy)| + \sum_{r=0}^{\Eons -1} \sum_{s\in T} 
 \sum_{t\in T\cap (D_{r+1}(s)\setminus D_r(s))} \prod_{\eon = 2}^\Eons \pr (Y_{J_\eon}(s), Y_{J_\eon}(t) \ge \yy q_\eon)\\
 &\le \e  |\cR_n(T,\yy)| + \sum_{r=0}^{\Eons -1} \sum_{s\in T} 
 \sum_{t\in T\cap (D_{r+1}(s)\setminus D_r(s))} \prod_{\eon=2}^{\Eons -r} \pr (Y_{J_\eon}(s) \ge \yy q_\eon) \prod_{\eon = \Eons -r+1}^\Eons \pr(Y_{J_\eon}(s), Y_{J_\eon}(t) \ge \yy q_\eon) ,
\end{align*}
where in the last line we have bounded $\pr  \big( Y_{J_\eon}(s), Y_{J_\eon}(t) \ge \yy q_\eon \big)$ by $\pr  \big( Y_{J_\eon}(s)\ge \yy q_\eon\big)$ for $2\le \eon\le \Eons -r$, the idea being that for these values of $\eon$ the events are not sufficiently decorrelated, so we do not lose much by assuming they are perfectly correlated.

From Proposition \ref{prop:decorr} and our assumptions on $\varrho,q_\eon,\kappa,\xi_0$ and $\Delta$,
\begin{equation}	\label{YJst.tail}
\pr (Y_{J_\eon}(s), Y_{J_\eon}(t) \ge \yy q_\eon) \le (1+o_{\yy,\Eons }(1)) p_\eon(\yy) ^2 \qquad  \forall s,t\in \tor \setminus \Maj(\xi_0, \kappa): t\notin D_r(s).
\end{equation}
Substituting the bounds \eqref{YJs.tail}, \eqref{YJst.tail} and reordering the sum, 
\begin{align*}
&\e|\cR_n(T,\yy)|^2 - \e |\cR_n(T,\yy)|\\
&\qquad\le  e^{o_{\yy,\Eons }(1)}\sum_{r=0}^{\Eons -1} \sum_{s\in T} 
 \sum_{t\in T\cap (D_{r+1}(s)\setminus D_r(s))} \prod_{\eon=2}^{\Eons -r} p_\eon(\yy) \prod_{\eon = \Eons -r+1}^\Eons p_\eon(\yy) ^2 \\
&\qquad=   e^{o_{\yy,\Eons }(1)} \left(\prod_{\eon=2}^\Eons p_\eon(\yy)  \right) 
\sum_{r=0}^{\Eons -1}  \sum_{s\in T} |T\cap (D_{r+1}(s)\setminus D_r(s))| \prod_{\eon = \Eons -r+1}^\Eons  p_\eon(\yy)  \\
&\qquad\le  e^{o_{\yy,\Eons }(1)} \bigg[    |T|(|T|-1) \prod_{\eon=2}^\Eons  p_\eon(\yy)^2 
+ 
\left(\prod_{\eon=2}^\Eons  p_\eon(\yy) \right)
\sum_{r=0}^{\Eons -2}  \sum_{s\in T}|T\cap (D_{r+1}(s)\setminus D_r(s))|\!\prod_{\eon = \Eons -r+1}^\Eons  p_\eon(\yy)   \bigg].
\end{align*}
From \eqref{eR.1} we thus have
\begin{equation}	\label{ER2}
\e|\cR_n(T,\yy)|^2 \le e^{o_{\yy,\Eons }(1)} \left[ (\e|\cR_n(T,\yy)|)^2 + R_0(\yy)\right]
\end{equation}
where
\begin{align*}
R_0(\yy)
&:=\left(\prod_{\eon=2}^\Eons  p_\eon(\yy) \right)
\sum_{r=0}^{\Eons -2}  \sum_{s\in T}|T\cap (D_{r+1}(s)\setminus D_r(s))|\prod_{\eon = \Eons -r+1}^\Eons  p_\eon(\yy)  .
\end{align*}

Now we estimate the cardinalities of the sets $ T\cap \left(D_{r+1}(s)\setminus D_r(s)\right)$.
For $r=\Eons -1$ we simply bound
\begin{equation}	\label{Tcap1}
 |T\cap \left(D_\Eons (s)\setminus D_{\Eons -1}(s)\right)|=|T\setminus D_{\Eons -1}(s)|\le |T| . 
\end{equation}
For $0\le r\le \Eons -2$, note that $D_r(s)$ can be expressed as a union of shifted Bohr sets (recall \eqref{def:bohr}):
\begin{align*}
D_r(s) &= \bigcup_{\xi,\xi'\in \{-\xi_0,\dots, \xi_0\}\setminus \{0\}} \{t\in \tor: \|\xi s+\xi't\|_\tor <\delta_r\}\\
&=\bigcup_{\xi,\xi'\in \{-\xi_0,\dots, \xi_0\}\setminus \{0\}} \left( -\frac{\xi}{\xi'}s + B_{\xi'}(\delta_r)\right).
\end{align*}
Applying Lemma \ref{lem:bohr.mesh},
we have that for each fixed $\xi,\xi'$, 
\begin{align*}
\left| \torr_{N_0,\theta} \cap \left(-\frac{\xi}{\xi'}s + B_{\xi'}(\delta_r)\right) \right| 
&= 2N_0\delta_r+ O(|\xi'|)
\end{align*}
(note that the shift $-\xi s/\xi'$ can be absorbed into the parameter $\theta$ in the lemma). 
Thus, by monotonicity and the union bound, for any $0\le r\le \Eons -2$, 
\begin{equation}	\label{Tcap2}
|T\cap \left(D_{r+1}(s)\setminus D_r(s)\right)| \le |\torr_{N_0,\theta}\cap D_{r+1}(s)|\le 8\xi_0^2N_0\delta_{r+1} + O(\xi_0^3).
\end{equation}
From the upper bound on $\Delta$ in \eqref{KXD},
\[
N_0\delta_{r+1} = \expo{ \varrho n\left( \frac{r+1}\Eons  + o(1)\right)} \ge \expo{ \varrho n(1+o(1))/\Eons }.
\]
From the upper bound on $\xi_0$ in \eqref{KXD} we conclude the the second term in the final bound in \eqref{Tcap2} is of lower order then the first, and hence
\begin{equation}	\label{Tcap3}
|T\cap \left(D_{r+1}(s)\setminus D_r(s)\right)| \ll \xi_0^2 N_0 \delta_{r+1} = \expo{ \varrho n\left( \frac{r+1}{\Eons } + o(1)\right)}.
\end{equation}

From \eqref{eR.1}, \eqref{Tcap3}, and \eqref{ply},
\begin{align*}
\frac{R_0(\yy)}{(\e|\cR_n(T,\yy)|)^2 } 
&\ll 
\frac{|T| \left(\prod_{\eon=2}^\Eons  p_\eon(\yy) \right)
\sum_{r=0}^{\Eons -2} \expo{ \varrho n\left( \frac{r+1}{\Eons } + o(1)\right)}\prod_{\eon = \Eons -r+1}^\Eons  p_\eon(\yy)  }{ |T|^2 \prod_{\eon=2}^\Eons  p_\eon(\yy)^2}
\\
&= \expo{- \varrho n(1+o(1))}
\sum_{r=0}^{\Eons -2} \expo{ \varrho n\left( \frac{r+1}{\Eons } \right)}\left(\prod_{\eon = 2}^{\Eons -r} p_\eon(\yy)\right)^{-1} \\
&= \sum_{r=0}^{\Eons -2} \expo{ -(1+o_{\yy,\Eons }(1)) (1-\lambda^*(y))(\Eons -r-1) \frac{\varrho n}{\Eons } }.
\end{align*}
Thus, 
\begin{align*}
\frac{R_0(\xcrit-\epp)}{(\e|\cR_n(T,\xcrit-\epp)|)^2 } 
& \ll_{\epp,\Eons } \sum_{r=0}^{\Eons -2} \expo{ -\frac{a(\epp)\varrho n}{2\Eons } ( \Eons -r-1)}
\ll_{\epp,\Eons } \expo{ -\frac{a(\epp)\varrho n}{2\Eons }},
\end{align*}
where in the second bound we summed the geometric series. 
Together with \eqref{ER2} this gives
\begin{align*}
\frac{\Var|\cR_n(T,\xcrit-\epp)|}{(\e |\cR_n(T,\xcrit-\epp)|)^2} 
&\ll_{\epp, \Eons } e^{o_{\epp,\Eons }(1)}-1 +   \expo{ -\frac{a(\epp)\varrho n}{2\Eons }} =o_{\epp,\Eons } (1)
\end{align*}
which yields \eqref{R.conc}, and hence the claim.
\end{proof}

\section{Lower bound: Late generations}	\label{sec:late}

\subsection{Structural dichotomy for high points}

Let $W$ be a slowly-growing function of $N$ to be chosen later, with
\begin{equation}	\label{Wrange}
\omega(1)\le W\le N^{o(1)},
\end{equation}
and recall from \eqref{XNW} the truncated field
\begin{equation}
X_{N}^{\le}(t) =\sum_{\ell\le N/W} C_\ell(P_N) \log |1-e(\ell t)|. 
\end{equation}
In this section we abbreviate
\begin{equation}
\super_\epp(T) :=  \big\{ t\in T: X_{N}^\le (t) \ge (\xcrit-\epp)\log N\big\}.
\end{equation}
Recall the notation $\torr_{N,\theta}$, see \eqref{def:torq2}.
From Theorem \ref{thm:arta} we deduce the following corollary of Proposition \ref{prop:poi.highpts}.

\begin{cor}	\label{cor:early.middle}
Let $N$ be a large integer and $N^{-20}\le \theta\le 1$. 
Let $T_0\subset \torr_{N,\theta}$ with $|T_0|\ge (1-o(1)) N$. For any fixed $\epp>0$, 
\begin{equation}	\label{early.middle.bd}
\left|\super_\epp(T_0)\right| \ge N^{c'(\epp)}\qquad \text{ with probability $1-o_\epp(1)$.}
\end{equation}
\end{cor}

Indeed, one simply applies Proposition \ref{prop:poi.highpts} with $N$ in place of $N_0$ and $N/W$ in place of $N$ (using \eqref{Wrange}), followed by Theorem \ref{thm:arta}. 

Now we want to rule out the event that the population of high points $\super_\epp(T_0)$ is wiped out by the high frequency tail 
\[
X_{N}^>(t) = \sum_{N/W<\ell\le N} C_\ell(P_N) \log|1-e(\ell t)|.
\] 
Let 
\[
H = \{\ell\in (N/W, N]: C_\ell(P_N)\ge 1\}
\]
be the random set of high frequencies contributing to the tail. 
Note that for any $\ell\in (N/W, N]$, 
\begin{equation}	\label{off.Bohr}
\log|1-e(\ell t)| > -\epp \log N + O(1)\qquad \forall t\notin B_\ell (N^{-\epp}). 
\end{equation}
(Recall the notation for Bohr sets defined in \eqref{def:bohr}.)
Thus, letting
\begin{equation}	\label{bad.high}
\cH(T_0,\epp) = \Big\{ \super_\epp(T_0) \subset \bigcup_{\ell\in H} B_\ell (N^{-\epp})\Big\},\end{equation}
we would like to show $\pr(\cH(T_0,\epp)) = o_\epp(1)$, 
perhaps under additional hypotheses on $T_0$.

From Lemma \ref{lem:bohr.mesh} we know that $B_\ell(N^{-\epp})$ 
contains at most around $2N^{1-\epp}$ elements of $T_0$. While this is a proportion $o(1)$, it is much larger than the number of high points $|S(T; \xcrit-\epp)|$, so we cannot rule out \eqref{bad.high} from a simple union bound.
Instead we will condition on the cycles $C_\ell(P_N)$ with $\ell \le N/W$ to fix $\super_\epp(T_0)$, and consider the (now deterministic) set $F$ of frequencies $\ell\in (N/W,N]$ for which the corresponding Bohr set $B_\ell (N^{-\epp})$ captures a large fraction of $\super_\epp(T_0)$. 
We will then use a dichotomy: either $F$ is a sparse subset of $(N/W,N]$ (the ``unstructured" case), in which case we argue it is unlikely to overlap with the sparse random set $H$, 
or $F$ is dense (the ``structured case"). If $F$ is dense -- that is, if $\super_\epp(T_0)$ has large overlap with a large proportion of the high frequency Bohr sets $B_\ell (N^{-\epp})$ -- 
it turns out one can use a double counting argument and a Vinogradov-type lemma (standard in applications of the circle method) to show that $\super_\epp(T_0)$ must contain a \emph{highly} structured element $t$ -- specifically, an element $t$ which is very close to a rational number with denominator of size $W^{O(1)}$. But we can easily arrange for $T_0$ to be disjoint from all such elements (which are contained in a union of Bohr sets of the form $\Maj(W^{O(1)},\kappa')$ for some small $\kappa'>0$). 

We turn to the details. 
Let 
\begin{equation}
\cG= \{ |H|\le C_0\log W\}
\end{equation}
for a suitable absolute constant $C_0>0$.
Since $H\le \sum_{N/W<\ell\le N} C_\ell(P_N)$, a straightforward second moment computation shows 
\begin{equation}	\label{hH:bd}
\pr(\cG^c) \ll 1/\log W
\end{equation}
if $C_0$ is taken sufficiently large. 
Let
\[
F(T_0,\epp) = \left\{\ell\in (N/W, N]: |\super_\epp(T_0)\cap B_\ell(N^{-\epp})| \ge \frac{ |\super_\epp(T_0)|}{C_0\log W}\right\}.
\]
Note that $F(T_0,\epp)$ is determined by the cycles $(C_\ell(P_N))_{\ell\le N/W}$. 
First we consider the unstructured case that
\begin{equation}	\label{FSa:unstruct}
|F(T_0,\epp)| \le N/W^3.
\end{equation}
Note that on the event $\cG\cap \cH(T_0,\epp)$, by the pigeonhole principle we have $|\super_\epp(T_0)\cap 
B_\ell(N^{-\epp})|\ge  |\super_\epp(T_0)|/(C_0\log W)$ for some $\ell\in H$,
and hence $H\cap F(T_0,\epp)\ne \emptyset$. 
The following lemma concludes the argument for the case that \eqref{FSa:unstruct} holds. We defer the proof the next subsection. 

\begin{lemma}	\label{lem:HcapF}
Assume $W\ge 10\log N$.
For any fixed $F\subset(N/W,N]$ of size $|F|\le N/W^3$, 
\[
\pr\big(H\cap F \ne \emptyset \mid (C_\ell(P_N))_{\ell\le N/W}\big) \ll 1/W.
\]
\end{lemma}

On the above lemma we have
\begin{align}
\pr(\cH(T_0,\epp)\cap \{ \text{ \eqref{FSa:unstruct} holds }\}) 
&\le \pr(\cG^c) + \pr( \cG\cap \cH(T_0,\epp) \cap \{\text{ \eqref{FSa:unstruct} holds }\}) \notag\\
&\le \pr(\cG^c) + \pr( H\cap F(T_0,\epp)\ne \emptyset) \notag\\
&= \pr(\cG^c) + \e\pr\big( H\cap F(T_0,\epp)\ne \emptyset \mid (C_\ell(P_N))_{\ell\le N/W}\big) 	\notag\\
&\ll \frac{1}{\log W} + \frac1{W} = o(1).	\label{HonF}
\end{align}
It only remains to show
\begin{equation}	\label{final.goal}
\pr(\cH(T_0,\epp)\cap \{  F(T_0,\epp)>N/W^3 \}) = o_\epp(1).
\end{equation}
On the event that $|F(T_0,\epp)|>N/W^3$, by Markov's inequality,
\begin{align*}
N/W^3
&\le |F(T_0,\epp)| \\
&= \sum_{N/W<\ell\le N} \ind\left(|\super_\epp(T_0)\cap B_\ell(N^{-\epp})|\ge \frac{|\super_\epp(T_0)|}{C_0\log W}\right)	\\
&\le C_0\log W\frac{1}{|\super_\epp(T_0)|} \sum_{N/W<\ell\le N} |\super_\epp(T_0)\cap B_\ell(N^{-\epp})|	\\
&= C_0\log W\frac{1}{|\super_\epp(T_0)|} \sum_{t\in \super_\epp(T_0)} \sum_{N/W<\ell\le N} \ind(\|\ell t\|_\tor\le N^{-\epp}).	
\end{align*}
From the pigeonhole principle it follows that for some $t\in \super_\epp(T_0)$ we have
\begin{equation}
\left|\big\{\ell\in (N/W,N]: \|\ell t\|_\tor \le N^{-\epp}\big\}\right| \ge \frac{ N}{C_0W^3\log W}.
\end{equation}
To summarize, we have shown that if a large number of high frequency Bohr sets $B_\ell(N^{-\epp})$ each captures a sizeable proportion of the population $\super_\epp(T_0)$, then there must exist an element $t\in \super_\epp(T_0)$ which is contained in a large proportion of the high frequency Bohr sets $\{B_\ell(N^{-\epp})\}_{N/W<\ell\le N}$. 
It turns out this implies $t$ must lie in a (very thin) \emph{low frequency} Bohr set, as the following lemma shows.

\begin{lemma}[Vinogradov lemma]	\label{lem:vino}
Let $I\subset \Z$ be interval of length at most $M$ and let $\theta\in \tor$ be such that for some $\kappa, \delta \in (0,1)$, $\|\ell\theta\|_\tor\le \kappa$ for at least $\delta M$ values of $\ell\in I$. Then either
\begin{equation}	\label{vino1}
M\le 2/\delta
\end{equation}
or
\begin{equation}	\label{vino2}
\kappa\ge \delta/100
\end{equation}
or else there exists a positive integer $\xi\le 2/\delta$ such that $\|\xi\theta\|_\tor\ll \kappa/(\delta M)$. 
\end{lemma}

A variant of this lemma is proved in \cite[Lemma A.4]{GrTa08_mobius} using the Erd\H{o}s--Tur\'an discrepancy inequality (see \cite{Montgomery_10lectures}). For a more direct proof see \cite{Tao_blog_254A8}. 

Applying the above lemma with $I=(N/W,N]$, $M=N$, $\kappa= N^{-\epp}$ and $\delta = 1/(C_0W^3\log W)$,
since $W=N^{o(1)}$ (here we only need $W\le N^{\epp/4}$, say) then neither of the alternatives \eqref{vino1}, \eqref{vino2} hold for all $N$ sufficiently large, and we conclude
$
t\in \Maj(\xi_0,\kappa')
$
for some 
\[
\xi_0\ll W^3\log W,\qquad \kappa'\ll \frac{\kappa W^3\log W}{N}= W^{O(1)} N^{-\epp-1}. 
\]
From Lemma \ref{lem:bohr.mesh}, for such $\xi_0,\kappa'$,
\begin{equation}	\label{bohr.small}
|\Maj(\xi_0,\kappa') \cap \torr_{N,\theta}| = W^{O(1)} =N^{o(1)}.
\end{equation}
Thus, taking 
\[
T_0:= \torr_{N,\theta}\setminus \Maj(\xi_0,\kappa') 
\]
with $\theta = N^{-20}$, say, we have $|T_0|\ge (1-o(1))N$, and
\[
\pr\big(\cH(T_0,\epp)\cap \{ F(T_0,\epp) > N/W^3\}\big) = 0
\]
for all $N$ sufficiently large. 
Together with \eqref{HonF} this gives
\[
\pr( \cH(T_0,\epp) ) = o_\epp(1)
\]
for this choice of $T_0$. Now by \eqref{bohr.small} we can apply Corollary \ref{cor:early.middle} with this choice of $T_0$ to conclude that
\eqref{early.middle.bd} holds. 
Finally, the intersection of $\cH(T_0,\epp)^c$ and the event in \eqref{early.middle.bd} has probability $1-o_\epp(1)$, and on this event, from \eqref{off.Bohr}, 
\[
 \big\{t\in \tor: X_N(t) \ge (\xcrit-2\epp)\log N\big\} \ne \emptyset.
\]
This concludes the proof of the lower bound in Theorem \ref{thm:main}.

\begin{remark}
Above we reduced the bad event $\cH(T_0,\epp)$ to the event that $X_{N}^\le(t)\ge (\xcrit-\epp)\log N$ for some $t\in \Maj(\xi_0, N^{-1-\epp})$, and then ruled out the latter by removing all low-frequency Bohr sets from $T_0$.
An alternative would have been to argue along the lines of Proposition \ref{prop:upper.major} that $X_{N}^{\le}(t)$ is unlikely to be large for such points $t$ in the first place. 
\end{remark}

\subsection{Proof of Lemma \ref{lem:HcapF}}

For this section it will be convenient to refer to the permutations themselves rather than their associated permutation matrices. 
Denote by $\fkS_N$ the set of all permutations on $[N]$. 
For $\sigma\in \mathfrak{S}_N$ we abuse notation and write $C_\ell(\sigma)$ for the number of $\ell$-cycles in $\sigma$. 
Throughout this subsection we denote by $\pi=\pi_N$ a uniform random element of $\fkS_N$. 
For $1\le M<N$ let
\[
\fkS_{N}^{>M}=\{\sigma \in \fkS_N: C_\ell(\sigma)=0\, \text{ for all } 1\le \ell\le M\}.
\] 

\begin{lemma}	\label{lem:Cell}
Let $1\le M<N$. For $\ell\le N$ we have
\begin{equation}	\label{bd:Cell}
\e \big(C_\ell(\pi) \mid \pi\in \fkS_N^{>M}\big) 
\le 
\begin{cases}
0 & \ell\le M \;\;\text{ or }\;\; N-M\le \ell \le N-1\\
\frac{N}{\ell(N-\ell)} & M+1\le \ell\le N-M-1.
\end{cases}
\end{equation}
Furthermore, if $M\le N/\log N$ then 
\begin{equation}	\label{bd:CN}
\e \big(C_N(\pi) \mid \pi\in \fkS_N^{>M}\big) = \pr \big(C_N(\pi)=1 \mid \pi\in \fkS_N^{>M}\big) = O(M/N).
\end{equation}
\end{lemma}

\begin{proof}
The case $\ell\le M$ in \eqref{bd:Cell} is immediate from the conditioning.
If there is a cycle of length at least $N-M$, but not $N$, then there must exist a cycle of length at most $M$, which is also ruled out by the conditioning. This establishes the bound in the first case.

For the second case in \eqref{bd:Cell} we use the switchings method. Fix $M+1\le \ell\le N-M-1$. We express
\[
\e \big( C_\ell(\pi) \mid  \pi\in \fkS_N^{>M} \big) 
= \sum_{S\subset [N], |S|=\ell}
\pr\big( S \text{ is a cycle of } \sigma \mid  \pi\in \fkS_N^{>M}\big).
\]
Fix $S\subset[N]$ of size $\ell$, and let 
\[
A_S(M) = \{\sigma \in \fkS_N^{>M}: S \text{ is a cycle of } \sigma\}.
\]
For $T\subset[N]$ and $\sigma \in \fkS_N$ we say that $T$ is an arc of $\sigma$ if the directed graph on vertex set $T$ associated to $\sigma$ is a path of $|T|-1$ edges with two ends. 
We let 
\[
B_S(M) = \{ \sigma\in \fkS_N^{>M} : S\text{ is an arc of } \sigma \}.
\]
Now we let $R\subset A_S(M)\times B_S(M)$ be the \emph{simple switching relation}. Specifically, $(\sigma,\tau) \in R$ when $\tau$ can be obtained from $\sigma$ by applying a simple switching, i.e. there exist distinct elements $i_1,i_2 \in [N]$ such that $\sigma(i_1) = \tau(i_2), \sigma(i_2) = \tau(i_1)$, and $\sigma(i)=\tau(i)$ for all $i\ne i_1,i_2$.
For $\sigma\in A_S(M)$, $\tau\in B_S(M)$ we write
\[
R(\sigma) = \{\rho\in B_S(M): (\sigma,\rho)\in R\}, \qquad R^{-1}(\tau) = \{\rho\in A_S(M): (\rho,\tau)\in R\}.
\]
For any $\sigma\in A_S(M)$, there are exactly $\ell(N-\ell)$ simple switchings that yield an element of $B_S(M)$ -- we apply the switching at a pair $\{i_1,i_2\}$ with $i_1\in S$ and $i_2\in S^c$. Thus, $|R(\sigma)| = \ell(N-\ell)$ for all $\sigma\in A_S(M)$ (we only need the lower bound on $|R(\sigma)|$).
On the other hand, for any $\tau \in B_S(M)$ there is exactly one switching that yields an element of $A_S(M)$, namely, at the pair $\{i_1,i_2\}$ with $S\setminus \tau(S) = \{i_1\}$, and $S\setminus \tau^{-1}(S) = \{i_2\}$. Thus, $|R^{-1}(\tau)|= 1$ for all $\tau \in B_S(M)$ (we only need the upper bound on $|R^{-1}(\tau)|$). 
Hence, 
\[
\ell(N-\ell) |A_S(M)| \le \sum_{\sigma \in A_S(M)} |R(\sigma)| = |R| = \sum_{\tau \in B_S(M)} |R^{-1}(\tau)| \le |B_S(M)|
\]
and rearranging we obtain $|A_S(M)| \le \frac1{\ell(N-\ell)} |B_S(M)|$. 
Normalizing these quantities we obtain
\[
\pr\big( \pi \in A_S(M) \mid \pi \in \fkS_N^{>M} \big) \le \frac1{\ell(N-\ell)} \pr\big( \pi \in B_S(M) \mid \fkS_N^{>M}\big). 
\]
Summing over $S\subset [N]$ of size $\ell$, we have
\[
\e \big( C_\ell(\pi ) \mid \fkS_N^{>M}\big) \le \frac1{\ell(N-\ell)} \e \Big( \big|\big\{ \text{ induced $\ell$-arcs in $\pi$ } \big\}\big| \,\big|\, \fkS_N^{>M}\Big).
\]
Since each cycle of length $k>\ell$ contains $k$ induced $\ell$-arcs, the number of induced $\ell$-arcs in $\pi$ is 
\[
\sum_{k>\ell} k C_k(\pi) \le \sum_{k\le N}  kC_k(\pi) = N.
\]
Substituting this deterministic bound in the previous line yields the second case of \eqref{bd:Cell}.

It remains to prove \eqref{bd:CN}. Let $\cC_N\subset \fkS_N$ denote the set of all $N$-cycles, and note that $\cC_N\subset \fkS_N^{>M}$. Thus
\[
\pr( C_N(\pi) \ge 1 \mid \pi \in \fkS_N^{>M} ) = \pr( \pi \in \cC_N \mid \pi \in \fkS_N^{>M}) = \frac{\pr(\pi \in \cC_N)}{ \pr( \pi \in \fkS_N^{>M})}.
\]
The numerator is easily seen to  equal  $1/N$. 
For the denominator, by Theorem \ref{thm:arta} we have
\begin{align*}
\pr( \pi \in \fkS_N^{>M}) 
&= \pr ( C_1(\pi)= \cdots = C_M(\pi) = 0) \\
&= \pr(Z_1=\cdots=Z_M=0) + O(e^{-(1+o(1))(N/M)\log(N/M)})\\
&= \exp\Big( - \sum_{\ell\le M} \frac1\ell\Big) + O(e^{-(1+o(1))(N/M)\log(N/M)})\\
& = (1-o(1)) e^{-\log M + O(1)} 
\gg 1/M,
\end{align*}
and the claim follows.
\end{proof}

For $1\le M\le N$ and $\sigma\in \fkS_N$ we write 
\begin{equation*}
\Len_M(\sigma) = \sum_{\ell \le M} \ell C_\ell(\sigma)
\end{equation*}
for the total length of all cycles of length at most $M$.
Then $\e \Len_M(\pi) = M$. 
We have the following corollary of the previous lemma.

\begin{cor}	\label{cor:EF}
Fix $10\log N\le W\le N$ and $F\subset(N/W,N]$. 
For $\sigma\in \fkS_N$ write $F(\sigma) = \{ \ell\in F: C_\ell(\sigma)\ge 1\}$.
We have
\[
\e\big( |F(\pi)| \mid (C_\ell(\pi))_{\ell\le N/W}\big)\ind\big( \Len_{N/W}(\pi) \le N/2\big) \le W^2|F|/N + O(1/W).
\]
\end{cor}

\begin{proof}
We have
\begin{equation}	\label{F.sum}
\e\big( |F(\pi)| \mid (C_\ell(\pi))_{\ell\le N/W}\big) \le \sum_{\ell\in F} \e\big( C_\ell(\pi) \mid (C_\ell(\pi))_{\ell\le N/W}\big).
\end{equation}
Conditioning on $(C_\ell(\pi))_{\ell\le N/W}$ fixes $\Len_{N/W}(\pi)$. We note that 
\[
\e\big( C_\ell(\pi) \mid (C_\ell(\pi))_{\ell\le N/W}\big) = 0 \qquad \text{for } \ell>N-\Len_{N/W}(\pi).
\]
Let us further condition on a realization of all of the cycles in $\pi$ of length at most $N/W$. Under this conditioning, the distribution of $\pi$ restricted to the subset $U$ of $[N]$ that is the complement of the union of fixed small cycles is the same as a uniform random permutation on $U$ conditioned to have no cycles of length at most $N/W$. 
After relabelling, we have that for each $N/W<\ell\le N-\Len$, 
\[
\ind\big(\Len_{N/W}(\pi)=\Len\big)\e\big( C_\ell(\pi) \mid (C_\ell(\pi))_{\ell\le N/W}\big) = \e\Big( C_\ell(\pi_{N-\Len}) \,\big|\, \pi_{N-\Len} \in \fkS_{N-\Len}^{>N/W} \Big)
\]
(recall that $\pi_{N-\Len}$ denotes a uniform random element of $\fkS_{N-\Len}$). 
For $\ell\le N-\Len - N/W$, by \eqref{bd:Cell} we have
\[
\e\Big( C_\ell(\pi_{N-\Len}) \,\big|\, \pi_{N-\Len} \in \fkS_{N-\Len}^{>N/W} \Big) 
\le \frac{N-\Len}{\ell(N-\Len - \ell)} \le \frac{N}{\ell(N/W)} = \frac{W}{\ell} \le \frac{W^2}{N},
\]
and for $N-\Len -N/W< \ell< N-\Len$ the left hand side is bounded by zero.
Finally, from \eqref{bd:CN}, our restriction to the event $\Len \le N/2$, and our assumption $W\ge 10\log N$, 
\[
\e\Big( C_\ell(\pi_{N-\Len}) \,\big|\, \pi_{N-\Len} \in \fkS_{N-\Len}^{>N/W} \Big)  = O(1/W).
\]
The claim follows by substituting the preceding bounds into \eqref{F.sum}.
\end{proof}

Now we prove Lemma \ref{lem:HcapF}.
We let $\pi=\pi_N$ be the permutation associated to $P_N$. 
Since $\e \Len_{N/W}(\pi) = \lf N/W\rf$, from Markov's inequality we have
\begin{equation*}	
\pr(\Len_{N/W}(\pi) \le N/2) \ge 1- O(1/W).
\end{equation*}
Thus,
\begin{align*}
\pr( H\cap F\ne\emptyset) 
&\le \pr( \Len_{N/W}(\pi)>N/2) + \e \pr\big( |H\cap F| \ge 1 \mid (C_\ell(\pi))_{\ell\le N/W}\big)\ind(\Len_{N/W}(\pi)\le N/2)\\
&\le  \frac{W^2|F|}{N} + O(1/W)
\ll 1/W,
\end{align*}
where in the second line we applied Corollary \ref{cor:EF} and our assumption on $W$, and in the final line we applied the assumption $|F|\le N/W^3$.

\bibliographystyle{alpha}
\bibliography{max-perm}

\newcommand{\etalchar}[1]{$^{#1}$}
\begin{thebibliography}{{Rem}17}

\bibitem[ABB{\etalchar{+}}]{ABBRS}
Louis-Pierre Arguin, David Belius, Paul Bourgade, Maksym Radziwi{\l}{\l}, and
  Kannan Soundararajan.
\newblock Maximum of the {Riemann} zeta function on a short interval of the
  critical line.
\newblock arXiv:1612.08575.

\bibitem[ABB17]{ABB}
Louis-Pierre Arguin, David Belius, and Paul Bourgade.
\newblock Maximum of the characteristic polynomial of random unitary matrices.
\newblock {\em Comm. Math. Phys.}, 349(2):703--751, 2017.

\bibitem[ABH17]{ABH_randzeta}
Louis-Pierre Arguin, David Belius, and Adam~J. Harper.
\newblock Maxima of a randomized {R}iemann zeta function, and branching random
  walks.
\newblock {\em Ann. Appl. Probab.}, 27(1):178--215, 2017.

\bibitem[Arg17]{Arguin_notes}
Louis-Pierre Arguin.
\newblock Extrema of log-correlated random variables principles and examples.
\newblock In {\em Advances in disordered systems, random processes and some
  applications}, pages 166--204. Cambridge Univ. Press, Cambridge, 2017.

\bibitem[AT92]{ArTa92}
Richard Arratia and Simon Tavare.
\newblock The cycle structure of random permutations.
\newblock {\em Ann. Probab.}, 20(3):1567--1591, 07 1992.

\bibitem[BAD15]{BeDa}
G{\'e}rard Ben~Arous and Kim Dang.
\newblock On fluctuations of eigenvalues of random permutation matrices.
\newblock {\em Ann. Inst. Henri Poincar{\'e} Probab. Stat.}, 51(2):620--647,
  2015.

\bibitem[Bah18a]{Bahier18_micro}
Valentin Bahier.
\newblock Characteristic polynomials of modified permutation matrices at
  microscopic scale.
\newblock arXiv: 1801.10461, 2018.

\bibitem[Bah18b]{bahier}
Valentin Bahier.
\newblock {On a limiting point process related to modified permutation
  matrices}.
\newblock arXiv:1803.03546, March 2018.

\bibitem[BHNY08]{BHNY}
Paul Bourgade, Christopher~P. Hughes, Ashkan Nikeghbali, and Marc Yor.
\newblock The characteristic polynomial of a random unitary matrix: a
  probabilistic approach.
\newblock {\em Duke Math. J.}, 145(1):45--69, 2008.

\bibitem[Big77]{Biggins77}
John~D. Biggins.
\newblock Martingale convergence in the branching random walk.
\newblock {\em J. Appl. Probability}, 14(1):25--37, 1977.

\bibitem[Bou10]{Bourgade_zeta-clt}
Paul Bourgade.
\newblock Mesoscopic fluctuations of the zeta zeros.
\newblock {\em Probab. Theory Related Fields}, 148(3-4):479--500, 2010.

\bibitem[Bra78]{Bramson78}
Maury~D. Bramson.
\newblock Maximal displacement of branching {B}rownian motion.
\newblock {\em Comm. Pure Appl. Math.}, 31(5):531--581, 1978.

\bibitem[BRR60]{BaRa60}
Raghu~R. Bahadur and Ramaswamy Ranga~Rao.
\newblock On deviations of the sample mean.
\newblock {\em Ann. Math. Statist.}, 31:1015--1027, 1960.

\bibitem[CMN16]{CMN}
Reda Chhaibi, Thomas Madaule, and Joseph Najnudel.
\newblock On the maximum of the c$\beta$e field.
\newblock Preprint. arXiv:1607.00243, 2016.

\bibitem[DZ98]{DeZe_book}
Amir Dembo and Ofer Zeitouni.
\newblock {\em Large deviations techniques and applications}, volume~38 of {\em
  Applications of Mathematics (New York)}.
\newblock Springer-Verlag, New York, second edition, 1998.

\bibitem[DZ14]{DaZe}
K.~Dang and D.~Zeindler.
\newblock The characteristic polynomial of a random permutation matrix at
  different points.
\newblock {\em Stochastic Process. Appl.}, 124(1):411--439, 2014.

\bibitem[Eva02]{Evans_wreath}
Steven~N. Evans.
\newblock Eigenvalues of random wreath products.
\newblock {\em Electron. J. Probab.}, 7:no. 9, 15, 2002.

\bibitem[Fel71]{Feller_vol2}
William Feller.
\newblock {\em An introduction to probability theory and its applications.
  {V}ol. {II}}.
\newblock Second edition. John Wiley \& Sons, Inc., New York-London-Sydney,
  1971.

\bibitem[FHK12]{FHK_freezing}
Yan~V. Fyodorov, Ghaith~A. Hiary, and Jonathan~P. Keating.
\newblock Freezing transition, characteristic polynomials of random matrices,
  and the {Riemann} zeta function.
\newblock {\em Phys. Rev. Lett.}, 108:170601, Apr 2012.

\bibitem[FK14]{FyKe_freezing}
Yan~V. Fyodorov and Jonathan~P. Keating.
\newblock Freezing transitions and extreme values: random matrix theory, and
  disordered landscapes.
\newblock {\em Philos. Trans. R. Soc. Lond. Ser. A Math. Phys. Eng. Sci.},
  372(2007):20120503, 32, 2014.

\bibitem[FMN17]{FMN}
Valentin {F{\'e}ray}, Pierre-Loic {M{\'e}liot}, and Ashkan {Nikeghbali}.
\newblock {Graphons, permutons and the Thoma simplex: three mod-Gaussian moduli
  spaces}.
\newblock {\em ArXiv e-prints}, December 2017.

\bibitem[GT08]{GrTa08_mobius}
Ben Green and Terence Tao.
\newblock Quadratic uniformity of the {M}{\"o}bius function.
\newblock {\em Ann. Inst. Fourier (Grenoble)}, 58(6):1863--1935, 2008.

\bibitem[Ham74]{Hammersley74}
John~M. Hammersley.
\newblock Postulates for subadditive processes.
\newblock {\em Ann. Probability}, 2:652--680, 1974.

\bibitem[HKOS00]{HKOS}
Ben~M. Hambly, Peter Keevash, Niel O'Connell, and Dudley Stark.
\newblock The characteristic polynomial of a random permutation matrix.
\newblock {\em Stochastic Process. Appl.}, 90(2):335--346, 2000.

\bibitem[HNNZ13]{HNNZ}
Christopher Hughes, Joseph Najnudel, Ashkan Nikeghbali, and Dirk Zeindler.
\newblock Random permutation matrices under the generalized {E}wens measure.
\newblock {\em Ann. Appl. Probab.}, 23(3):987--1024, 2013.

\bibitem[HNY08]{HNY}
Christopher~P. Hughes, Ashkan Nikeghbali, and Marc Yor.
\newblock An arithmetic model for the total disorder process.
\newblock {\em Probab. Theory Related Fields}, 141(1-2):47--59, 2008.

\bibitem[IK04]{IwKo_book}
Henryk Iwaniec and Emmanuel Kowalski.
\newblock {\em Analytic number theory}, volume~53 of {\em American Mathematical
  Society Colloquium Publications}.
\newblock American Mathematical Society, Providence, RI, 2004.

\bibitem[Kin75]{Kingman75}
J.~F.~C. Kingman.
\newblock The first birth problem for an age-dependent branching process.
\newblock {\em Ann. Probability}, 3(5):790--801, 1975.

\bibitem[Kin76]{Kingman76}
John F.~C. Kingman.
\newblock Subadditive processes.
\newblock In {\em {\'E}cole d'{\'E}t{\'e} de {P}robabilit{\'e}s de
  {S}aint-{F}lour, {V}--1975}, pages 167--223. Lecture Notes in Math., Vol.
  539. Springer, Berlin, 1976.

\bibitem[Kis15]{Kistler_notes}
Nicola Kistler.
\newblock Derrida's random energy models. {F}rom spin glasses to the extremes
  of correlated random fields.
\newblock In {\em Correlated random systems: five different methods}, volume
  2143 of {\em Lecture Notes in Math.}, pages 71--120. Springer, Cham, 2015.

\bibitem[KS99]{KaSa}
Nicholas~M. Katz and Peter Sarnak.
\newblock Zeroes of zeta functions and symmetry.
\newblock {\em Bull. Amer. Math. Soc. (N.S.)}, 36(1):1--26, 1999.

\bibitem[KS00]{KeSn}
Jonathan~P. Keating and Nina~C. Snaith.
\newblock Random matrix theory and {$\zeta(1/2+it)$}.
\newblock {\em Comm. Math. Phys.}, 214(1):57--89, 2000.

\bibitem[Mon73]{Montgomery_pair}
Hugh~L. Montgomery.
\newblock The pair correlation of zeros of the zeta function.
\newblock In {\em Analytic number theory ({P}roc. {S}ympos. {P}ure {M}ath.,
  {V}ol. {XXIV}, {S}t. {L}ouis {U}niv., {S}t. {L}ouis, {M}o., 1972)}, pages
  181--193. Amer. Math. Soc., Providence, R.I., 1973.

\bibitem[Mon94]{Montgomery_10lectures}
Hugh~L. Montgomery.
\newblock {\em Ten lectures on the interface between analytic number theory and
  harmonic analysis}, volume~84 of {\em CBMS Regional Conference Series in
  Mathematics}.
\newblock Published for the Conference Board of the Mathematical Sciences,
  Washington, DC; by the American Mathematical Society, Providence, RI, 1994.

\bibitem[Naj16]{Najnudel_zetamax}
Joseph Najnudel.
\newblock On the extreme values of the {Riemann} zeta function on random
  intervals of the critical line.
\newblock arXiv:1611.05562, 2016.

\bibitem[NN13]{NaNi13}
Joseph Najnudel and Ashkan Nikeghbali.
\newblock The distribution of eigenvalues of randomized permutation matrices.
\newblock {\em Ann. Inst. Fourier (Grenoble)}, 63(3):773--838, 2013.

\bibitem[PZ17]{PaZe}
Elliot Paquette and Ofer Zeitouni.
\newblock The maximum of the {CUE} field.
\newblock {\em Int. Math. Res. Not.}, rnx033, 2017.

\bibitem[{Rem}17]{Remy}
Guillaume {Remy}.
\newblock {The Fyodorov-Bouchaud formula and Liouville conformal field theory}.
\newblock arXiv:1710.06897, October 2017.

\bibitem[Sel46]{Selberg_clt}
Atle Selberg.
\newblock Contributions to the theory of the {R}iemann zeta-function.
\newblock {\em Arch. Math. Naturvid.}, 48(5):89--155, 1946.

\bibitem[Tao15]{Tao_blog_254A8}
Terence Tao.
\newblock 254a, notes 8: The {Hardy-Littlewood} circle method and
  {Vinogradov's} theorem.
\newblock Url:
  https://terrytao.wordpress.com/2015/03/30/254a-notes-8-the-hardy-littlewood-circle-method-and-vinogradovs-theorem/,
  March 2015.

\bibitem[TV15]{TaVu_poly}
Terence Tao and Van Vu.
\newblock Local universality of zeroes of random polynomials.
\newblock {\em Int. Math. Res. Not. IMRN}, 2015(13):5053--5139, 2015.

\bibitem[Vau97]{Vaughan}
Robert~C. Vaughan.
\newblock {\em The {H}ardy-{L}ittlewood method}, volume 125 of {\em Cambridge
  Tracts in Mathematics}.
\newblock Cambridge University Press, Cambridge, second edition, 1997.

\bibitem[Wie00]{Wieand_permutation}
Kelly Wieand.
\newblock Eigenvalue distributions of random permutation matrices.
\newblock {\em Ann. Probab.}, 28(4):1563--1587, 2000.

\bibitem[Wie03]{Wieand_wreath}
Kelly Wieand.
\newblock Permutation matrices, wreath products, and the distribution of
  eigenvalues.
\newblock {\em J. Theoret. Probab.}, 16(3):599--623, 2003.

\bibitem[Zei13]{Zeindler13}
Dirk Zeindler.
\newblock Central limit theorem for multiplicative class functions on the
  symmetric group.
\newblock {\em J. Theoret. Probab.}, 26(4):968--996, 2013.

\bibitem[Zei16]{Zeitouni_brw-notes}
Ofer Zeitouni.
\newblock Branching random walks and gaussian fields.
\newblock In {\em Proceedings of Symposia in Pure Mathematics}, volume~91,
  pages 437--471. Amer. Math. Soc., Providence, R.I., 2016.

\end{thebibliography}

\end{document}